\documentclass[twocolumn,amsthm]{autart}

\makeatletter
\def\@overkeywordskip{0pt}
\makeatother

\usepackage{xcolor}
\usepackage{amsmath,amssymb,amsfonts}
\usepackage{graphicx}
\usepackage{xr}
\usepackage{float}
\usepackage{needspace}
\usepackage{mathtools}
\usepackage{amsthm}
\usepackage{amsfonts}
\usepackage{amsmath}
\usepackage{setspace}
\usepackage{amssymb}
\usepackage{bm}
\usepackage{wrapfig}
\usepackage{algorithm}
\usepackage{algpseudocode}
\usepackage{algorithmicx}
\usepackage{dashrule}
\usepackage{lmodern}
\usepackage{multirow}
\usepackage{enumitem}
\usepackage{fix-cm}
\usepackage{tikz}
\usetikzlibrary{arrows.meta,decorations.pathreplacing}
\usepackage{tabularx}
\usepackage{booktabs}

\usepackage[numbers,sort&compress]{natbib}
\usepackage[hyphens]{url}
\usepackage{etoolbox}

\usepackage{mathtools}
\usepackage{xurl}
\usepackage[colorlinks=true,linkcolor=red,citecolor=blue,urlcolor=blue]{hyperref}
\AddToHook{env/frontmatter/after}{\endNoHyper}

\newtheorem{theorem}{Theorem}
\newtheorem{remark}{Remark}

\newtheorem{corollary}{Corollary}
\newtheorem{lemma}{Lemma}

\newtheorem{assumption}{Assumption}

\newcommand{\diag}{\text{diag}}
\newcommand{\tablesize}{\fontsize{8}{9.6}\selectfont}

\definecolor{lighter-green}{rgb}{0, 0.6, 0.00392156862} 

\newlength{\leftstackrelawd}
\newlength{\leftstackrelbwd}
\def\leftstackrel#1#2{\settowidth{\leftstackrelawd}%
{${{}^{#1}}$}\settowidth{\leftstackrelbwd}{$#2$}%
\addtolength{\leftstackrelawd}{-\leftstackrelbwd}%
\leavevmode\ifthenelse{\lengthtest{\leftstackrelawd>0pt}}%
{\kern-.5\leftstackrelawd}{}\mathrel{\mathop{#2}\limits^{#1}}}

\graphicspath{{../}}
\allowdisplaybreaks
\AtBeginDocument{
\setlength{\abovedisplayskip}{6pt plus 2pt minus 2pt}
\setlength{\belowdisplayskip}{6pt plus 2pt minus 2pt}
\setlength{\abovedisplayshortskip}{0pt plus 2pt}
\setlength{\belowdisplayshortskip}{3pt plus 2pt minus 2pt}
}

\begin{document}
\begin{frontmatter}
\title{An Adaptive, Parallel, and Inexact Newton Method for Large-scale Nonlinear Optimal Control}
\author[ucsd]{Luke Bhan\corauthref{corresponding}}\ead{lbhan@ucsd.edu},\quad
\author[berkeley]{Michael W. Mahoney}\ead{mmahoney@stat.berkeley.edu},\quad
\author[gatech]{Sen Na}\ead{senna@gatech.edu}
\corauth[corresponding]{Corresponding author.}
\begingroup
\setlength{\parskip}{0pt}
\address[ucsd]{University of California, San Diego, USA}
\address[berkeley]{University of California, Berkeley, USA}
\address[gatech]{Georgia Institute of Technology, USA}
\endgroup

\begin{abstract}
We design \emph{adaptive overlapping temporal decomposition} (AOTD), a parallel sequential quadratic programming (SQP) method for long-horizon nonlinear optimal control problems (OCPs). As designed, AOTD is the first overlapping temporal-decomposition (OTD) algorithm that adapts the overlap size online, eliminating the need to fix a suitable overlap a priori. Specifically, at each iteration, it partitions the long horizon into overlapping nonlinear subproblems, performs one SQP step on each in parallel, and adaptively selects both the overlap size and the accuracy to which the resulting local KKT systems are solved. Two conditions govern the accepted step: an adaptive \emph{residual condition} bounding the distance from the concatenated inexact direction to the exact Newton direction of the full problem, and a  \emph{descent condition} on the adaptive penalty parameters that certifies descent for a carefully chosen exact augmented Lagrangian. The residual condition is solver independent and accommodates both deterministic iterative solvers and randomized sketching solvers. Under standard regularity assumptions, we prove that the KKT residual converges to zero from any initialization and establish a local linear rate for sequences converging to a strict local minimizer. 
We validate AOTD on a power-system frequency regulation OCP and a Burgers PDE OCP, attaining a reduction in estimated FLOPs of approximately 2.5$\times$ to the best fixed-overlap variant, while converging across multiple horizon lengths with a single parameter setting.
\end{abstract}
\end{frontmatter}
\setlength{\parskip}{.5em} 

\section{Introduction} \label{sec:introduction}
In this work, we study long-horizon nonlinear dynamic programs (NLDPs) with equality constraints:
\begin{subequations}\label{eq:main-problem}
\begin{align}
\min_{\{\bm{x}_k\}, \{\bm{u}_k\}} \quad &\sum_{k=0}^{N-1}  g_k(\bm{x}_k, \bm{u}_k) + g_N(\bm{x}_N)\,, \\
\text{s.t.} \quad &\bm{x}_{k+1} = f_k(\bm{x}_k, \bm{u}_k)\,, \\
\quad   &\bm{x}_0 = \bar{\bm{x}}_0\,.
\end{align}
\end{subequations}
Here, $\bm{x}_k \in \mathbb{R}^{n_x}$ and $\bm{u}_k \in \mathbb{R}^{n_u}$ denote the state and control input with initial condition $\bm{x}_0=\bar{\bm{x}}_0$. The functions $g_k:\mathbb{R}^{n_x}\times\mathbb{R}^{n_u}\to\mathbb{R}$ and $g_N:\mathbb{R}^{n_x}\to\mathbb{R}$ denote the stage and terminal costs, while $f_k:\mathbb{R}^{n_x}\times\mathbb{R}^{n_u}\to\mathbb{R}^{n_x}$ defines the system dynamics. The horizon length is denoted by $N$. We allow the functions $f_k$ and $g_k$ to be nonconvex, so \eqref{eq:main-problem} is a nonlinear dynamic program. Such programs provide a flexible framework for modeling and solving sequential decision-making problems across a wide range of applications, including power systems~\cite{DIAS2013212,10115409}, robotics~\cite{6386025}, finance~\cite{TOPALOGLOU20081501}, and medicine~\cite{LIBOTTE2020105664}.

In this study, we focus on a subset of NLDPs known as \emph{long-horizon} optimal control problems (OCPs) in which $N$, the number of stages, is extremely large. Such problems commonly arise from fine temporal discretizations of dynamical systems~\cite{9658150,6953879}. For these problems, centralized nonlinear solvers, such as IPOPT~\cite{10.5555/3114195.3114663}, can require substantial computational resources and may fail because of computational and memory constraints~\cite{10.5555/3114195.3114663,ZAVALA201643}. The standard technique to resolve this is parallelization. In particular, most approaches decompose the long temporal horizon into a set of tractable subproblems that can be solved in parallel across a cluster of nodes.

These subproblem solutions are then concatenated via a coordination mechanism. For example, the alternating direction method of multipliers (ADMM) \cite{8186925} and Lagrangian dual decomposition \cite{1428793} are common first-order coordination methods. Other approaches, such as multiple shooting~\cite{BOCK19841603}, accommodate both first- and second-order nonlinear program (NLP) solvers. These methods provide a flexible framework for distributed computation, but their performance depends on how effectively information is exchanged across temporal boundaries. In particular, ADMM can converge slowly to high accuracy, and its behavior on nonconvex problems can depend on initialization and penalty selection~\cite{8186925}. This motivates decomposition methods that exploit the temporal structure of the optimality conditions while retaining the computational advantages of parallel local solves.

To resolve these pitfalls, overlapping temporal decomposition has emerged as a computationally efficient approach for long-horizon nonlinear OCPs \cite{doi:10.1137/16M1081993,zhao2025overlappingschwarzschemelinearquadratic,ni2024distributedsequentialquadraticprogramming}. The core idea is to extend each subproblem's interval at its boundaries to form overlapping subdomains, trading slightly larger subproblems for faster convergence by reducing the influence of boundary errors. Empirically, this trade-off has been worthwhile for certain long-horizon problems (e.g~\cite{osti_1123223}).

Recent work has also formalized this theoretically. In particular, \cite{doi:10.1137/19M1265065,9840913} showed the exponential decay of sensitivity (EDS) of OCPs. The core result demonstrates that, under suitable uniform regularity conditions, the influence of perturbations at a subproblem boundary decreases exponentially with the distance from that boundary. Consequently, increasing the overlap reduces the effect of boundary errors on the retained interior trajectory. For overlapping Schwarz methods, this property yields a local linear convergence rate that improves exponentially with the overlap size. Related results have also been developed for continuous-time linear-quadratic OCPs~\cite{zhao2025overlappingschwarzschemelinearquadratic}. These results explain why overlap can improve convergence, but still face a core computational trade-off: larger overlaps improve coordination while increasing the dimensions and cost of the local subproblems.

To resolve the expensive subproblem solves, \cite{naFOTD} introduced the fast overlapping temporal decomposition (FOTD) which uses sequential quadratic programming (SQP) for each subproblem. The idea is to replace the nonlinear subproblem with a single Newton step, and then coordinate after each small step using an exact augmented Lagrangian. This replaces solving each nonlinear subproblem to convergence with solving a linear system at each outer iteration. 
Additionally, it was shown that under this design, uniform stagewise local linear convergence is maintained.

This study builds on this line of research by addressing two core limitations of the previous design. First, in all previous work, the overlap size for the temporal decomposition must be specified a priori by the user. This results in either (1) a conservative overlap size that forces the algorithm to do additional work or (2) an overlap size that is too small, leading to the failure of the resulting optimization scheme (see ~\cite[Figure~2]{zhao2025overlappingschwarzschemelinearquadratic}). Second, solving every local KKT system to high accuracy can be wasteful when the resulting direction is still limited by the temporal decomposition. For example, for traditional centralized solvers, there has been significant development in designing inexact SQP methods that show that accurate outer iterations need not require exact linear solves at every step~\cite{doi:10.1137/0917003,doi:10.1137/060674004,doi:10.1137/130918320,doi:10.1137/20M1354556,pmlr-v202-hong23b}.

Given these limitations, we introduce \emph{Adaptive Overlapping Temporal Decomposition (AOTD)}, which adapts each subproblem's overlap size and solve accuracy through an adaptive residual condition. The core algorithm is as follows. At each iteration, AOTD partitions the horizon into nonlinear subproblems with individually specified overlap sizes. For each problem, we then compute one SQP direction by solving its local KKT system \emph{inexactly} to satisfy an adaptive residual condition in parallel. The local directions are then concatenated over the original, non-overlapping intervals to form a full-horizon primal--dual direction.

The concatenated direction is then checked against two conditions that control the errors from both the temporal decomposition and inexact local KKT solves. The first condition is an adaptive residual that controls the discrepancy between the concatenated direction and the exact Newton direction of the full problem. AOTD adjusts the overlap size and local solve accuracy to satisfy this condition relying on the fact that the decomposition and local accuracy errors can be explicitly decoupled as shown in Lemma~\ref{lem:global-residual-error}. Second, we enforce a descent condition so that the resulting direction is a descent direction for an exact augmented Lagrangian. Once both conditions are satisfied, AOTD uses the accepted direction to update the global primal--dual iterate and repeats.

We establish convergence guarantees for AOTD and demonstrate its computational efficiency in numerical experiments. Theoretically, we prove the resulting design is well-posed in the sense that, in finite iterations, all adaptive overlaps and accuracy conditions stabilize (Corollary \ref{corr:achievability} and Lemma \ref{lem:stability}) and  KKT residual tends to zero (Theorem \ref{thm:global}). Moreover, we establish local linear convergence (Theorem \ref{thm:local-linear}) for sequences that converge toward strict local minimizers. Empirically, we show that, using a fixed parameter setting on both a frequency regulation problem and a Burgers control problem, AOTD converges across the tested initial overlap sizes and reduces estimated FLOPs by approximately a factor of $2.5$ relative to the best~fixed-overlap variants discussed above.

\emph{Contributions.} Thus, the contributions of this work can be briefly summarized as follows:
\begin{enumerate}
\item \emph{Adaptive algorithm design.} We develop AOTD, a parallel SQP method that jointly adapts the temporal overlap and accuracy of local KKT solves. Specifically, we design a residual condition and a descent condition that links these choices to the quality of the full-horizon directions enabling local solves with independent, adaptive overlap sizes to terminate inexactly.

\item \emph{Well-posedness and finite stability of the parameters.} We prove that the adaptive design is well-posed ensuring that all conditions are attained in finite iterations (Corollary \ref{corr:achievability}) and that all adaptive parameters in both the overlap and residual conditions eventually stabilize (Lemma~\ref{lem:stability}).

\item \emph{Global and local linear convergence guarantees.} We establish that the KKT residual tends to zero without requiring initialization near a solution (Theorem \ref{thm:global}). For sequences converging to a strict local minimizer, we then establish a local linear convergence rate (Theorem \ref{thm:local-linear}) ensuring fast convergence.

\item \emph{Computational efficiency.} We evaluate AOTD on a power-system frequency regulation OCP and a Burgers PDE control benchmark. The experiments demonstrate convergence across multiple horizon lengths using a single parameter setting and a reduction in estimated floating-point operations by a factor of approximately $2.5$ relative to the best fixed-overlap variant tested.
\end{enumerate}

\emph{Paper organization.} The remainder of the paper is organized as follows. Section \ref{sec:preliminaries} reviews the SQP framework and overlapping temporal decomposition (OTD) for completeness. Section \ref{sec:main-algorithm-design} presents the proposed AOTD algorithm and design intuition. Sections \ref{sec:global-convergence} and \ref{sec:local-convergence} establish global convergence and local linear convergence, respectively. Section \ref{sec:simulations} demonstrates the effectiveness of our approach on two benchmark nonlinear OCPs, and Section \ref{sec:conclusions} concludes the paper.

\emph{Notation.} All vectors are column vectors, and $(\bm{x};\bm{y})$ denotes vertical concatenation. The norm $\|\cdot\|$ is Euclidean for vectors and spectral for matrices. For a symmetric matrix $A$, $\lambda_{\min}(A)$ and $\lambda_{\max}(A)$ denote its smallest and largest eigenvalues, respectively. For a general matrix $A$, $\sigma_{\min}(A)$ and $\sigma_{\max}(A)$ denote its smallest and largest singular values, respectively. The symbols $\mathbb{I}$ and $\bm{0}$ denote identity and zero objects of appropriate dimension. For integers $a\leq b$, $[a,b]\coloneqq\{a,\ldots,b\}$, $[a,b)\coloneqq\{a,\ldots,b-1\}$, and $[n]\coloneqq[0,n]$. For any sequence of state iterates, we use $\bm{x}\coloneqq\bm{x}_{0:N}=(\bm{x}_0;\ldots;\bm{x}_N)\in\mathbb{R}^{(N+1)n_x}$ as shorthand. The same applies for control iterates $\bm{u}$ or dual iterates $\bm{\lambda}$. We define the primal iterates as $\bm{z}_k=(\bm{x}_k;\bm{u}_k)$ for $k<N$ and $\bm{z}_N=\bm{x}_N$ which gives $\bm{z}\coloneqq(\bm{z}_0;\ldots;\bm{z}_N)\in\mathbb{R}^{n_z}$, where $n_z=(N+1)n_x+Nn_u$. Throughout, we use subscripts to index horizon stages and superscripts to index optimization iterations. A tilde marks quantities produced by the decomposition (local or composed); Newton quantities without a tilde refer to the corresponding exact full-horizon quantities.

\section{Preliminaries} \label{sec:preliminaries}
\subsection{Sequential Quadratic Programming}
We briefly review the sequential quadratic linearization of \eqref{eq:main-problem}. To begin, \eqref{eq:main-problem} can be given compactly as
\begin{align}\label{eq:main-problem-compact-form}
\min_{\bm z}\ g(\bm z)\qquad\text{s.t.}\quad f(\bm z)=0,
\end{align}
where $f(\bm z)\in\mathbb R^{n_x(N+1)}$ and
\newpage 
\begin{align}
g(\bm z)&=\sum_{k=0}^{N-1}g_k(\bm x_k,\bm u_k)+g_N(\bm x_N), \notag\\
f(\bm z)&=\begin{pmatrix}
\bm x_0-\bar{\bm x}_0\\
\bm x_1-f_0(\bm x_0,\bm u_0)\\
\vdots\\
\bm x_N-f_{N-1}(\bm x_{N-1},\bm u_{N-1})
\end{pmatrix}.
\notag 
\end{align}
The standard Lagrangian is $\mathcal L(\bm z,\bm\lambda)=g(\bm z)+\bm\lambda^Tf(\bm z)$, with equality multipliers $\bm\lambda\in\mathbb R^{n_x(N+1)}$. At iteration $\tau$, SQP updates the primal--dual iterate by
\begin{align*}
\begin{pmatrix}\bm z^{\tau+1}\\\bm\lambda^{\tau+1}\end{pmatrix}
=\begin{pmatrix}\bm z^\tau\\\bm\lambda^\tau\end{pmatrix}
+\alpha^\tau\begin{pmatrix}\Delta\bm z^\tau\\\Delta\bm\lambda^\tau\end{pmatrix}\,, 
\end{align*}
where the Newton direction solves
\begin{align}\label{eq:newtonSystem}
\begin{pmatrix}\hat H^\tau&(G^\tau)^T\\G^\tau&\bm0\end{pmatrix}
\begin{pmatrix}\Delta\bm z^\tau\\\Delta\bm\lambda^\tau\end{pmatrix}
=-\begin{pmatrix}\nabla_{\bm z}\mathcal L^\tau\\\nabla_{\bm\lambda}\mathcal L^\tau\end{pmatrix}\,,
\end{align}
and $\alpha^\tau$ is the step-size to be specified in Section \ref{sec:main-algorithm-design}.

Here, $H^\tau=\nabla_{\bm z}^2\mathcal L(\bm z^\tau,\bm\lambda^\tau)$, $G^\tau=\nabla_{\bm z}f(\bm z^\tau)$, and $\hat H^\tau$ is the Hessian modification that preserves its block-diagonal structure while ensuring positive definiteness in the tangent space of constraints (see \eqref{nsequ:1}) \cite{nocedalAndWright}. The superscript $\tau$ on Lagrangian derivatives denotes evaluation at $(\bm z^\tau,\bm\lambda^\tau)$. Notice, by the OCP structure, we~have~that~the Hessian diagonalizes as 
\begin{align*}
H^\tau=\diag(H_0^\tau,\ldots,H_N^\tau),\qquad
\hat H^\tau=\diag(\hat H_0^\tau,\ldots,\hat H_N^\tau).
\end{align*}
For $k\in[N-1]$, the blocks are
\begin{align*}
H_k^\tau=\begin{pmatrix}Q_k^\tau&(S_k^\tau)^T\\S_k^\tau&R_k^\tau\end{pmatrix}
=\begin{pmatrix}
\nabla_{\bm x_k}^2\mathcal L^\tau&\nabla_{\bm x_k\bm u_k}\mathcal L^\tau\\
\nabla_{\bm u_k\bm x_k}\mathcal L^\tau&\nabla_{\bm u_k}^2\mathcal L^\tau
\end{pmatrix},
\end{align*}
with stage dependence $(\bm z_k^\tau,\bm\lambda_{k+1}^\tau)$ and terminal block $H_N^\tau=\nabla_{\bm x_N}^2\mathcal L(\bm z_N^\tau)$. Each $\hat H_k^\tau$ has the corresponding modified blocks $\hat Q_k^\tau,\hat S_k^\tau,\hat R_k^\tau$ preserving this structure. Defining $A_k:=\nabla_{\bm x_k}f_k(\bm x_k,\bm u_k)$ and $B_k:=\nabla_{\bm u_k}f_k(\bm x_k,\bm u_k)$, the Jacobian $G^\tau=G(\bm z^\tau)$ has an initial identity block and dynamics block rows $(-A_k,-B_k,\mathbb I)$ in the columns of $(\bm x_k,\bm u_k,\bm x_{k+1})$, with zeros elsewhere and all derivatives evaluated at $\bm z^\tau$.

\subsection{Overlapping Temporal Decomposition (OTD)}

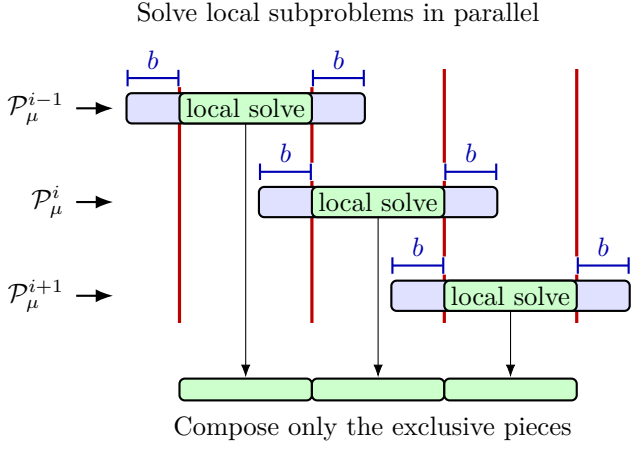
\begin{figure}[tp]
\centering
\begin{tikzpicture}[
x=0.70cm,y=0.8cm,
>=Latex,
font=\normalsize,
horizon/.style={black, thick},
knot/.style={red!75!black, very thick},
extbox/.style={draw=black, fill=blue!12, thick, rounded corners=2pt},
keepbox/.style={draw=black, fill=green!20, thick, rounded corners=2pt},
dim/.style={blue!70!black, thick, |-|},
lab/.style={font=\normalsize},
biglab/.style={font=\normalsize}
]

\draw[knot] (2.5,1.0) -- (2.5,4.9);
\draw[knot] (5.0,1.0) -- (5.0,3.35)
(5.0,3.65) -- (5.0,4.9);
\draw[knot] (7.5,1.0) -- (7.5,1.8)
(7.5,2.1) -- (7.5,3.35)
(7.5,3.65) -- (7.5,5.2);
\draw[knot] (10.0,1.0) -- (10.0,1.8)
(10.0,2.1) -- (10.0,5.2);

\node[anchor=east] at (0.5,4.55) {$\mathcal{P}_\mu^{i-1}$};
\node[anchor=east] at (0.50,3.00) {$\mathcal{P}_\mu^{i}$};
\node[anchor=east] at (0.50,1.45) {$\mathcal{P}_\mu^{i+1}$};

\draw[->, thick] (0.55,4.55) -- (1.25,4.55);
\draw[->, thick] (0.55,3.00) -- (1.25,3.00);
\draw[->, thick] (0.55,1.45) -- (1.25,1.45);

\draw[extbox]  (1.5,4.3) rectangle (6.0,4.8);
\draw[keepbox] (2.5,4.3) rectangle (5.0,4.8);
\node[lab] at (3.75,4.55) {local solve};

\draw[dim] (1.5,5.05) -- (2.5,5.05);
\node[blue!70!black, above] at (2.0,5.05) {$b$};
\draw[dim] (5.0,5.05) -- (6.0,5.05);
\node[blue!70!black, above] at (5.5,5.05) {$b$};

\draw[extbox]  (4.0,2.75) rectangle (8.5,3.25);
\draw[keepbox] (5.0,2.75) rectangle (7.5,3.25);
\node[lab] at (6.25,3.00) {local solve};

\draw[dim] (4.0,3.50) -- (5.0,3.50);
\node[blue!70!black, above] at (4.5,3.50) {$b$};
\draw[dim] (7.5,3.50) -- (8.5,3.50);
\node[blue!70!black, above] at (8.0,3.50) {$b$};

\draw[extbox]  (6.5,1.2) rectangle (11.0,1.7);
\draw[keepbox] (7.5,1.2) rectangle (10.0,1.7);
\node[lab] at (8.75,1.45) {local solve};

\draw[dim] (6.5,1.95) -- (7.5,1.95);
\node[blue!70!black, above] at (7.0,1.95) {$b$};
\draw[dim] (10.0,1.95) -- (11.0,1.95);
\node[blue!70!black, above] at (10.5,1.95) {$b$};

\node at (5.5,6.1) {Solve local subproblems in parallel};

\node[biglab] at (6.15,-0.75) {Compose only the exclusive pieces};

\draw[horizon] (2.5,-0.1) -- (10.0,-0.1);
\draw[keepbox] (2.5,-0.28) rectangle (5.0,0.08);
\draw[keepbox] (5.0,-0.28) rectangle (7.5,0.08);
\draw[keepbox] (7.5,-0.28) rectangle (10.0,0.08);

\draw[->] (3.75,4.3) -- (3.75,0.1);
\draw[->] (6.25,2.75) -- (6.25,0.1);
\draw[->] (8.75,1.2) -- (8.75,0.1);

\end{tikzpicture}
\caption{Demonstration of overlapping temporal decomposition structure. Every subproblem is augmented with an overlap size of $b$, solved in parallel, and then the solution iterate is composed by removing the overlaps of each problem. }
\label{fig:otdScheme}
\end{figure}

We now introduce the formalization of the OTD paradigm, detailed in Figure \ref{fig:otdScheme}. Let $M\in\{1,\ldots,N\}$ be the number of temporal subproblems and $0=n_0<\cdots<n_M=N$ the partition knots. 
The disjoint cores~are
\begin{align*}
\mathcal K_i=[n_i,n_{i+1})\ (i<M-1),\qquad \mathcal K_{M-1}=[n_{M-1},N].
\end{align*}
For each $i\in[M-1]$, a nominal overlap $b_i\in\mathbb Z_{\geq1}$ defines the extended window by
\begin{align}\label{eq:extended-intervals}
m_1^i=\max\{n_i-b_i,0\},\enspace m_2^i=\min\{n_{i+1}+b_i,N\}\,,
\end{align}
where we cap the overlaps at the endpoints such that
\begin{align}\label{eq:effective-overlaps}
b_i^L:=\min\{b_i,n_i\},\qquad b_i^R:=\min\{b_i,N-n_{i+1}\}\,. 
\end{align}
Thus, $m_1^i=n_i-b_i^L$ and $m_2^i=n_{i+1}+b_i^R$. On each window, define the nonlinear local OCP $\mathcal P_\mu^i(\bm d_i)$ by
\begin{subequations}\label{eq:decomposed-main-problem}
\begin{align}
\min_{\tilde{\bm{x}},\,\tilde{\bm{u}}}\;
& \sum\nolimits_{k=m_1^i}^{m_2^i-1} g_k(\tilde{\bm{x}}_k, \tilde{\bm{u}}_k)
+ \tilde{g}_{m_2^i}(\tilde{\bm{x}}_{m_2^i}; \bm{d}_{i,2:4}), \label{eq:dec-obj}\\
\text{s.t.}\;\;
& \tilde{\bm{x}}_{k+1} = f_k(\tilde{\bm{x}}_k, \tilde{\bm{u}}_k),
\; k \in [m_1^i, m_2^i), \label{eq:dec-dyn}\\
& \tilde{\bm{x}}_{m_1^i} = \bm{d}_{i,1}, \label{eq:dec-init}
\end{align}
\end{subequations}
with boundary input $\bm d_i=(\bar{\bm x}_{m_1^i};\bar{\bm x}_{m_2^i};\bar{\bm u}_{m_2^i};\bar{\bm\lambda}_{m_2^i+1})$. For $m_2^i<N$ and $\mu > 0$, we use the regularized terminal cost introduced in \cite{naFOTD} as
\begin{multline*}
\tilde g_{m_2^i}(\tilde{\bm x}_{m_2^i};\bm d_{i,2:4}) =g_{m_2^i}(\tilde{\bm x}_{m_2^i},\bar{\bm u}_{m_2^i})\\
\hskip2cm -\bar{\bm\lambda}_{m_2^i+1}^Tf_{m_2^i}(\tilde{\bm x}_{m_2^i},\bar{\bm u}_{m_2^i})+\frac\mu2\|\tilde{\bm x}_{m_2^i}-\bar{\bm x}_{m_2^i}\|^2\,, 
\end{multline*}
and for $m_2^i=N$, it is $\tilde g_N(\tilde{\bm x}_N)=g_N(\tilde{\bm x}_N)$ where only $\bm d_{i,1}=\bar{\bm x}_{m_1^i}$ is needed. Notice, for each subproblem, the left boundary is an equality constraint, while $\bm d_{i,2:4}$ enters the terminal cost. The reason for this design is that \cite[Theorem 2.1]{naFOTD} shows that any KKT point of the full problem \eqref{eq:main-problem-compact-form} is a KKT point of all the subproblems in \eqref{eq:decomposed-main-problem} when $\bm d_i$ is set from that KKT point for each $i$; moreover, any KKT point of all the subproblems \emph{whose boundaries agree} is a KKT point of \eqref{eq:main-problem-compact-form}. Now, given that $\bm{d}_i$, the boundary conditions, are given for every subproblem, \eqref{eq:decomposed-main-problem} can be solved in parallel.

Lastly, as illustrated in Figure \ref{fig:otdScheme}, we define the composition to keep only the interior of each subproblem. Given local solutions $\{(\tilde{\bm z}_i,\tilde{\bm\lambda}_i)\}_{i=0}^{M-1}$, the composition operator $\mathcal{C}$ yields
\begin{align*}
\mathcal{C}(\{(\tilde{\bm{z}}_i, \tilde{\bm{\lambda}}_i)\}_i) = (\bm{z}, \bm{\lambda})\,, 
\end{align*}
where $(\bm{z}_k, \bm{\lambda}_k) = (\tilde{\bm{z}}_{i, k}, \tilde{\bm{\lambda}}_{i, k})$ if $k \in \mathcal K_i$ for $i\in [M-1]$. 

Analogously, for the full-horizon variables $(\bm{z}, \bm{\lambda})$, we define the \textit{decomposition operator} $\mathcal{D}$ as  
\begin{align*}
\mathcal{D}(\bm{z}, \bm{\lambda}) = \big\{\mathcal{D}_i(\bm{z}, \bm{\lambda})\big\}_{i=0}^{M-1}\,, 
\end{align*}
where each component $\mathcal{D}_i(\bm{z}, \bm{\lambda}) = (\tilde{\bm{z}}_i, \tilde{\bm{\lambda}}_i)$ is given by $\tilde{\bm{z}}_i = (\tilde{\bm{x}}_i, \tilde{\bm{u}}_i) = (\bm{x}_{m_1^i:m_2^i}, \bm{u}_{m_1^i:m_2^i-1})$ and the dual variables are given by $ \tilde{\bm{\lambda}}_i = \bm{\lambda}_{m_1^i:m_2^i}.$

\section{Main Algorithm} \label{sec:main-algorithm-design}

\subsection{Adaptive Overlapping Temporal Decomposition}
AOTD combines the SQP framework and temporal decomposition of Section~\ref{sec:preliminaries} through two adaptive mechanisms. First, for every OCP in \eqref{eq:decomposed-main-problem}, we apply SQP locally where the resulting Newton direction is solved approximately to a tolerance $\epsilon_i^\tau$ as specified in Section~\ref{sec:localInexactness}.

Second, each subproblem introduces two distinct sources of error: one arising from the domain decomposition and another from the inexact solution of the Newton direction. To control their combined impact on the global solution, we scale the overlap size of each subproblem adaptively with its contribution to the global residual. On a failed global accuracy test, every subproblem is updated, with larger residual shares producing stricter tolerance reductions and larger overlap increments.

At iteration $\tau$, let $\epsilon^\tau>0$ denote the global accuracy tolerance, $\epsilon_i^\tau>0$ the local solve tolerances, and $\eta_1^\tau,\eta_2^\tau>0$ the penalty parameters of the augmented Lagrangian in \eqref{eq:augmentedLagrangian}. Then, each algorithm iteration can be summarized in four steps:

\noindent \underline{Step 1:} \textbf{Iteration setup.} Compute the Hessian, modified Hessian, and Jacobian, and set the adaptive parameters $\eta_1^\tau$, $\eta_2^\tau$, $\epsilon^\tau$, $\epsilon_i^\tau$, and $b_i^\tau$ to their initial values at $\tau=0$ or carry them over from the preceding iteration.

\noindent \underline{Step 2:} \textbf{Local quadratic and KKT solves (in parallel).} Form each local quadratic subproblem approximation of \eqref{eq:decomposed-main-problem} with overlap size $b_i^\tau$,  solve its local KKT linear system approximately to tolerance $\epsilon_i^\tau$, with boundary conditions $\bm{d}_i^\tau = (\bm{0}; \bm{0}; \bm{0}; \bm{0})$ (any subproblem with $m_1^i = 0$ uses the correcting initial-state $\bm{d}_{i,1}^\tau = \bar{\bm{x}}_0 - \bm{x}_0^\tau$).

\noindent \underline{Step 3:} \textbf{Acceptance and adaptive updates.} Compose the local primal--dual directions, then evaluate the global accuracy and descent conditions:

\begin{enumerate}[label=(\alph*), leftmargin=3em, itemsep=0.5ex]

\item If the accuracy condition fails, update each $\epsilon_i^\tau$ and $b_i^\tau$ according to the corresponding subproblem's contribution to the global residual, then return to Step~2.
\newpage 

\item If the accuracy condition holds but the descent condition fails, update $\epsilon^\tau$ and the augmented-Lagrangian penalty parameters $\eta_1^\tau$ and $\eta_2^\tau$, then return to Step~2.

\item If both conditions hold, accept the approximate Newton step.

\end{enumerate}

\noindent \underline{Step 4:} \textbf{Step size selection and iterate update.} Select the step size $\alpha^\tau$ via a line search using the augmented Lagrangian \eqref{eq:augmentedLagrangian} as the merit function and then update the iterate.

We now formalize the design of each step. 

\subsection{Local Quadratic Subproblems and KKT Systems}\label{sec:local-quadratic-subproblems}

In Step~2, AOTD forms the local SQP subproblem of \eqref{eq:decomposed-main-problem} on each adaptive window and approximately solves its KKT system. All coefficients are evaluated at $(\bm z^\tau,\bm\lambda^\tau)$. Suppressing $i$ on the endpoints and local coefficients, the quadratic subproblem $\mathcal{LP}_\mu^i(\bm d_i)$ is
\begin{align}
\min\quad&\sum_{k=m_1}^{m_2-1}\tilde g_k(\bm p_k,\bm q_k)+\tilde g_{m_2}(\bm p_{m_2};\bm d_{2:4}), \notag\\
\text{s.t.}\quad&\bm p_{m_1}=\bm d_1, \notag\\
&\bm p_{k+1}=\tilde f_k(\bm p_k,\bm q_k),\quad k\in[m_1,m_2)\,, 
\notag 
\end{align}
where, for $k\in [m_1, m_2)$, the stage cost and dynamics are
\begin{align*}
\tilde g_k(\bm p_k,\bm q_k)
&=\frac12\begin{bmatrix}\bm p_k\\\bm q_k\end{bmatrix}^T
\begin{bmatrix}\hat Q_k&\hat S_k^T\\\hat S_k&\hat R_k\end{bmatrix}
\begin{bmatrix}\bm p_k\\\bm q_k\end{bmatrix}\\
&\quad+\begin{bmatrix}\nabla_{\bm x_k}\mathcal L^\tau\\\nabla_{\bm u_k}\mathcal L^\tau\end{bmatrix}^T
\begin{bmatrix}\bm p_k\\\bm q_k\end{bmatrix},\\
\tilde f_k(\bm p_k,\bm q_k)&=A_k^\tau\bm p_k+B_k^\tau\bm q_k-\nabla_{\bm\lambda_{k+1}}\mathcal L^\tau\,, 
\end{align*}
where recall $A_k^\tau:=\nabla_{\bm x_k}f_k(\bm x_k^\tau,\bm u_k^\tau)$ and $B_k^\tau:=\nabla_{\bm u_k}f_k(\bm x_k^\tau,\bm u_k^\tau)$. For $m_2<N$, the terminal cost is
\begin{align*}
\tilde g_{m_2}(\bm p_{m_2};\bm d_{2:4})
&=\frac12\begin{bmatrix}\bm p_{m_2}\\\bm d_3\end{bmatrix}^T
\begin{bmatrix}\hat Q_{m_2}&\hat S_{m_2}^T\\\hat S_{m_2}&\hat R_{m_2}\end{bmatrix}
\begin{bmatrix}\bm p_{m_2}\\\bm d_3\end{bmatrix}\\*
&\quad+\nabla_{\bm x_{m_2}}^T\mathcal L^\tau\bm p_{m_2}-\bm d_4^TA_{m_2}^\tau\bm p_{m_2}\\*
&\quad+\frac\mu2\|\bm p_{m_2}-\bm d_2\|^2\,,
\end{align*}
and, for $m_2=N$, the boundary arguments are omitted and the terminal cost is
\begin{align*}
\tilde g_N(\bm p_N)=\tfrac12\bm p_N^T\hat Q_N\bm p_N+\nabla_{\bm x_N}^T\mathcal L^\tau\bm p_N.
\end{align*}

The local primal directions $\bm p_k=\Delta\tilde{\bm x}_k$, $\bm q_k=\Delta\tilde{\bm u}_k$ and QP multipliers $\bm\zeta_k$ are stacked as $\Delta\tilde{\bm\omega}(\bm d_i)=(\bm p_{m_1},\allowbreak\bm q_{m_1},\allowbreak\ldots,\allowbreak\bm p_{m_2-1},\allowbreak\bm q_{m_2-1},\allowbreak\bm p_{m_2})$ and $\Delta\tilde{\bm\zeta}(\bm d_i)=(\bm\zeta_{m_1},\ldots,\bm\zeta_{m_2})$. Because the linear term in the local quadratic objective uses $\nabla\mathcal L^\tau$, its multiplier $\bm\zeta_k$ represents the local dual increment. Thus $\Delta\tilde{\bm\lambda}_k=\bm\zeta_k$, and the local primal--dual direction is $(\Delta\tilde{\bm\omega},\Delta\tilde{\bm\zeta})$. The QP Lagrangian is
\begin{align}
\nonumber &\mathcal L_{QP}^{(i)}(\Delta\tilde{\bm\omega}(\bm d_i),\Delta\tilde{\bm\zeta}(\bm d_i);\bm d_i)\\
\nonumber&\quad=\sum_{k=m_1}^{m_2-1}\tilde g_k(\bm p_k,\bm q_k)+\tilde g_{m_2}(\bm p_{m_2};\bm d_{2:4})\\
&\qquad+\sum_{k=m_1}^{m_2-1}\bm\zeta_{k+1}^T\bigl(\bm p_{k+1}-\tilde f_k(\bm p_k,\bm q_k)\bigr) +\bm\zeta_{m_1}^T(\bm p_{m_1}-\bm d_1). \notag
\end{align}

The exact Newton direction $(\Delta\tilde{\bm\omega}^\ast(\bm d_i),\Delta\tilde{\bm\zeta}^\ast(\bm d_i))$ solves
\begin{align}\label{eq:newton}
\underbrace{\begin{pmatrix}\hat H_i^\tau&(G_i^\tau)^T\\G_i^\tau&0\end{pmatrix}}_{\Gamma_i^\tau}
\begin{pmatrix}\Delta\tilde{\bm\omega}^\ast\\\Delta\tilde{\bm\zeta}^\ast\end{pmatrix}
=-\nabla\mathcal L_{\mathrm{QP}}^{(i),\tau},
\end{align}
where the fixed right-hand side is defined by
\begin{align*}
\nabla\mathcal L_{\mathrm{QP}}^{(i),\tau}:=
\left.\begin{pmatrix}\nabla_{\bm\omega}\mathcal L_{\mathrm{QP}}^{(i)}\\\nabla_{\bm\zeta}\mathcal L_{\mathrm{QP}}^{(i)}\end{pmatrix}
\right|_{(\Delta\tilde{\bm\omega},\Delta\tilde{\bm\zeta})=(\bm0,\bm0)}.
\end{align*}
The local modified Hessian is 
\begin{align*}
\hat H_i^\tau
&= \diag\!\left(\hat H_{m_1}^\tau,\ldots,\hat H_{m_2-1}^\tau,
\hat H_{i,\mathrm T}^\tau\right), \\
\hat H_{i,\mathrm T}^\tau
&\coloneqq
\begin{cases}
\hat Q_{m_2}+\mu\mathbb I, & m_2<N,\\
\hat Q_N, & m_2=N.
\end{cases}
\end{align*}
Similar to the global design, the local Jacobian $G_i$ has an identity block for $\bm p_{m_1}=\bm d_1$ and successive block rows $(-A_k,-B_k,\mathbb I)$ in the columns of $(\bm p_k,\bm q_k,\bm p_{k+1})$ for $k\in[m_1,m_2)$, with zeros elsewhere. We now specify the residual tolerance $\epsilon_i^\tau$ for the inexact solution of~\eqref{eq:newton}.

\subsection{Local KKT-System Inexactness} \label{sec:localInexactness}
We begin by defining the residual condition and corresponding error for each local KKT linear system. In particular, consider the approximate local primal--dual direction $(\Delta \bm{\tilde{\omega}}_i, \Delta \bm{\tilde{\zeta}}_i)$. Its residual in \eqref{eq:newton} is
\begin{align}
\bm{\tilde{r}}_i(\Delta \bm{\tilde{\omega}}_i, \Delta \bm{\tilde{\zeta}}_i) \coloneq \Gamma_i^\tau \begin{pmatrix}
\Delta \bm{\tilde{\omega}}_i \\ \Delta \bm{\tilde{\zeta}}_i
\end{pmatrix} + \nabla \mathcal{L}_{\mathrm{QP}}^{(i), \tau}, \label{eq:local-residual-def}
\end{align}
where $\nabla \mathcal{L}_{\mathrm{QP}}^{(i), \tau} = \bm{\tilde{r}}_i(\bm{0}, \bm{0})$ is the fixed local KKT right-hand side defined in \eqref{eq:newton}, which is held constant throughout the local solve.

Note that this residual requires only the local quadratic-subproblem matrices and Lagrangian. We then require each local KKT residual to satisfy
\begin{align} \label{eq:local-residual-cond}
\|\bm{\tilde{r}}_i(\Delta \bm{\tilde{\omega}}_i, \Delta \bm{\tilde{\zeta}}_i)\| &\leq \epsilon_i^\tau \|\nabla \mathcal{L}_{\mathrm{QP}}^{(i), \tau} \| \,,
\end{align}
where $\epsilon_i^\tau$ is determined via the local subproblem's contribution to the global accuracy (see \eqref{eq:subproblem-error-update}). 


Throughout, we assume that every local solver used in Algorithm~\ref{alg:aotd-compact} can satisfy \eqref{eq:local-residual-cond} for any prescribed $\epsilon_i^\tau>0$ in finitely many iterations, although its residuals need not decrease monotonically. For randomized solvers, we assume that this property holds almost surely.

\subsection{Global Problem Inexactness}
\subsubsection{Merit Function}
For positive penalty parameters $\bm\eta=(\eta_1,\eta_2)$, we use the differentiable exact augmented Lagrangian
\begin{align}\label{eq:augmentedLagrangian}
\nonumber  \mathcal L_{\bm\eta}(\bm z,\bm\lambda)
&=\mathcal L(\bm z,\bm\lambda)+\tfrac{\eta_1}{2}\|\nabla_{\bm\lambda}\mathcal L(\bm z,\bm\lambda)\|^2\\
&\quad+\tfrac{\eta_2}{2}\|\nabla_{\bm z}\mathcal L(\bm z,\bm\lambda)\|^2.
\end{align}
Here $\eta_1$ weights the constraint violation and $\eta_2$ weights stationarity. Every KKT point is stationary for this function, with converse results under suitable regularity and penalty choices~\cite{doi:10.1137/0317044,alBook}. We use it as a merit function, without requiring KKT points to be minimizers. Throughout, we use the short hand notation to represent this augmented Lagrangian $\mathcal L_{\bm\eta}^\tau=\mathcal L_{\bm\eta}(\bm z^\tau,\bm\lambda^\tau)$. 

\subsubsection{Global Concatenation Accuracy Condition}\label{sec:global-concatenation-accuracy}
We now define a global accuracy condition that controls the discrepancy between the concatenated direction and the exact Newton direction of the full problem. Let $(\Delta \tilde{\bm{z}}, \Delta \tilde{\bm{\lambda}}) =\mathcal{C}(\{(\Delta \bm{\tilde{\omega}}, \Delta \bm{\tilde{\zeta}})\}_i)$ be the concatenated direction generated by solving each subproblem to $\epsilon_i^\tau$ tolerance. Then, we define the global residual analogous to the local subproblem residual in Section \ref{sec:localInexactness}
\begin{align*}
\bm{r}(\Delta \tilde{\bm{z}}, \Delta \tilde{\bm{\lambda}}) =  \underbrace{\begin{pmatrix}
\hat{H}^{\tau} & (G^\tau)^T \\ G^\tau & 0 
\end{pmatrix}}_{\Gamma^\tau} \begin{pmatrix}
\Delta \tilde{\bm{z}} \\ \Delta \tilde{\bm{\lambda}}
\end{pmatrix} + \begin{pmatrix}
\nabla_{\bm{z}} \mathcal{L}^\tau \\ \nabla_{\bm{\lambda}} \mathcal{L}^\tau
\end{pmatrix},
\end{align*}
where $\hat{H}^\tau$, $G^\tau$ and $\mathcal{L}^\tau$ are the global modified Hessian, Jacobian, and Lagrangian. The residual condition then becomes 
\begin{subequations}\label{eq:globalResidualBound}
\begin{align}
\|\bm{r}(\Delta \tilde{\bm{z}}, \Delta \tilde{\bm{\lambda}})\| &\leq \frac{\epsilon^\tau \|\nabla\mathcal{L}^\tau\|}{\|\Gamma^\tau\|\Psi^\tau}\,,  \\  
\Psi^\tau &\coloneqq \frac{20 \max(\|\hat{H}^\tau\|^2, 1)}{\min\{\gamma_{\rm RH}, 1\}\min\{(\sigma_{\min}^\tau)^2, 1\}}.
\end{align}
\end{subequations}
Here, $\gamma_{\rm RH}$ is a lower bound on the reduced Hessian (see Assumption~\ref{assumption:uniform-lower-bound-hessian}), $\sigma_{\min}^\tau:=\sigma_{\min}(G^\tau)$, and $\epsilon^\tau$ is an adaptive accuracy parameter satisfying $\epsilon^\tau\leq\epsilon^\tau_{\text{max}}$, with $\epsilon^\tau_{\text{max}}$ and $\Upsilon^\tau$ defined as:
\begin{subequations}\label{eq:globalEpsilonMin}
\begin{align}
\epsilon^\tau_{\text{max}} &\coloneqq \frac{(0.5-\beta)\eta_2^\tau}{(1+\eta_1^\tau + \eta_2^\tau)(\Psi^\tau\Upsilon^\tau)^2}\,,  \\ 
\Upsilon^\tau &\coloneqq \max\{\|G
^\tau \|, \|\hat{H}^\tau\|, 1\}\,. 
\end{align}
\end{subequations}
$\eta_1^\tau, \eta_2^\tau$ are both adaptive penalty parameters as given in the augmented Lagrangian function \eqref{eq:augmentedLagrangian}. The residual condition \eqref{eq:globalResidualBound} controls the error in the composed direction relative to the exact global Newton direction. Equation~\eqref{eq:globalEpsilonMin} caps $\epsilon^\tau$ so that the inexactness is small relative to the curvature of the merit function, which allows $\alpha^\tau=1$ eventually (Lemma~\ref{lem:unit-step-size}). 

\begin{remark} \label{rem:computable-constants}
Exact spectral quantities in \eqref{eq:globalResidualBound}--\eqref{eq:globalEpsilonMin} may be costly to compute, but conservative bounds suffice. Operator-norm upper bounds follow cheaply from block row sums, while a lower bound for $\sigma_{\min}(G^\tau)$ can be certified by standard eigenvalue bisection with inertia counts from block $LDL^T$ factorizations of $G^\tau(G^\tau)^T$ \cite{golubVanLoan}. Moreover, for a prescribed $\gamma_{\rm RH}$, the reduced-Hessian condition can be certified by increasing $c$ and applying parallel Cholesky factorization to $\hat H^\tau+c(G^\tau)^TG^\tau-\gamma_{\rm RH}\mathbb I$ \cite{nocedalAndWright,1039748,doi:10.1137/0912044}.
\end{remark}

We now design the update of the local adaptive parameters $\epsilon_i^\tau$, $b_i^\tau$ based on each subproblem's contribution to the global residual. Since the overlaps adapt, the windows of Section~\ref{sec:preliminaries} and the local systems of Section~\ref{sec:local-quadratic-subproblems} use the current, iteration-dependent endpoints $m_1^i = n_i - b_i^{L,\tau}$ and $m_2^i = n_{i+1} + b_i^{R,\tau}$ (per \eqref{eq:extended-intervals}--\eqref{eq:effective-overlaps} with $b_i \to b_i^\tau$).  Let $\bm r_i$ be the restriction of $\bm r(\Delta\tilde{\bm z},\Delta\tilde{\bm\lambda})$ to the core $\mathcal K_i$ and, for brevity, write $\bm r = \bm r(\Delta\tilde{\bm z},\Delta\tilde{\bm\lambda})$. Then, if the accuracy condition is not satisfied, the individual tolerances and overlaps are updated via
\begin{subequations}
\begin{align}
\epsilon_i^\tau&\gets\epsilon_i^\tau\exp\!\left(-\varrho_\epsilon\left[\frac{\|\bm r_i\|}{\|\bm r\|}+\frac1{\sqrt M}\right]\right),\label{eq:subproblem-error-update}\\
b_i^\tau&\gets b_i^\tau+\max\left\{1,\left\lceil\varrho_b\frac{\|\bm r_i\|}{\|\bm r\|}\right\rceil\right\},\label{eq:subproblem-overlap-update}
\end{align}
\end{subequations}
where $\varrho_\epsilon, \varrho_b > 0$ are the adaptive gains chosen by the user. The nominal overlap $b_i^\tau$ enters the subproblem only through the clipped effective overlaps $b_i^{L,\tau} = \min\{b_i^\tau,\, n_i\}$ and $b_i^{R,\tau} = \min\{b_i^\tau,\, N - n_{i+1}\}$ of \eqref{eq:effective-overlaps}, so each window is capped at the full horizon.

The design of this update law ensures two properties. First, the floor $1/\sqrt M$ in \eqref{eq:subproblem-error-update} and the $\max$ in \eqref{eq:subproblem-overlap-update} guarantee \emph{uniform progress}: every failing pass contracts every tolerance by at least the fixed factor $e^{-\varrho_\epsilon/\sqrt M}$ and increases every nominal overlap by at least one stage, which ensures the accuracy loop terminates and the adaptive parameters stabilize in finitely many updates (see Lemmas~\ref{lem:global-adaptive-well-posed} and \ref{lem:stability}). Second, the updates match the error structure of Lemma \ref{lem:global-residual-error}: the boundary error decays \emph{exponentially} in the overlap, so the overlaps grow additively (the ceiling yields at least one full stage), while the inexactness error is \emph{linear} in the tolerance, so $\epsilon_i^\tau$ contracts multiplicatively by an exponential to keep pace.

\begin{remark} \label{rem:preconditioning}
Because enlarging the overlap changes the local KKT matrix, any preconditioner for solving the Newton direction may need to be updated or rebuilt when the effective window changes. With an appropriate structure-exploiting preconditioner, such as a fixed-fill ILU or block-Jacobi construction, this rebuild can remain inexpensive in practice, but should be considered.
\end{remark}

\subsubsection{Global Descent Direction Condition}
Arbitrary choices of the penalty parameters $\eta_1^\tau, \eta_2^\tau$ need not satisfy the sufficient conditions of Lemma~\ref{lem:wellposed-global-descent}, so accuracy of the concatenated direction alone does not guarantee descent for the exact augmented Lagrangian. Hence, we enforce the  global descent condition:
\begin{align}
(\nabla \mathcal{L}_{\bm{\eta}^\tau}^\tau)^T \begin{pmatrix} \Delta \tilde{\bm{z}}  \\ \Delta \tilde{\bm{\lambda}} \end{pmatrix} \leq -\eta_2^\tau \|\nabla \mathcal{L}^\tau\|^2 / 2\,. \label{eq:descent-direction-condition}
\end{align}
When the global descent direction condition is not satisfied, we update the penalty parameters using a user-chosen factor $\nu>1$:
\begin{align}
\eta_1^\tau \gets \eta_1^\tau \nu^2\,, \quad  
\eta_2^\tau \gets \eta_2^\tau / \nu .
\label{eq:adaptive-parameter-update}
\end{align}
Given the update in the penalty parameters, we recompute $\epsilon^\tau_{\text{max}}$ and update $\epsilon^\tau$ via
\begin{align}
\epsilon^\tau 
\gets 
\min\left\{\epsilon^\tau_{\text{max}}, 
\epsilon^\tau/\nu^4\right\}.
\label{eq:adaptive-epsilon-update}
\end{align}
As above, these update laws make uniform progress. Lemma~\ref{lem:stability} shows that all global and local adaptive parameters stabilize after finitely many updates.

\subsubsection{Line Search and Global Iteration Update}
To complete the update, once the accuracy conditions are satisfied, we choose a step-size $\alpha^\tau$ by computing a backtracking line search to enforce the Armijo condition on the exact augmented Lagrangian
\begin{align}
\mathcal L_{\bm\eta^\tau}\Bigg(\begin{pmatrix}\bm z^\tau\\\bm\lambda^\tau\end{pmatrix}&+\alpha^\tau\begin{pmatrix}\Delta\tilde{\bm z}^\tau\\\Delta\tilde{\bm\lambda}^\tau\end{pmatrix}\Bigg)\notag\\*
&\quad\leq\mathcal L_{\bm\eta^\tau}^\tau+\alpha^\tau\beta(\nabla\mathcal L_{\bm\eta^\tau}^\tau)^T\begin{pmatrix}\Delta\tilde{\bm z}^\tau\\\Delta\tilde{\bm\lambda}^\tau\end{pmatrix},
\label{eq:line-search-cond}
\end{align}
where $\beta \in (0, 0.5)$. In particular, the step size is initialized as $\alpha^\tau \gets 1$ and repeatedly reduced via $\alpha^\tau \gets \delta \alpha^\tau$ with a fixed backtracking factor $\delta \in (0,1)$ until \eqref{eq:line-search-cond} holds. Then the next iteration is computed via
\begin{align}
\begin{pmatrix}
\bm{z}^{\tau+1}\\\bm{\lambda}^{\tau+1} 
\end{pmatrix} = \begin{pmatrix}
\bm{z}^{\tau}\\\bm{\lambda}^{\tau} 
\end{pmatrix} + \alpha^\tau \begin{pmatrix}
\Delta \tilde{\bm{z}}^{\tau}\\ \Delta \tilde{\bm{\lambda}}^{\tau} 
\end{pmatrix} \label{eq:iterate-update}\,. 
\end{align}

\subsection{Full Algorithm and Computational Complexity}

\begin{algorithm}[tp]
\caption[Adaptive Overlapping Temporal Decomposition]{Adaptive Overlapping Temporal\protect\newline Decomposition}
\label{alg:aotd-compact}
\small
\algrenewcommand\algorithmicindent{1em}
\begin{algorithmic}[1]
\Require Initial iterate $(\bm z^0,\bm\lambda^0)$, $\eta_1^0,\eta_2^0,\epsilon^0>0$, local tolerances $\epsilon_i^0>0$ and overlaps $b_i^0\in\mathbb Z_{\geq1}$ for all $i$, $\beta\in(0,0.5)$, $\delta\in(0,1)$, $\varrho_b,\varrho_\epsilon>0$, $\nu>1$, $\mu>0$, stopping tolerance $\mathrm{tol}_{\rm KKT}\geq0$, and the partition knots $\{n_i\}$.
\For{$\tau=0,1,2,\ldots$}
\State Compute $\nabla\mathcal L^\tau$.
\If{$\|\nabla\mathcal L^\tau\|
\leq \mathrm{tol}_{\rm KKT}$}
\State \Return $(\bm z^\tau,\bm\lambda^\tau)$.
\EndIf
\State Compute $\hat H^\tau,G^\tau,\Gamma^\tau,
\Upsilon^\tau,\Psi^\tau$ and
$\epsilon_{\rm max}^\tau$ via
\eqref{eq:globalEpsilonMin}.
\If{$\epsilon^\tau>\epsilon_{\rm max}^\tau$}
\State $\epsilon^\tau\gets\epsilon_{\rm max}^\tau/\nu$.
\EndIf
\State $(\Delta\tilde{\bm z},\Delta\tilde{\bm\lambda})\gets\bm0$.
\While{accuracy \eqref{eq:globalResidualBound} or descent \eqref{eq:descent-direction-condition} fails}
\While{accuracy \eqref{eq:globalResidualBound} fails}
\State \textbf{In parallel, for each subproblem $i$:}
\State \hspace{\algorithmicindent}Set $\bm d_i^\tau=\bm0$ (if $m_1^i=0$, set
$\bm d_{i,1}^\tau=\bar{\bm x}_0-\bm x_0^\tau$).
\State \hspace{\algorithmicindent}Solve local QP with overlap $b_i^\tau$ to tolerance $\epsilon_i^\tau$.
\State Compose $(\Delta\tilde{\bm z},\Delta\tilde{\bm\lambda})\gets
\mathcal C(\{(\Delta\tilde{\bm\omega}_i,
\Delta\tilde{\bm\zeta}_i)\}_i)$.
\If{accuracy \eqref{eq:globalResidualBound} fails}
\State Update $\epsilon_i^\tau,b_i^\tau$ via
\eqref{eq:subproblem-error-update}--\eqref{eq:subproblem-overlap-update}.
\EndIf
\EndWhile
\If{descent \eqref{eq:descent-direction-condition} fails}
\State Update $\eta_1^\tau,\eta_2^\tau,\epsilon^\tau$ via
\eqref{eq:adaptive-parameter-update}, \eqref{eq:adaptive-epsilon-update}.
\EndIf
\EndWhile
\State Choose $\alpha^\tau$ by \eqref{eq:line-search-cond}.
Update the iterate by \eqref{eq:iterate-update}.
\State Carry $\eta_1^\tau,\eta_2^\tau,\epsilon^\tau,
\{\epsilon_i^\tau,b_i^\tau\}_i$ to iteration $\tau+1$.
\EndFor
\end{algorithmic}
\end{algorithm}

The AOTD design, summarized in Algorithm~\ref{alg:aotd-compact}, avoids selecting a sufficiently large overlap in advance and solves every local KKT system to only the required tolerance governed by the global acceptance conditions. The two nested \texttt{while} loops in Algorithm~\ref{alg:aotd-compact}, which coordinate the accuracy and descent tests, take inspiration from the AdaSketch-Newton design of Hong et al.~\cite[Algorithm~1]{pmlr-v202-hong23b}. We next quantify the resulting solve-phase savings by comparing a local subproblem window of length $N_{\rm loc}$ with a fixed-overlap window of length $N_{\rm loc}+\Delta N_{\rm loc}$. 

Consider a local subproblem with horizon length $N_{\rm loc}$, $n_x+n_u$ primal variables per stage, and $n_x$ constraints per stage. The KKT system then has dimension $O(N_{\rm loc}(2n_x+n_u))$ and half-bandwidth $2n_x+n_u$, so a block-banded direct solve costs $O(N_{\rm loc}(2n_x+n_u)^3)$ FLOPs with $O(N_{\rm loc}(2n_x+n_u)^2)$ storage, and each $\Delta N_{\rm loc}$ stages of extra overlap add $O(\Delta N_{\rm loc}(2n_x+n_u)^3)$ FLOPs. However, AOTD does not require a full solve. Hence, a $K$-iteration GMRES solve costs $O(K N_{\rm loc}(2n_x+n_u)^2 + K^2 N_{\rm loc}(2n_x+n_u))$. Thus, if adaptivity saves $\Delta K$ iterations and $\Delta N_{\rm loc}$ overlap stages per subproblem relative to a fixed-overlap solve, the resulting savings are of order $O\bigl((K\Delta N_{\rm loc} + N_{\rm loc}\Delta K)(2n_x+n_u)^2 + (K^2\Delta N_{\rm loc} + N_{\rm loc}K\Delta K)(2n_x+n_u)\bigr)$. As confirmed in Section~\ref{sec:simulations}, accumulated over $M$ subproblems and for the large state and control spaces arising in optimal control of PDEs, these savings become substantial.

Now, given we have detailed the AOTD design, we next establish its theoretical convergence properties. 

\section{Global Convergence} \label{sec:global-convergence}
For all theoretical results in this paper, we provide proofs in Appendices \ref{appendix:proofs1} and \ref{appendix:proofs2}.

\subsection{Well-Posedness of the Local Subproblems}
We begin by introducing the following sufficient assumptions for studying the global convergence of the system.
\begin{assumption}[Uniform lower bound on the reduced Hessian] \label{assumption:uniform-lower-bound-hessian}
For each SQP iteration $\tau \geq 0$, let $Z^\tau$ be a matrix whose columns form an orthonormal basis for the null space of the linearized constraint Jacobian $G^\tau$ (i.e., $\mathrm{range}(Z^\tau) = \mathrm{null}(G^\tau)$ and $(Z^\tau)^T Z^\tau = \mathbb{I}$). We assume that there exists a constant $\gamma_{\rm RH} > 0$, independent of $\tau$, such that
\begin{equation}\label{nsequ:1}
(Z^\tau)^T \hat{H}^\tau Z^\tau \succeq \gamma_{\rm RH} \mathbb{I}.
\end{equation}
\end{assumption}
Whenever the original Hessian satisfies $(Z^\tau)^T H^\tau Z^\tau \succeq \gamma_{\rm RH}\mathbb I$, we retain $\hat H^\tau=H^\tau$. Otherwise, one sufficient modification is the stagewise shift $\hat H_k^\tau=H_k^\tau+(\gamma_{\rm RH}+\|H_k^\tau\|)\mathbb I$ for $k\in[N]$ (see \cite[Chapter 3.4]{nocedalAndWright}). These shifts can be computed independently in parallel and preserve the block-diagonal structure. They ensure positive definiteness on the full space, which is sufficient (and slightly stronger than required) for Assumption~\ref{assumption:uniform-lower-bound-hessian}.

The local convergence results additionally require $\|\hat H^\tau-H^\tau\|\to0$ (Assumption~\ref{assumption:modified-hessian-vanishes}). For a sequence converging to a KKT pair with a positive definite reduced Hessian, choosing $\gamma_{\rm RH}$ strictly below its limiting smallest eigenvalue ensures that the original Hessian satisfies the bound eventually, so the above modification becomes inactive.

\begin{assumption}[Uniform controllability] \label{assumption:uniform-controllability}
For any stage $k \in [N-1]$ and integer $t \in \{1, \dots, N-k\}$, define the controllability matrix by
\begin{align*}
&\Xi_{k,t}(\bm{z}_{k:k+t-1})
\\ &\coloneqq \bigl[
B_{k+t-1},\;
A_{k+t-1}B_{k+t-2},\;\ldots,\;
A_{k+t-1}\cdots A_{k+1}B_k
\bigr]
\\ &\qquad \in \mathbb{R}^{n_x \times t n_u}.
\end{align*}
For $t=1$, this definition reduces to $\Xi_{k,1}=B_k$. For each SQP iteration $\tau \geq 0$, we assume that there exist constants $\gamma_C \in (0,1]$ and $t_C \in \{1,\dots,N\}$, independent of $\tau$, such that for every $k \in [N-t_C]$, there exists an integer $t_k^\tau \in [1,t_C]$ satisfying
\begin{align*}
\Xi_{k,t_k^\tau}^\tau (\Xi_{k,t_k^\tau}^\tau)^T &\succeq \gamma_C \mathbb{I},
\end{align*}
where
\begin{align*}
\Xi_{k,t_k^\tau}^\tau
&\coloneqq
\Xi_{k,t_k^\tau}(\bm{z}_{k:k+t_k^\tau-1}^\tau).
\end{align*}
\end{assumption}

\begin{assumption}[Uniform boundedness] \label{assumption:uniform-boundedness}
For each SQP iteration $\tau \geq 0$, there exists a constant $\Upsilon_{\rm Upper} \geq 2$, independent of $\tau$, such that
\begin{align*}
\max \Bigl\{
\|\hat{H}_k^\tau\|,\,
\|A_k^\tau\|,\,
\|B_k^\tau\|
\Bigr\}
\leq&\, 
\Upsilon_{\rm Upper}\,, \quad k \in [N-1] \\ 
\|\hat H_N^\tau\|\leq&\,  \Upsilon_{\rm Upper}. 
\end{align*}
\end{assumption}
The normalizations $\gamma_C \leq 1$ and $\Upsilon_{\rm Upper} \geq 2$ are without loss of generality, since any valid constants may be replaced by $\min\{\gamma_C,1\}$ and $\max\{\Upsilon_{\rm Upper},2\}$ in the theoretical results that follow. 
\begin{assumption}[Compactness]\label{assumption:compactness}
There exists a compact set $\mathcal Z \times \Lambda$, where $\mathcal Z = \mathcal Z_0 \times \cdots \times \mathcal Z_N$ and $\Lambda = \Lambda_0 \times \cdots \times \Lambda_N$, such that for every iteration $\tau \ge 0$ and every $\alpha \in [0,1]$,
\begin{align*}
(\bm z^\tau + \alpha \Delta \tilde{\bm z}^{\,\tau},\;
\bm  \lambda^\tau + \alpha \Delta \tilde{\bm \lambda}^{\,\tau})
\in \mathcal Z \times \Lambda.
\end{align*}
Moreover, we assume that the functions $\{g_k,g_N,f_k\}$ are three times continuously differentiable and, for every $k\in[N-1]$, satisfy
\begin{align*}
\sup_{\mathcal Z_k \times \Lambda_{k+1}} \|H_k(\bm z_k, \bm \lambda_{k+1})\| \leq&\, \Upsilon_{\rm Upper},
\\ \max\Bigl\{ \sup_{\mathcal Z_k} \|A_k(\bm z_k)\|,\;
\sup_{\mathcal Z_k} \|B_k(\bm z_k)\|
\Bigr\} \le&\, \Upsilon_{\rm Upper}, \\ \nonumber 
\sup_{\mathcal Z_N}\|H_N(\bm z_N)\|\leq&\, \Upsilon_{\rm Upper},
\end{align*}
where, with a slight abuse of notation, $\Upsilon_{\rm Upper} \geq 2$ as in Assumption~\ref{assumption:uniform-boundedness}.
\end{assumption}

We briefly discuss the implications of Assumptions~\ref{assumption:uniform-lower-bound-hessian}--\ref{assumption:compactness}. Assumption~\ref{assumption:uniform-lower-bound-hessian} ensures that the global KKT matrix is invertible, as is standard in SQP analysis (see \cite[Lemma 16.1]{nocedalAndWright}). Assumption~\ref{assumption:uniform-controllability} ensures that the linearized dynamics are controllable within at most $t_C$ stages (see \cite{Keerthi1988,9840913,9954905}). Assumption~\ref{assumption:uniform-boundedness} is needed to bound the global KKT matrix in the analysis. Assumption~\ref{assumption:compactness} provides the compactness and differentiability needed for the augmented Lagrangian merit-function analysis and is standard in SQP designs (e.g., see \cite{bertsekasBook,Na2023,10.1007/s10107-022-01846-z}).

The structural matrix and sensitivity constants below depend only on the displayed assumption-level quantities, compactness, and the user-specified algorithm parameters. These constants are independent of the subproblem index, overlaps, and iteration, and any dependence on $M$ is displayed explicitly. 

Under Assumptions~\ref{assumption:uniform-lower-bound-hessian}--\ref{assumption:uniform-boundedness}, the global SQP subproblem is uniformly well posed, and for sufficiently large $\mu$, the local linear-quadratic subproblems $\{\mathcal{LP}_\mu^i(\bm d_i)\}_{i=0}^{M-1}$ are also uniformly well posed, independently of the SQP iteration and subproblem index. In particular, there is a uniform lower bound on $G^\tau(G^\tau)^T$, and the reduced Hessian of all the subproblems $\{\mathcal{LP}_\mu^i(\bm{d}_i)\}_{i \in [M-1]}$ is lower bounded when $\mu$ is large enough \cite[Lemma 4.1, 4.2]{naFOTD}. Hence, all the subproblems $\mathcal{LP}_\mu^i(\bm{d}_i)$ have unique global solutions, and we additionally have that they satisfy Assumptions~\ref{assumption:uniform-lower-bound-hessian}--\ref{assumption:uniform-boundedness} for any $\bm{d}_i$ with the appropriate change of constants as in \cite[Corollary 4.1]{naFOTD}. Beyond this, Assumption \ref{assumption:compactness} along with Assumptions \ref{assumption:uniform-lower-bound-hessian}--\ref{assumption:uniform-boundedness} guarantees the inverse of the \emph{global} KKT matrix is upper bounded by a constant independent of iteration $\tau$ \cite[Lemma 5.1]{naFOTD}. We will require this to relate the inexact local KKT solutions to the global iterate approximation error. 

\subsection{Global Residual Error and Well-Posedness of the Adaptive Accuracy Conditions}
We begin by first characterizing the error introduced globally by the local approximations.

\begin{lemma}[Global direction-error bound] \label{lem:global-residual-error}
Suppose Assumptions~\ref{assumption:uniform-lower-bound-hessian}--\ref{assumption:compactness} hold, and let $(\Delta \bm z^\tau,\Delta \bm \lambda^\tau)$ denote the exact solution of the global Newton system \eqref{eq:newtonSystem}. Let $(\Delta\tilde{\bm z}^\tau,\Delta\tilde{\bm\lambda}^\tau)$ be the composed direction generated by Algorithm~\ref{alg:aotd-compact}, with every local solve satisfying \eqref{eq:local-residual-cond}. Define
\begin{align*}
\bar{\mu}(\gamma_C, t_C, \Upsilon_{\rm Upper}) := \frac{32 \Upsilon_{\rm Upper}^{4t_C+1}}{\gamma_C} > 0\,.
\end{align*}
For $\mu\geq\bar\mu$, there exist constants $C_1,C_2>0$ and $\rho\in(0,1)$, independent of $i$ and $\tau$, such that
\begin{align}
&\left\|\begin{pmatrix}
\Delta\tilde{\bm z}^\tau-\Delta\bm z^\tau\\
\Delta\tilde{\bm\lambda}^\tau-\Delta\bm\lambda^\tau
\end{pmatrix}\right\|\nonumber\\
& \;\;\leq\!\left[C_1\left(\sum_{i=0}^{M-1}\rho^{2b_i^\tau}\right)^{\frac{1}{2}}
+C_2\left(\sum_{i=0}^{M-1}(\epsilon_i^\tau)^2\right)^{\frac{1}{2}}\right] \left\|\begin{pmatrix}
\Delta\bm z^\tau\\\Delta\bm\lambda^\tau
\end{pmatrix}\right\|,\label{eq:global-inexact-error-bound}
\end{align}
where $b_i^\tau$ is the nominal overlap, and its boundary contribution is omitted whenever subproblem $i$ covers the full horizon.
\end{lemma}

Notice that the bound \eqref{eq:global-inexact-error-bound} of Lemma \ref{lem:global-residual-error} controls the error on the global solution after composing the local solutions, demonstrating that \emph{each subproblem contributes individually} to the global error: \emph{linearly} in its tolerance $\epsilon_i^\tau$ and with \emph{exponential decay} in its own overlap size $b_i^\tau$. The update rules mirror these two functional dependences, with additive overlap growth targeting the exponentially decaying term and multiplicative tolerance contraction targeting the linear term. 

We next establish finite termination of the adaptive accuracy test. In particular, the residual-driven local updates produce a concatenated direction satisfying the accuracy condition \eqref{eq:globalResidualBound} in finitely many inner iterations, after which the descent condition \eqref{eq:descent-direction-condition} is attainable through updates of the penalty parameters.
\begin{lemma}[Finite termination of the accuracy loop]\label{lem:global-adaptive-well-posed}
Let Assumptions \ref{assumption:uniform-lower-bound-hessian}--\ref{assumption:compactness} hold and let $\mu \geq \bar{\mu}$. Fix any outer iteration $\tau$ and hold the merit parameters $(\eta_1^\tau, \eta_2^\tau, \epsilon^\tau)$, with $\epsilon^\tau > 0$, constant. Then, starting from any local parameters $\{\epsilon_i^\tau\}_{i=0}^{M-1}$, $\{b_i^\tau\}_{i=0}^{M-1}$, the residual-driven updates \eqref{eq:subproblem-error-update}--\eqref{eq:subproblem-overlap-update} attain the global accuracy condition \eqref{eq:globalResidualBound} after finitely many inner iterations of the accuracy loop of Algorithm~\ref{alg:aotd-compact}.
\end{lemma}

\begin{lemma}[Sufficient condition for a descent direction]
\label{lem:wellposed-global-descent}
Suppose Assumptions~\ref{assumption:uniform-lower-bound-hessian}--\ref{assumption:compactness} hold, and let $(\Delta \tilde{\bm z},\Delta \tilde{\bm \lambda})$ satisfy the global residual condition \eqref{eq:globalResidualBound}. There exists $\varpi \geq 1$, independent of $\tau$, such that the global descent condition \eqref{eq:descent-direction-condition} holds whenever
\begin{align}
\eta_1^\tau \eta_2^\tau \geq \varpi,
\qquad
\max\{\eta_2^\tau,\epsilon^\tau\eta_1^\tau/\eta_2^\tau\}\le 1/\varpi\,.
\label{eq:wellposed-global-descent-parameter-condition}
\end{align}
\end{lemma}

Consequently, from Lemma~\ref{lem:wellposed-global-descent}, under the adaptive update rule \eqref{eq:adaptive-parameter-update}--\eqref{eq:adaptive-epsilon-update}, the global descent direction condition is eventually satisfied. Hence, the scheme is well-posed as summarized in the corollary below:

\begin{corollary}\label{corr:achievability}
Under Assumptions \ref{assumption:uniform-lower-bound-hessian}--\ref{assumption:compactness} with $\mu \geq \bar{\mu}$, the adaptive update rules in Algorithm \ref{alg:aotd-compact} admit global parameters $\epsilon^\tau, \eta_1^\tau, \eta_2^\tau$ and local parameters  $\{\epsilon_i^\tau\}_{i=0}^{M-1}$, $\{b_i^\tau\}_{i=0}^{M-1}$ for which both \eqref{eq:globalResidualBound} and \eqref{eq:descent-direction-condition} hold simultaneously in finitely many iterations.
\end{corollary}

\subsection{Global Convergence of the Algorithm}
We are now ready to present our main result of global convergence. To establish this, we first need to ensure that all the adaptive parameters stabilize.

\begin{lemma}[Stabilization of the adaptive parameters]\label{lem:stability}
Suppose Assumptions \ref{assumption:uniform-lower-bound-hessian}--\ref{assumption:compactness} hold with $\mu \geq \bar{\mu}$. Then there exists a finite threshold ${\bar{\tau}} \geq 0$ such that, for all $\tau \geq {\bar{\tau}}$ and every $i \in [M-1]$, the adaptive parameters of Algorithm~\ref{alg:aotd-compact} freeze,
\begin{align*}
(\eta_1^\tau, \eta_2^\tau, \epsilon^\tau, \epsilon_i^\tau, b_i^\tau)
= (\eta_1^{\bar{\tau}}, \eta_2^{\bar{\tau}}, \epsilon^{\bar{\tau}}, \epsilon_i^{\bar{\tau}}, b_i^{\bar{\tau}})\,,
\end{align*}
the descent direction condition \eqref{eq:descent-direction-condition} holds, and 
\begin{align}  \label{eq:stability-local-bounds}
\epsilon^{\bar{\tau}}\in(0,1)\,, \qquad \epsilon_i^{\bar{\tau}}>0\,.
\end{align}
\end{lemma}
Lemma \ref{lem:stability} guarantees, in finite iterations, that all constants stabilize and, moreover, that the inexactness parameters $\epsilon, \epsilon_i$ are bounded away from $0$. The nominal overlaps therefore freeze at finite values, while the effective overlaps are bounded by the horizon through the clipping in \eqref{eq:effective-overlaps}. Since all adaptive parameters stabilize, the iteration eventually reduces to a fixed-parameter scheme with a constant exact-augmented-Lagrangian merit function. Hence, the uniform descent enforced by the descent-direction condition \eqref{eq:descent-direction-condition} yields the following global convergence guarantee for the KKT residual.

\begin{theorem}\label{thm:global}
Let Assumptions \ref{assumption:uniform-lower-bound-hessian}--\ref{assumption:compactness} hold and let $\mu \geq \bar{\mu}$. Then, the iterates generated by Algorithm \ref{alg:aotd-compact} satisfy the following properties:
\begin{enumerate}
\item For every $\tau \geq 0$, there exists a step size $\alpha^\tau \in (0, 1]$ satisfying the Armijo condition in \eqref{eq:line-search-cond}.
\item The KKT residual satisfies $\|\nabla \mathcal{L}^\tau\| \to 0$ as $\tau \to \infty$. 
\end{enumerate}
\end{theorem}

Theorem~\ref{thm:global} shows that the KKT residual vanishes from any initialization. Because the iterates remain in a compact set, accumulation points exist, and continuity of the KKT residual ensures that every accumulation point satisfies the KKT conditions for \eqref{eq:main-problem}. As expected, the result is first-order and does not by itself certify local optimality or a rate of convergence. To characterize this, the next section shows that once the adaptive parameters stabilize (Lemma~\ref{lem:stability}), the iterates admit a unit step size and contract linearly toward a local minimizer.

\begin{figure*}[t!]
\centering
\includegraphics[width=0.9\linewidth]{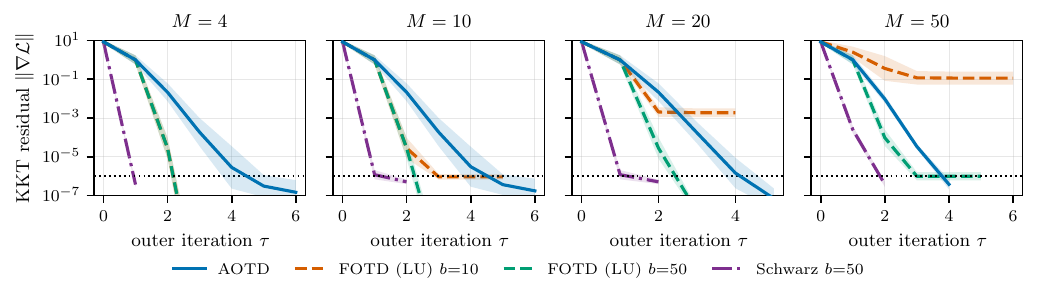}
\vspace{-1em}
\caption{KKT convergence of various algorithms on the NE39 swing OCP with $N=1000$. Lines and lighter shading show means and standard deviations over $5$ seeds. The dotted line indicates the stopping criterion tolerance of $10^{-6}$.}
\label{fig:freq_control_convergence}
\end{figure*}



\section{Local Linear Convergence} \label{sec:local-convergence}

To set the stage for the local convergence analysis, we make the additional standing assumption throughout this section that the full primal--dual sequence $\{(\bm z^\tau,\bm\lambda^\tau)\}$ converges to a KKT pair $(\bm z^*,\bm\lambda^*)$, where $\bm z^*$ is a strict local minimizer of \eqref{eq:main-problem}. This assumption is stronger than Theorem~\ref{thm:global}, which guarantees only that the KKT residual converges to zero.

\newpage 
\begin{assumption}\label{assumption:modified-hessian-vanishes}
Assume that $\|H^\tau - \hat{H}^\tau\| = o(1)$ where $H^\tau = \nabla^2_{\bm{z}}\mathcal{L}$ is the Lagrangian Hessian and $\hat{H}^\tau$ is its modification.
\end{assumption}

\begin{assumption}[Local Lipschitz continuity]
\label{assumption:local-lipschitz}
There exist an open neighborhood $\mathcal U$ of $(\bm z^\ast,\bm\lambda^\ast)$ and a constant $\Upsilon_L>0$, independent of the stage $k$, such that for any $(\bm z,\bm\lambda),(\bm z',\bm\lambda')\in\mathcal U$ and every $k\in[N-1]$,
\begin{align*}
&\max\bigl\{
\|A_k(\bm z_k)-A_k(\bm z_k')\|,\,
\|B_k(\bm z_k)-B_k(\bm z_k')\|
\bigr\}\\
&\qquad \qquad \leq
\Upsilon_L\|\bm z_k-\bm z_k'\|,\\
&\|H_k(\bm z_k,\bm\lambda_{k+1})
-H_k(\bm z_k',\bm\lambda_{k+1}')\|\\
&\qquad \qquad\leq
\Upsilon_L
\|(\bm z_k-\bm z_k';\,
\bm\lambda_{k+1}-\bm\lambda_{k+1}')\|,\\
&\|H_N(\bm z_N)-H_N(\bm z_N')\|\\
&\qquad \qquad\leq
\Upsilon_L\|\bm z_N-\bm z_N'\|.
\end{align*}
\end{assumption}

Assumptions~\ref{assumption:modified-hessian-vanishes} and \ref{assumption:local-lipschitz} are standard in local SQP analysis \cite{Boggs_Tolle_1995,nocedalAndWright}. The vanishing Hessian approximation error is needed for the eventual unit-step result in Lemma~\ref{lem:unit-step-size}, while local Lipschitz continuity controls the Taylor remainders used in that lemma and Theorem~\ref{thm:local-linear}.

Under Assumptions \ref{assumption:modified-hessian-vanishes} and \ref{assumption:local-lipschitz}, the iterates generated by Algorithm \ref{alg:aotd-compact} eventually admit a unit step size $\alpha^\tau = 1$ in the line search \eqref{eq:line-search-cond}. This is the central justification for the design of the global residual condition \eqref{eq:globalResidualBound}--\eqref{eq:globalEpsilonMin}: the threshold $\epsilon^\tau_{\rm max}$ is chosen precisely so that the inexact direction remains close enough to the exact Newton direction for the second-order Taylor remainder in the Armijo condition to be dominated by the descent term, allowing $\alpha^\tau = 1$ to satisfy the line search.

\begin{lemma}\label{lem:unit-step-size}
Let Assumptions \ref{assumption:uniform-lower-bound-hessian}--\ref{assumption:local-lipschitz} hold with $\mu \geq \bar{\mu}$. Then, the iterates given by Algorithm \ref{alg:aotd-compact} with directions $\{(\Delta \tilde{\bm{z}}^\tau, \Delta \tilde{\bm{\lambda}}^\tau)\}_\tau$ yield $\alpha^\tau = 1$ for sufficiently large $\tau$. 
\end{lemma}

With the unit step size established, we now characterize the local rate of contraction. By Lemma~\ref{lem:stability}, once the parameters have settled the accuracy tolerance is frozen at $\epsilon^\tau = \epsilon^{\bar\tau} < 1$ for all $\tau \geq \bar\tau$. This $\epsilon^\tau$ directly gives the contraction factor for the rate.

\begin{theorem}[Local linear convergence]\label{thm:local-linear}
Suppose Assumptions~\ref{assumption:uniform-lower-bound-hessian}--\ref{assumption:local-lipschitz} hold with $\mu \geq \bar\mu$, and let $\epsilon^{\bar\tau} < 1$ be the frozen accuracy tolerance of Lemma~\ref{lem:stability}. Then, for every $\varphi \in (0,\, 1/\epsilon^{\bar\tau} - 1)$ and all sufficiently large $\tau$,
\begin{align*}
&\big\|(\bm{z}^{\tau+1} - \bm{z}^\ast,\, \bm{\lambda}^{\tau+1} - \bm{\lambda}^\ast)\big\|
\\ &\qquad \leq (1 + \varphi)\,\epsilon^{\bar\tau}\,
\big\|(\bm{z}^{\tau} - \bm{z}^\ast,\, \bm{\lambda}^{\tau} - \bm{\lambda}^\ast)\big\|\,,
\end{align*}
so the iterates converge linearly with rate $(1+\varphi)\epsilon^{\bar\tau} < 1$.
\end{theorem}

In contrast to \cite[Theorem 6.3]{naFOTD}, which assumes exact solutions of all local KKT systems and establishes uniform stagewise linear convergence, Theorem~\ref{thm:local-linear} establishes linear convergence in the $\ell_2$ norm of the full primal--dual iterate $(\bm z^\tau,\bm\lambda^\tau)$. This is due to the fact that when one solves each local KKT linear system exactly, its only error source is the boundary perturbation, which has a stagewise sensitivity structure. In contrast, AOTD additionally introduces inexact local KKT solves controlled by aggregate per-subproblem tolerances, not stagewise ones. Hence, one could adapt Theorem \ref{thm:local-linear} to obtain uniform stagewise convergence, but this would still depend on the subproblem values $b_i^\tau$, $\epsilon_i^\tau$ for each $i$.

\section{Numerical Illustrations} \label{sec:simulations}

We evaluate AOTD on two challenging nonlinear OCPs to examine the reliability and computational benefit of adapting overlap and local solve accuracy. All runs target a KKT residual threshold $\|\nabla\mathcal L\|\leq10^{-6}$ for convergence over five perturbed initial conditions. We compare with IPOPT, multiple shooting~\cite{BOCK19841603}, ADMM~\cite{8186925}, Schwarz~\cite{9840913}, and FOTD~\cite{naFOTD}. Moreover, in the implementation of AOTD, we set the factors $\Psi^\tau=1$ and $\Upsilon^\tau=1$ for all $\tau$ as a practical estimate and find the scheme converges without issue. The code is publicly available in~\cite{bhan_repo}. 

\subsection{Nonlinear Frequency Control for Stability of Power Grids}
\label{subsec:swing_numerical}

\begin{figure*}[t]
\centering
\includegraphics[width=0.85\linewidth]{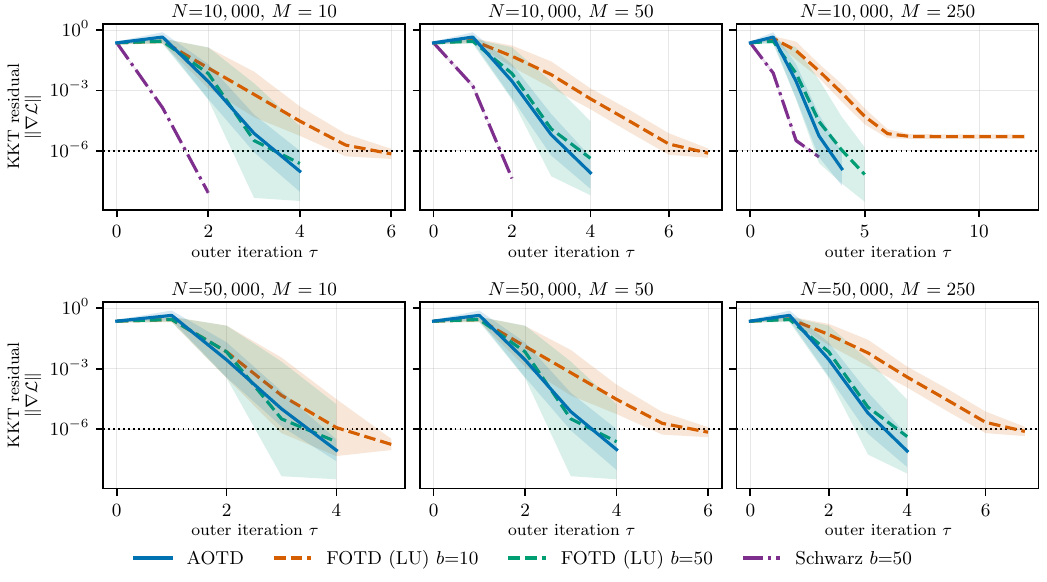}
\caption{KKT convergence on the Burgers OCP across subproblem counts $M$ and horizons $N$ (top: $N{=}10{,}000$, bottom: $N{=}50{,}000$). The solid line shows the mean over $5$ seeds and the dotted line marks the $10^{-6}$ tolerance. }
\label{fig:burgers_convergence}
\end{figure*}
We consider frequency regulation on the nonlinear Kron-reduced New England 39-bus (NE39) network with $n_g=10$ generators~\cite{8619580,10543148,10820007}. We refer to this benchmark as the NE39 swing OCP. Its swing dynamics are
\begin{align*}
M_i^{ \rm swng}\ddot\theta_i+D_i^{\rm swng}\dot\theta_i
=p_{\mathrm{in},i}-\sum_{j\in\mathcal N_i}b_{ij}^{\rm swng}\sin(\theta_i-\theta_j),
\end{align*}
where $\theta_i$ is the phase angle, $M_i^{\rm swng}$ the inertia, $D_i^{\rm swng}$ the damping, $p_{\mathrm{in},i}$ the controlled power injection, and $b_{ij}^{\rm swng}$ the susceptance to neighbor $j\in\mathcal N_i$. With frequency deviation $\dot\theta_i$, the state is $\bm{x}=(\{\theta\}_i; \{\dot{\theta}\}_i)$, and the input is $\bm u=\{p_{\mathrm{in, i}}\}_i$. Forward Euler gives
\begin{align}
\min_{\{\bm x_k,\bm u_k\}}\quad
&\sum_{k=0}^{N-1}\bigl(\bm x_k^\top Q_k^{\rm cost}\bm x_k+\bm u_k^\top R_k^{\rm cost}\bm u_k\bigr)\nonumber\\
&\quad+\bm x_N^\top Q_N^{\rm cost}\bm x_N, \notag\\
\text{s.t.}\quad &\bm\theta_{k+1}=\bm\theta_k+\Delta t\,\dot{\bm\theta}_k, \notag\\
&\dot{\bm\theta}_{k+1}=\dot{\bm\theta}_k+\Delta t (M^{\rm swng})^{-1}\bigl(\bm u_k \nonumber \\ & \qquad\qquad \qquad-D^{\rm swng}\dot{\bm\theta}_k-\bm h(\bm\theta_k)\bigr), \notag\\
&\bm x_0=(\bar{\bm\theta}_0;\dot{\bar{\bm\theta}}_0)\,, 
\notag 
\end{align}
where $M^{\rm swng}=\diag(M_i^{\rm swng})$, $D^{\rm swng}=\diag(D_i^{\rm swng})$, and $\bm h(\bm\theta)$ has components $h_i(\bm\theta)=\sum_{j\in\mathcal N_i}b_{ij}^{\rm swng}\sin(\theta_i-\theta_j)$. For $k=0,\ldots,N-1$, the objective weights are $Q_k^{\rm cost}=10I_{20}$ and
$R_k^{\rm cost}=0.05\diag(0.94,1.15,0.55,0.71,\allowbreak 1.27,1.03,1.29,1.02,0.99,0.73)$, with the diagonal coefficients rounded to two decimal places here and given exactly in Table~\ref{tab:experiment-config}. Both weights are constant across time stages, and the terminal weight is $Q_N^{\rm cost}=10I_{20}$. The horizon is $N=1000$, and $\Delta t=0.1$, giving $30{,}020$ primal variables.

\begin{figure*}[t]
\centering
\includegraphics[width=0.85\linewidth]{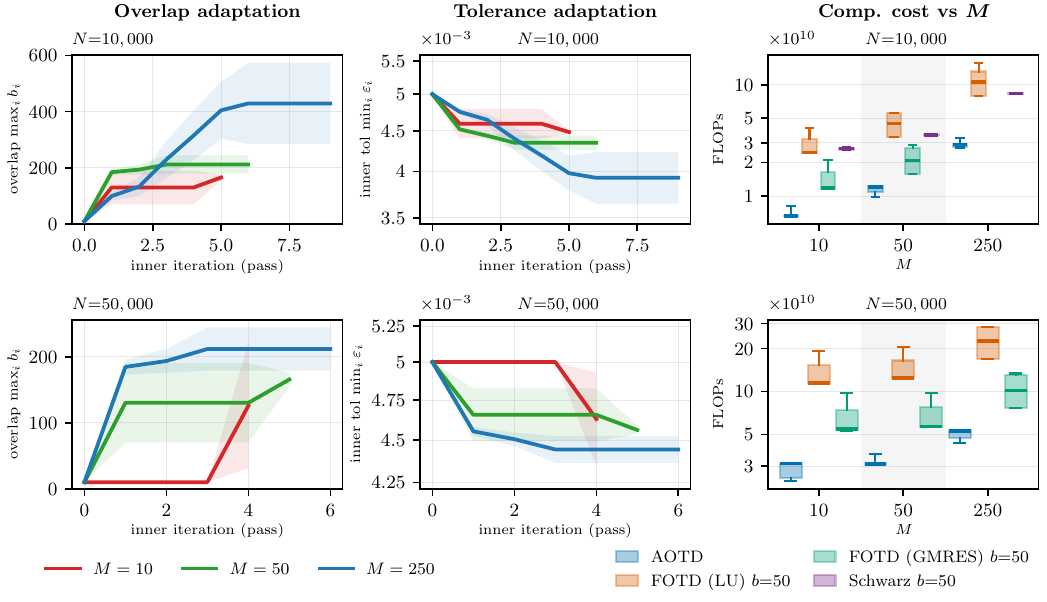}
\caption{Burgers adaptation and cost for $N=10{,}000$ (top) and $N=50{,}000$ (bottom). Left: maximum nominal overlap $\max_i b_i$ over accuracy-loop passes. Center: minimum local tolerance $\min_i\epsilon_i$. Right: estimated FLOPs. AOTD uses deterministic GMRES. At $N=50{,}000$, Schwarz exhausts memory and is omitted.}
\label{fig:burgers_adaptation}
\end{figure*}

Figure~\ref{fig:freq_control_convergence} shows AOTD reaching tolerance across all partitions, whereas fixed-overlap methods depend strongly on the chosen overlap. FOTD with $b=50$ succeeds at $M=4,10,20$ but fails at $M=50$, where increasing to $b=75$ restores convergence. Moreover, we find small-overlap Schwarz, multiple shooting, and ADMM also fail as the partition is refined. Thus, in fixed schemes, the user must experiment with repeated tuning to find the correct overlap for convergence, or choose a sufficiently large overlap that is conservative. In contrast, AOTD selects sufficient overlaps during the solve and converges from the same initialization throughout. Lastly, from Figure \ref{fig:freq_control_convergence} the residual contraction is consistent with the local linear convergence analysis corroborating the theoretical results. 

\begin{table}[H]\centering\tablesize\setlength{\tabcolsep}{3pt}
\fontsize{6.5}{6}\selectfont
\setlength{\aboverulesep}{1pt}\setlength{\belowrulesep}{1pt}
\caption{Estimated FLOPs ($\times 10^{7}$) on the NE39 swing OCP ($N{=}1000$, mean$\pm$stdev over 5 seeds). $M$ is the number of temporal subproblems and dashes denotes non-convergence. Lowest mean estimated FLOPs in each column are bold.}
\label{tab:highlight_swing}
\begin{tabular}{lc|rrrr}
\toprule
Method & $b$ & $M{=}4$ & $M{=}10$ & $M{=}20$ & $M{=}50$ \\
\midrule
IPOPT & --- & \multicolumn{4}{c}{$75.1{\pm}0.0$} \\
\midrule
MultiShoot & --- & $18.9{\pm}0.0$ & --- & --- & --- \\
ADMM ($\chi{=}10$) & --- & $18.8{\pm}0.0$ & $9.6{\pm}4.0$ & --- & --- \\
Schwarz & $10$ & $19.6{\pm}0.0$ & $11.8{\pm}2.8$ & --- & --- \\
& $50$ & $22.6{\pm}0.0$ & $16.6{\pm}4.3$ & $24.9{\pm}4.9$ & $36.3{\pm}0.0$ \\
& $75$ & $24.4{\pm}0.0$ & $19.6{\pm}5.2$ & $33.2{\pm}8.6$ & $67.4{\pm}0.0$ \\
FOTD (LU) & $10$ & $80.0{\pm}0.0$ & --- & --- & --- \\
& $50$ & $98.0{\pm}0.0$ & $143.5{\pm}0.0$ & $219.3{\pm}0.0$ & --- \\
& $75$ & $109.3{\pm}0.0$ & $177.4{\pm}0.0$ & $287.0{\pm}0.0$ & $615.2{\pm}0.0$ \\
FOTD (GMRES) & $10$ & $36.5{\pm}0.6$ & --- & --- & --- \\
& $50$ & $44.6{\pm}0.6$ & $26.0{\pm}0.3$ & $31.9{\pm}0.5$ & --- \\
& $75$ & $49.6{\pm}0.7$ & $32.3{\pm}0.8$ & $52.6{\pm}0.9$ & $91.9{\pm}1.1$ \\
\midrule
AOTD (GMRES) & --- & {\boldmath $10.0{\pm}1.7$} & {\boldmath $7.1{\pm}1.2$} & {\boldmath $9.2{\pm}1.3$} & {\boldmath $33.7{\pm}0.4$} \\
AOTD (sGMRES) & --- & $13.7{\pm}2.3$ & $8.4{\pm}1.1$ & $10.9{\pm}1.7$ & $36.9{\pm}3.7$ \\
\bottomrule
\end{tabular}
\end{table}

Table~\ref{tab:highlight_swing} makes this convergence trade-off precise. We see that AOTD (GMRES) has the lowest mean estimated FLOPs in every column, while increasing fixed overlaps in FOTD and Schwarz to ensure convergence can be expensive for various subproblem decompositions. For example, at $M=50$, FOTD (GMRES) requires an estimated $91.9\times10^7$ FLOPs at $b=75$, compared with $33.7\times10^7$ for AOTD. Lastly, both GMRES and a sketched form of GMRES (sGMRES) converge with similar estimated FLOP counts throughout. This illustrates that the acceptance test accommodates different local solvers. We now move to a larger scale problem in the Burgers PDE.

\subsection{Optimal Control of a Burgers PDE}\label{subsec:burgers_numerical}

We next consider regulation of the viscous Burgers equation~\cite{HinzePinnauUlbrichUlbrich2009,doi:10.1137/23M1566935,hao2023bilevel}, with state $y$, distributed control $w$, viscosity $\kappa>0$, control weight $\psi>0$, and target $y_{\rm des}=0$:
\begin{align*}
\min_{y,w}\quad
&\frac12\int_0^T\!\bigl(\|y(\cdot,t)-y_{\rm des}\|_{L^2}^2
+\psi\|w(\cdot,t)\|_{L^2}^2\bigr)\,dt \\
&\quad+\frac12\|y(\cdot,T)-y_{\rm des}\|_{L^2}^2, \\
\text{s.t.}\quad
&\partial_t y+y\partial_\xi y=\kappa\partial_{\xi\xi}y+w, \\
&y(0,t)=y(1,t)=0,\qquad y(\xi,0)=y_0(\xi).
\end{align*}
The dynamics hold on $(\xi,t)\in(0,1)\times(0,T]$. We discretize the dynamics using IMEX Euler, treating diffusion implicitly and convection and control explicitly~\cite{ascher1995imex}. The discretization has $n_x=n_u=24$ and $N\in\{10{,}000,50{,}000\}$. Hence, this problem has $2{,}400{,}024$ (at $N=50{,}000$) primal variables.

Figure~\ref{fig:burgers_convergence} shows AOTD reaching tolerance at all horizons and decomposition sizes. Meanwhile, at $N=10{,}000$, $M=250$, FOTD (LU) fails with $b=10$ and succeeds with $b=50$ and at $N=50{,}000$, $M=250$, FOTD (LU) succeeds with $b=10$, but FOTD (GMRES) fails all seeds. Moreover, for $N=50{,}000$, we find that both the Schwarz scheme and IPOPT exhaust the compute resources and only report results at $N=10{,}000$. However, at this horizon, IPOPT requires an estimated $44.8{\times}10^{9}$ FLOPs, roughly $6.5\times$ the estimated FLOP count of AOTD at $M=10$. Lastly, the multiple shooting fails to converge for both horizons for at least one seed.

\begin{table}[H]\centering\tablesize\setlength{\tabcolsep}{3pt}
\fontsize{6.5}{6}\selectfont
\setlength{\aboverulesep}{1pt}\setlength{\belowrulesep}{1pt}
\caption{Estimated FLOPs ($\times 10^{9}$) on the viscous Burgers OCP (mean$\pm$stdev over 5 seeds). $M$ is the number of temporal subproblems and dashes denote non-convergence. Lowest mean estimated FLOPs in each column are bold. At $N{=}50{,}000$, Schwarz exhausts memory and is omitted.}
\label{tab:burgers_highlight}
\resizebox{\columnwidth}{!}{%
\begin{tabular}{lc|rrr}
\toprule
Method & $b$ & $M{=}10$ & $M{=}50$ & $M{=}250$ \\
\midrule
\multicolumn{5}{l}{\emph{$N = 10{,}000$ \quad (core lengths $1000/200/40$)}} \\
\midrule
Schwarz & $50$ & $26.6{\pm}0.6$ & $35.4{\pm}0.3$ & $83.3{\pm}0.4$ \\
FOTD (LU) & $10$ & $42.9{\pm}6.1$ & $54.6{\pm}6.6$ & -- \\
& $50$ & $29.5{\pm}6.6$ & $44.9{\pm}10.0$ & $110.6{\pm}30.7$ \\
FOTD (GMRES) & $10$ & -- & -- & -- \\
& $50$ & $14.5{\pm}3.8$ & $21.6{\pm}5.5$ & -- \\
AOTD (GMRES) & --- & {\boldmath $6.9{\pm}0.6$} & {\boldmath $11.5{\pm}1.0$} & {\boldmath $29.3{\pm}2.2$} \\
AOTD (sGMRES) & --- & $8.2{\pm}1.1$ & $13.8{\pm}1.2$ & $38.0{\pm}2.6$ \\
\midrule[\heavyrulewidth]
\multicolumn{5}{l}{\emph{$N = 50{,}000$ \quad (core lengths $5000/1000/200$)}} \\
\midrule
FOTD (LU) & $10$ & $165.9{\pm}18.5$ & $214.6{\pm}30.7$ & $273.5{\pm}33.1$ \\
& $50$ & $137.7{\pm}30.6$ & $148.6{\pm}33.0$ & $225.6{\pm}50.4$ \\
FOTD (GMRES) & $10$ & -- & -- & -- \\
& $50$ & $66.8{\pm}17.4$ & $69.0{\pm}16.6$ & $103.8{\pm}24.8$ \\
AOTD (GMRES) & --- & {\boldmath $28.4{\pm}3.5$} & {\boldmath $32.0{\pm}2.2$} & {\boldmath $49.9{\pm}3.9$} \\
AOTD (sGMRES) & --- & $34.2{\pm}4.3$ & $36.1{\pm}4.1$ & $56.2{\pm}4.2$ \\
\bottomrule
\end{tabular}}
\end{table}

Figure~\ref{fig:burgers_adaptation} illustrates the effectiveness of favoring overlap growth over aggressive tolerance reduction in this setting. The left and middle columns illustrate the adaptation process associated with Lemma~\ref{lem:stability}: the maximum nominal overlap and minimum local tolerance level off in several configurations, although some traces continue to change over the displayed accuracy-loop passes. Finally, the right column of Figure~\ref{fig:burgers_adaptation} compares estimated FLOPs across methods, and Table~\ref{tab:burgers_highlight} shows that the best converged fixed-overlap method uses $1.9$--$2.8$ times as many estimated FLOPs as AOTD (GMRES) across the six configurations.

Across both OCPs, the fixed-overlap baselines require overlap tuning, and jointly adapting the overlap and local solve accuracy reduces estimated FLOPs.

\section{Conclusion} \label{sec:conclusions}

We developed AOTD, a parallel inexact-SQP method that jointly adapts  the temporal overlap and local solve accuracy. Under the stated assumptions, its accuracy and descent tests terminate finitely, its adaptive parameters stabilize, and its KKT residual converges to zero (Theorem~\ref{thm:global}). If the primal--dual iterates converge to a strict local KKT solution, we additionally establish eventual unit steps and local linear convergence (Theorem~\ref{thm:local-linear}). On the challenging power-grid and Burgers OCPs, AOTD converged across all tested partitions without prescribing a sufficient final overlap, with the best converged fixed-overlap comparators requiring $1.9$--$2.8$ times as many estimated FLOPs on Burgers.


\bibliography{references}
\bibliographystyle{plain}
\allowdisplaybreaks[0] 

\clearpage
\appendix
\setcounter{equation}{0}
\renewcommand{\theequation}{A.\arabic{equation}}

\section{Proofs in Section \ref{sec:global-convergence}}\label{appendix:proofs1}
\subsection{Proof of Lemma \ref{lem:global-residual-error}}
\begin{proof}
We will suppress the iteration index $\tau$ throughout the proof, retaining it only on the fixed iterate $\bm{x}_0^\tau$. Let $(\Delta \bm{z}, \Delta \bm{\lambda})$ be the exact solution to the global Newton system \eqref{eq:newtonSystem}. For each subproblem $i \in [M-1]$, define $(\Delta \tilde{\bm{\omega}}^\ast_i, \Delta \tilde{\bm{\zeta}}_i^\ast)$ to be the exact solution of the local quadratic subproblem $\mathcal{LP}_{\mu}^i(\bm{d}_i)$ where $\bm{d}_i = (0;0;0;0)$ is the boundary condition (with $\bm{d}_{i,1} = \bar{\bm{x}}_0 - \bm{x}_0^\tau$ for every subproblem with $m_1^i = 0$). Then, let $(\Delta \tilde{\bm{z}}^\ast, \Delta \tilde{\bm{\lambda}}^\ast) =  \mathcal{C}\left(\{(\Delta \tilde{\bm{\omega}}_i^\ast, \Delta \tilde{\bm{\zeta}}_i^\ast )\}_{i=0}^{M-1} \right)$ and similarly, $(\Delta \tilde{\bm{z}}, \Delta \tilde{\bm{\lambda}}) =  \mathcal{C}\left(\{(\Delta \tilde{\bm{\omega}}_i, \Delta \tilde{\bm{\zeta}}_i)\}_{i=0}^{M-1} \right)$ be the concatenations of the exact and inexact local solutions of the subproblems. By the triangle inequality, we have
\begin{align}
 \|(\Delta \tilde{\bm{{z}}} - \Delta \bm{z}, \Delta \tilde{\bm{\lambda}} - \Delta \bm{\lambda})\|  \leq&\; \|(\Delta \tilde{\bm{{z}}} - \Delta \tilde{\bm{z}}^\ast, \Delta \tilde{\bm{\lambda}} - \Delta \tilde{\bm{\lambda}}^\ast )\|  \nonumber   \\
&  + \|(\Delta \tilde{\bm{{z}}}^\ast - \Delta \bm{z}, \Delta \tilde{\bm{\lambda}}^\ast - \Delta \bm{\lambda})\| . \label{eq:lemma-1-triangle-inequality}
\end{align}
These two terms can be handled separately. Since $\mu \geq \bar{\mu}$ every subproblem is uniformly well posed \cite[Lemma 4.2, Corollary 4.1]{naFOTD}, and we begin with the first (inexactness) term, returning to the second (boundary) term at the end using the same sensitivity estimate.

To handle the first term of \eqref{eq:lemma-1-triangle-inequality}, note that by definition, the exact solution satisfies
\begin{align*}
\Gamma_i \begin{pmatrix} \Delta \tilde{\bm{\omega}}_i^\ast \\ \Delta \tilde{\bm{\zeta}}_i^\ast \end{pmatrix} = - \nabla \mathcal{L}_{\rm QP}^{(i)}\,,
\end{align*}
where $\nabla \mathcal{L}_{\rm QP}^{(i)}$ is the fixed local KKT right-hand side of \eqref{eq:newton} (with the iteration index $\tau$ suppressed). Hence, subtracting this identity from the definition of the local residual \eqref{eq:local-residual-def}, evaluated at the inexact solution $(\Delta \tilde{\bm{\omega}}_i, \Delta \tilde{\bm{\zeta}}_i)$, gives
\begin{align}
& \bm{\tilde{r}}_i(\Delta \tilde{\bm{\omega}}_i, \Delta \tilde{\bm{\zeta}}_i) = \Gamma_i\begin{pmatrix}
\Delta \tilde{\bm{\omega}}_i - \Delta \tilde{\bm{\omega}}_i^\ast  \\ \Delta \tilde{\bm{\zeta}}_i - \Delta \tilde{\bm{\zeta}}_i^\ast
\end{pmatrix} \nonumber\,, 
\end{align}
implying
\begin{align}
 \left\| \begin{pmatrix}
\Delta \tilde{\bm{\omega}}_i - \Delta \tilde{\bm{\omega}}_i^\ast  \\ \Delta \tilde{\bm{\zeta}}_i - \Delta \tilde{\bm{\zeta}}_i^\ast
\end{pmatrix} \right\| \leq \|\Gamma_i^{-1}\| \|\bm{\tilde{r}}_i(\Delta \tilde{\bm{\omega}}_i, \Delta \tilde{\bm{\zeta}}_i)\|. \label{eq:lemma-1-local-residual-bound}
\end{align}
By \cite[Lemma 4.2 and Corollary 4.1]{naFOTD}, the local KKT matrices inherit uniform reduced-Hessian positivity, Jacobian conditioning, and matrix-norm bounds from Assumptions \ref{assumption:uniform-lower-bound-hessian}--\ref{assumption:uniform-boundedness}. Hence, applying the explicit inverse bound of \cite[Lemma 5.1]{naFOTD} at the subproblem level yields $\|\Gamma_i^{-1}\| \leq C_{\Gamma^{-1}_{\rm loc}}$ with $C_{\Gamma^{-1}_{\rm loc}}$ independent of $\tau$ and $i$. Combining this with the local residual condition \eqref{eq:local-residual-cond} and \eqref{eq:lemma-1-local-residual-bound} yields
\begin{align*}
\left\| \begin{pmatrix}
\Delta \tilde{\bm{\omega}}_i - \Delta \tilde{\bm{\omega}}_i^\ast  \\ \Delta \tilde{\bm{\zeta}}_i - \Delta \tilde{\bm{\zeta}}_i^\ast 
\end{pmatrix} \right\|  \leq C_{\Gamma^{-1}_{\rm loc}} \epsilon_i \left\| \nabla \mathcal{L}_{\rm QP}^{(i)} \right\|\,.
\end{align*}
Using the definition of $\nabla \mathcal{L}_{\rm QP}^{(i)}$ yields
\begin{align}
\left\| \nabla \mathcal{L}_{\rm QP}^{(i)} \right\| \leq \|\Gamma_i\| \left\|\begin{pmatrix} \Delta \tilde{\bm{\omega}}_i^\ast \\ \Delta \tilde{\bm{\zeta}}_i^\ast \end{pmatrix}\right\|\,.  \label{eq:lemma-1-subproblem-qp-bound}
\end{align}
Since $\mu$ is fixed, the terminal block $\hat{Q}_{m_2}+\mu I$ is uniformly bounded. Moreover, Assumption~\ref{assumption:uniform-boundedness}, together with the fixed block-banded structure of the local Hessian and Jacobian, implies that there exists a constant $C_{\Gamma_{\rm loc}}>0$, independent of $\tau$, $i$, and the window length, such that $\|\Gamma_i\|\leq C_{\Gamma_{\rm loc}}$.

Now, since we use $\bm{d}_i = (0; 0; 0;0)$ on the interface components, there is a local error between this choice and the correct choice corresponding to the solution of the global SQP problem (for any subproblem with $m_1^i = 0$ the initial condition $\bm{d}_{i,1} = \bar{\bm{x}}_0 - \bm{x}_0^\tau$ is already the correct one). Let $\Delta \bm{d}_i = (\Delta \bm{x}_{m_1}; \Delta \bm{x}_{m_2}; \Delta \bm{u}_{m_2}; \Delta \bm{\lambda}_{m_2+1})$ when $m_2 < N$, and $\Delta \bm{d}_i = \Delta \bm{x}_{m_1}$ when $m_2 = N$ (matching the boundary variables of \eqref{eq:decomposed-main-problem}), be the correct boundary condition for the SQP step at $\tau$. Then, the decomposed solution of $(\Delta \bm{\omega}_i, \Delta \bm{\zeta}_i) = \mathcal{D}_i(\Delta \bm{z}, \Delta \bm{\lambda})$ is exactly that of the local subproblem solved exactly with  $\Delta \bm{d}_i$. We can apply the sensitivity estimate of \cite[Theorem 4.1]{naFOTD}, whose constants $C_{d} > 0$ and $\rho_d \in (0, 1)$ are built from the uniform truncated-problem constants above via the exponential decay of sensitivity \cite{doi:10.1137/19M1265065}, and are hence independent of $i$ and $\tau$. Then, for every $k \in [m_1, m_2]$,
\begin{align}
\max\{\|\Delta \tilde{\bm{\omega}}_{i, k}^\ast &- \Delta \bm{\omega}_{i, k} \|\,, \|\Delta \tilde{\bm{\zeta}}_{i, k}^\ast - \Delta \bm{\zeta}_{i, k} \| \}   \nonumber\\
   \leq&\; C_d \bigg(\rho_d^{k-m_1} \|\Delta \bm{x}_{m_1}\| \nonumber \\ &+ \rho_d^{m_2-k} \bigg\| ( \Delta \bm{x}_{m_2}; \Delta \bm{u}_{m_2}; \Delta \bm{\lambda}_{m_2+1}) \bigg\| \bigg) \,, \label{eq:lemma-1-estimate}     
\end{align}
where the left (respectively right) boundary term is omitted when $m_1 = 0$ (respectively $m_2 = N$) since for $m_1 = 0$ the correcting boundary $\bm{d}_{i,1} = \bar{\bm{x}}_0 - \bm{x}_0^\tau$ makes the initial boundary exact, and for $m_2 = N$ the true terminal cost $g_N$ carries no boundary perturbation. Because the sensitivity constants depend only on the uniform stagewise constants, the cited estimate applies uniformly to every clipped window generated by the adaptive overlap rule.

Now, noting that for all $i$, the right hand vectors are subvectors of the full Newton step. Here, the two boundary blocks $\Delta\bm{x}_{m_1}$ and $(\Delta\bm{x}_{m_2};\Delta\bm{u}_{m_2};\Delta\bm{\lambda}_{m_2+1})$ are disjoint subvectors of the full Newton step, so $\|\Delta\bm{x}_{m_1}\|^2 + \|(\Delta\bm{x}_{m_2};\Delta\bm{u}_{m_2};\Delta\bm{\lambda}_{m_2+1})\|^2 \leq \|(\Delta\bm{z},\Delta\bm{\lambda})\|^2$. Hence, squaring \eqref{eq:lemma-1-estimate}, applying $(a+b)^2 \leq 2a^2 + 2b^2$, summing the two components $(\bm{\omega},\bm{\zeta})$ and the resulting geometric series, and taking the square root, we obtain
\begin{align}
\|(\Delta \tilde{\bm{\omega}}_i^\ast, \Delta \tilde{\bm{\zeta}}_i^\ast ) - (\Delta \bm{\omega}_i , \Delta \bm{\zeta}_i)\| \leq \frac{2C_d}{\sqrt{1-\rho_d^2}} \|(\Delta \bm{z}, \Delta \bm{\lambda})\|\,.  \label{eq:lemma-1-sensitive-bound}
\end{align}
Using the fact that the restricted global direction is bounded by the full direction, $\|(\Delta \bm{\omega}_i, \Delta \bm{\zeta}_i)\| \leq \|(\Delta \bm{z}, \Delta \bm{\lambda})\|$, along with \eqref{eq:lemma-1-subproblem-qp-bound}, \eqref{eq:lemma-1-sensitive-bound}, we obtain
\begin{align*}
\left\| \begin{pmatrix}
\Delta \tilde{\bm{\omega}}_i - \Delta \tilde{\bm{\omega}}_i^\ast  \\ \Delta \tilde{\bm{\zeta}}_i - \Delta \tilde{\bm{\zeta}}_i^\ast
\end{pmatrix} \right\|  \leq&\; \epsilon_i C_{\rm sub} \|(\Delta \bm{z}, \Delta \bm{\lambda})\|\,, \\
C_{\rm sub} =&\;  C_{\Gamma^{-1}_{\rm loc}} C_{\Gamma_{\rm loc}} \left(1 + \frac{2C_d}{\sqrt{1 - \rho_d^2}} \right)\,.
\end{align*}
To finish the bound on the first term, notice that by definition,
\begin{align}
 \|(\Delta \tilde{\bm{{z}}} - \Delta \tilde{\bm{z}}^\ast, &\Delta \tilde{\bm{\lambda}} - \Delta \tilde{\bm{\lambda}}^\ast )\| ^2  \nonumber\\
 \leq&\; \sum_{i=0}^{M-1}\left\|\begin{pmatrix} \Delta \tilde{\bm{\omega}}_i - \Delta \tilde{\bm{\omega}}_i^\ast \\ \Delta \tilde{\bm{\zeta}}_i - \Delta \tilde{\bm{\zeta}}_i^\ast  \end{pmatrix} \right\|^2   
\nonumber \\ 
 \leq&\; C_{\rm sub}^2 \sum_{i=0}^{M-1} \epsilon_i^2 \|(\Delta \bm{z}, \Delta \bm{\lambda})\|^2\,. \label{eq:lemma-1-first-term-bound}    
\end{align}
We now return to the second (boundary) term of \eqref{eq:lemma-1-triangle-inequality}. By the definition of the composition operator, the composition takes each core from its own subproblem, so the error decomposes exactly over the cores:
\begin{align}
\nonumber  &\|(\Delta \tilde{\bm{{z}}}^\ast - \Delta \bm{z}, \Delta \tilde{\bm{\lambda}}^\ast - \Delta \bm{\lambda})\|^2  \\
&\quad = \sum_{i=0}^{M-1}\, \sum_{k \in \mathcal{K}_i} \left( \|\Delta \tilde{\bm{\omega}}_{i,k}^\ast - \Delta \bm{\omega}_{i,k}\|^2 + \|\Delta \tilde{\bm{\zeta}}_{i,k}^\ast - \Delta \bm{\zeta}_{i,k}\|^2 \right),\label{eq:lemma-1-core-decomposition} 
\end{align}
where the core intervals $\mathcal K_i$ are defined in Section~\ref{sec:preliminaries}. Every core stage lies at distance at least the nominal overlap $b_i$ from each \emph{interior} interface of subproblem $i$: by \eqref{eq:effective-overlaps} the clipping binds only when the corresponding end reaches the horizon, so every interior side satisfies $b_i^{S} = b_i$, whence $k - m_1^i \geq b_i^{L} = b_i$ when $m_1^i > 0$ and $m_2^i - k \geq b_i^{R} = b_i$ when $m_2^i < N$. Hence, squaring the per-stage estimate \eqref{eq:lemma-1-estimate}, applying $(a+b)^2 \leq 2a^2+2b^2$, and summing the geometric series over the core stages only, which now starts at distance $b_i$, yields
\begin{align*}
&\sum_{k \in \mathcal{K}_i} \left( \|\Delta \tilde{\bm{\omega}}_{i,k}^\ast - \Delta \bm{\omega}_{i,k}\|^2 + \|\Delta \tilde{\bm{\zeta}}_{i,k}^\ast - \Delta \bm{\zeta}_{i,k}\|^2 \right) \\
& \leq \frac{4C_d^2\,\rho_d^{\,2b_i}}{1-\rho_d^2}\bigg(\|\Delta\bm{x}_{m_1^i}\|^2  + \|(\Delta\bm{x}_{m_2^i};\Delta\bm{u}_{m_2^i};\Delta\bm{\lambda}_{m_2^i+1})\|^2\bigg), 
\end{align*}
where a boundary contribution is omitted when the corresponding end is exact ($m_1^i = 0$ or $m_2^i = N$), so that the $i$th bound vanishes entirely when the window covers the full horizon. Since the two boundary blocks are subvectors of the full Newton step, the parenthesized factor is at most $\|(\Delta \bm{z}, \Delta \bm{\lambda})\|^2$. Hence, substituting into \eqref{eq:lemma-1-core-decomposition} and taking square roots gives
\begin{align}
&\|(\Delta \tilde{\bm{{z}}}^\ast - \Delta \bm{z}, \Delta \tilde{\bm{\lambda}}^\ast - \Delta \bm{\lambda})\|  \nonumber \\
&\quad \leq \frac{2C_d}{\sqrt{1-\rho_d^2}} \left(\sum_{i=0}^{M-1}\rho_d^{\,2b_i}\right)^{\frac{1}{2}}\|(\Delta \bm{z}, \Delta \bm{\lambda})\|\,. \label{eq:lemma-1-second-term-bound}
\end{align}
Taking square roots of \eqref{eq:lemma-1-first-term-bound} and combining it with \eqref{eq:lemma-1-second-term-bound}, \eqref{eq:lemma-1-triangle-inequality} completes the result with $\rho := \rho_d$, $C_1 := \tfrac{2C_d}{\sqrt{1-\rho_d^2}}$, and $C_2 := C_{\rm sub}$.
\end{proof}

\subsection{Proof of Lemma \ref{lem:global-adaptive-well-posed}}

\begin{proof}
Fix the outer iteration $\tau$. Throughout the accuracy loop the outer iterate $(\bm z^\tau, \bm \lambda^\tau)$ and the merit parameters $(\eta_1^\tau, \eta_2^\tau, \epsilon^\tau)$ are held fixed, so $\Gamma^\tau$, $\nabla\mathcal L^\tau$, $\Psi^\tau$ and the exact Newton step $(\Delta \bm z^\tau, \Delta \bm \lambda^\tau)$ are constant. If $\nabla\mathcal L^\tau=\bm0$, Algorithm~\ref{alg:aotd-compact} returns the current KKT point before entering the accuracy loop. Hence, consider an iteration with $\nabla\mathcal L^\tau\neq\bm0$, so that the right-hand side of the global accuracy condition \eqref{eq:globalResidualBound} is strictly positive. The argument is as follows: every time \eqref{eq:globalResidualBound} fails, the update law forces a definite, monotone decrease of the parameter-only bound of Lemma~\ref{lem:global-residual-error}, which cannot fall indefinitely below a fixed positive floor. Thus, the condition is met after finitely many inner iterations.

\emph{Step 1 (uniform progress on every failing pass).} Each full-horizon stage belongs to exactly one subproblem core, so the core residuals $\bm r_i$ defined before \eqref{eq:subproblem-error-update} partition the global residual and $\sum_{i=0}^{M-1}\|\bm r_i\|^2 = \|\bm r\|^2$. Consequently, whenever \eqref{eq:globalResidualBound} fails, $\|\bm r\|$ exceeds its right-hand side $\epsilon^\tau\|\nabla\mathcal L^\tau\|/(\|\Gamma^\tau\|\Psi^\tau)$, that is, $\|\bm r\| > \sqrt M\,\varsigma$, where $\varsigma \coloneq \epsilon^\tau\|\nabla\mathcal L^\tau\|/(\sqrt M\,\|\Gamma^\tau\|\Psi^\tau) > 0$ is a fixed positive floor. Hence, by \eqref{eq:subproblem-error-update}--\eqref{eq:subproblem-overlap-update}, on every failing pass and for every $j$,
\begin{equation}\label{eq:lemma2-blocking-progress} 
\begin{aligned}[c]
\epsilon_j^\tau \;\gets\;& \epsilon_j^\tau\,e^{-\varrho_\epsilon\left(\|\bm r_j\|/\|\bm r\| + 1/\sqrt M\right)} \;\leq\; e^{-\varrho_\epsilon/\sqrt M}\epsilon_j^\tau,  \\
b_j^\tau \;\gets\;& b_j^\tau + \max\left\{1,\ \lceil\varrho_b\|\bm r_j\|/\|\bm r\|\rceil\right\} \;\geq\; b_j^\tau + 1.
\end{aligned}
\end{equation}
That is, on every failing pass \emph{every} subproblem shrinks its tolerance by at least the fixed factor $e^{-\varrho_\epsilon/\sqrt M} < 1$ (the floor in \eqref{eq:subproblem-error-update}) and enlarges its nominal overlap by at least one stage (the $\max$ in \eqref{eq:subproblem-overlap-update}).

\emph{Step 2 (a parameter-only certificate for the accuracy test).} Since the exact Newton step satisfies $\Gamma^\tau(\Delta\bm z^\tau, \Delta\bm\lambda^\tau) + \nabla\mathcal L^\tau = \bm 0$, subtracting it from the definition of the residual gives $\bm r = \Gamma^\tau(\Delta\tilde{\bm z}^\tau - \Delta\bm z^\tau,\ \Delta\tilde{\bm\lambda}^\tau - \Delta\bm\lambda^\tau)$, and Lemma~\ref{lem:global-residual-error} bounds the concatenated error directly:
\begin{align}
\|\bm r\|  \leq& \|\Gamma^\tau\|\left[C_1\Big(\sum_{j=0}^{M-1}\rho^{\,2b_j^\tau}\Big)^{\frac{1}{2}} + C_2\Big(\sum_{j=0}^{M-1}(\epsilon_j^\tau)^2\Big)^{\frac{1}{2}}\right] \nonumber \\ 
& \times \|(\Delta \bm{z}^\tau, \Delta \bm{\lambda}^\tau)\|\,,\label{eq:lemma2-per-subproblem-residual} 
\end{align}
with $C_1, C_2, \rho$ the constants of Lemma~\ref{lem:global-residual-error}. Note that every decomposed residual obeys $\|\bm r_i\| \leq \|\bm r\|$. Since the iterate is fixed throughout the accuracy loop, the test \eqref{eq:globalResidualBound} holds as soon as
\begin{align}
&C_1\Big(\sum_{j=0}^{M-1}\rho^{\,2b_j^\tau}\Big)^{\frac{1}{2}} + C_2\Big(\sum_{j=0}^{M-1}(\epsilon_j^\tau)^2\Big)^{\frac{1}{2}} \;  \nonumber \\
&\quad \leq\; \frac{\sqrt M\,\varsigma}{\|\Gamma^\tau\|\,\|(\Delta \bm{z}^\tau, \Delta \bm{\lambda}^\tau)\|}\,, 
\label{eq:lemma2-certificate}
\end{align}
whose right-hand side is a fixed positive threshold and whose left-hand side depends only on the local parameters. (If $\|(\Delta \bm{z}^\tau, \Delta \bm{\lambda}^\tau)\| = 0$ then $\nabla\mathcal L^\tau = \bm 0$, excluded above.) In particular, since each $\ell_2$ sum is at most $\sqrt M$ times its largest term, \eqref{eq:lemma2-certificate} holds once every subproblem satisfies the symmetric per-subproblem targets
\begin{subequations}\label{eq:lemma2-per-subproblem-targets}
\begin{align}
C_1\,\rho^{\,b_j^\tau} \;\leq&\; \frac{\varsigma}{2\|\Gamma^\tau\|\,\|(\Delta \bm{z}^\tau, \Delta \bm{\lambda}^\tau)\|}, \\
C_2\,\epsilon_j^\tau \;\leq&\; \frac{\varsigma}{2\|\Gamma^\tau\|\,\|(\Delta \bm{z}^\tau, \Delta \bm{\lambda}^\tau)\|}\,. 
\end{align}
\end{subequations}
\emph{Step 3 (finite termination).} By \eqref{eq:lemma2-blocking-progress}, on every failing pass \emph{every} subproblem contracts its tolerance by at least the fixed factor $e^{-\varrho_\epsilon/\sqrt M}$ and enlarges its nominal overlap by at least one stage. Hence the tolerance target in \eqref{eq:lemma2-per-subproblem-targets} is met by \emph{every} subproblem after at most
\begin{align*}
k^{\epsilon} \coloneqq
\max_j\left\{
0,\,
\left\lceil
\frac{\sqrt M}{\varrho_\epsilon}
\ln\!\left(
\frac{2C_2\|\Gamma^\tau\|
\|(\Delta\bm z^\tau,\Delta\bm\lambda^\tau)\|
\epsilon_j^{\tau,0}}
{\varsigma}
\right)
\right\rceil
\right\} 
\end{align*}
failing passes, where $\epsilon_j^{\tau,0}$ denotes the value of $\epsilon_j^\tau$ on entry to the accuracy loop, and the overlap target is met by \emph{every} subproblem after at most
\begin{align*}
k^b \coloneqq
\max_j\left\{
0,\,
\left\lceil
\log_{1/\rho}\!\left(
\frac{2C_1\|\Gamma^\tau\|
\|(\Delta\bm z^\tau,\Delta\bm\lambda^\tau)\|}
{\varsigma}
\right)
-b_j^{\tau,0}
\right\rceil
\right\} 
\end{align*}
failing passes, where, as above, $b_j^{\tau, 0}$ denotes the value of $b_j^\tau$ on entry to the accuracy loop. Therefore, after at most $\max\{k^{\epsilon}, k^{b}\}$ failing passes every subproblem satisfies both targets in \eqref{eq:lemma2-per-subproblem-targets}, so \eqref{eq:lemma2-certificate} and hence \eqref{eq:globalResidualBound} hold and the loop terminates.
\end{proof}

\subsection{Proof of Lemma \ref{lem:wellposed-global-descent}}

\begin{proof}
For simplicity, we suppress the iteration index $\tau$ and use the shorthand notation $\nabla_{\bm{z}}\mathcal{L} = \nabla_{\bm{z}}\mathcal{L}(\bm{z}, \bm{\lambda})$,  $\nabla_{\bm{\lambda}}\mathcal{L} = \nabla_{\bm{\lambda}}\mathcal{L}(\bm{z}, \bm{\lambda})\,, $ and  $\nabla\mathcal{L} = (\nabla_{\bm{z}}\mathcal{L}; \nabla_{\bm{\lambda}}\mathcal{L})$ for both the Lagrangian $\mathcal{L}$ and augmented Lagrangian $\mathcal{L}_{\bm{\eta}}$ respectively.  By definition, we have that 
\begin{align}
\begin{pmatrix}
\nabla_{\bm{z}} \mathcal{L}_{\bm{\eta}} \\
\nabla_{\bm{\lambda}} \mathcal{L}_{\bm{\eta}}
\end{pmatrix}^T \begin{pmatrix}
\Delta \tilde{\bm{z}} \\ \Delta \tilde{\bm{\lambda}}
\end{pmatrix} =& \underbrace{\begin{pmatrix}
\nabla_{\bm{z}} \mathcal{L}_{\bm{\eta}} \\ 
\nabla_{\bm{\lambda}} \mathcal{L}_{\bm{\eta}}
\end{pmatrix}^T \begin{pmatrix}
\Delta \bm{z} \\ \Delta \bm{\lambda}
\end{pmatrix}}_{=: \mathcal{I}_1}  \nonumber \\
& \hskip-0.8cm +  \underbrace{\begin{pmatrix}
\nabla_{\bm{z}} \mathcal{L}_{\bm{\eta}} \\ 
\nabla_{\bm{\lambda}} \mathcal{L}_{\bm{\eta}}
\end{pmatrix}^T \begin{pmatrix}
\Delta \tilde{\bm{z}} -\Delta \bm{z} \\ \Delta \tilde{\bm{\lambda}} - \Delta \bm{\lambda} 
\end{pmatrix}}_{=: \mathcal{I}_2} \,. 
\label{eq:lemma3-aug-expand}
\end{align}
Notice that $\mathcal{I}_1$ is the exact Newton descent direction considered the "good" term which will be shown to be large enough to cancel out the perturbation terms $\mathcal{I}_2$. By a standard approach as in \cite[Theorem 5.1]{naFOTD}, we have that
\begin{align}
\mathcal{I}_1 :=& \begin{pmatrix}
\nabla_{\bm{z}} \mathcal{L}_{\bm{\eta}} \\ 
\nabla_{\bm{\lambda}}\mathcal{L}_{\bm{\eta}}
\end{pmatrix}^T \begin{pmatrix}
\Delta \bm{z} \\ \Delta \bm{\lambda}
\end{pmatrix}   \nonumber \\
=& \begin{pmatrix}
\Delta \bm{z} \\ \Delta \bm{\lambda}
\end{pmatrix}^T \begin{pmatrix}
I + \eta_2 H & \eta_1 G^T \\ \eta_2G & I 
\end{pmatrix} \begin{pmatrix}
\nabla_{\bm{z}} \mathcal{L} \\ \nabla_{\bm{\lambda}}\mathcal{L} 
\end{pmatrix}   \nonumber \\
=&  -\begin{pmatrix}
\Delta \bm{z} \\ \Delta \bm{\lambda}
\end{pmatrix}^T \begin{pmatrix}
I + \eta_2 H & \eta_1 G^T \\ \eta_2G & I 
\end{pmatrix} \begin{pmatrix}
\hat{H} & G^T \\ G & \bm{0}
\end{pmatrix}
\begin{pmatrix}
\Delta \bm{z} \\ \Delta \bm{\lambda}
\end{pmatrix}   \nonumber \\
=& -\frac{\eta_1}{2} \|\nabla_{\bm{\lambda}}\mathcal{L}\|^2 - \frac{\eta_2}{2} \|\nabla_{\bm{z}}\mathcal{L}\|^2  \nonumber \\
& - (\Delta \bm{z})^T \bigg[(I+\eta_2H)\hat{H} + \frac{\eta_1}{2}G^TG\bigg]\Delta \bm{z}   \nonumber \\
& -(\Delta \bm{\lambda})^T G \bigg[2I + \eta_2(\hat{H}+H)\bigg]\Delta \bm{z}  \nonumber \\
&+ \bigg( \frac{\eta_2}{2}\|\nabla_{\bm{z}}\mathcal{L}\|^2 - \eta_2 \|\hat{H}\Delta \bm{z} + \nabla_{\bm{z}}\mathcal{L}\|^2 \bigg)\,. 
\label{eq:lemma3-i1-expand}
\end{align}
Note, the last two terms, by direct substitution of $\nabla_{\bm{z}}\mathcal{L} = -\hat{H}\Delta \bm{z} - G^T \Delta \bm{\lambda}$, satisfy 
\begin{align}
\nonumber    \frac{\eta_2}{2} \|\nabla_{\bm{z}}\mathcal{L}\|^2 &- \eta_2\|\hat{H}\Delta \bm{z} + \nabla_{\bm{z}}\mathcal{L}\|^2 \\ &\leq \eta_2 \|\hat{H}\Delta \bm{z}\|^2  \nonumber
+ \eta_2 (\Delta \bm{z})^T \hat{H}G^T\Delta \bm{\lambda} \nonumber \\ & \quad - \frac{\eta_2}{2} \|G^T\Delta \bm{\lambda}\|^2\,. 
\label{eq:lemma3-last-two}
\end{align}
Choose the threshold $\varpi\geq1$. Under \eqref{eq:wellposed-global-descent-parameter-condition}, we have $\eta_2\leq1/\varpi\leq1$ and $\eta_1\geq\varpi/\eta_2\geq\varpi^2\geq1$, hence $\eta_1\geq\eta_2$. Then substituting \eqref{eq:lemma3-last-two} into \eqref{eq:lemma3-i1-expand}, using $\|\hat{H}\|, \|H\| \leq \Upsilon_{\rm Upper}$ from Assumptions \ref{assumption:uniform-boundedness} and \ref{assumption:compactness} yields
\begin{align*}
\mathcal{I}_1 \le&-\frac{\eta_2}{2} \|\nabla \mathcal{L}\|^2 \\
& - (\Delta \bm{z})^T \bigg[(I+\eta_2H)\hat{H} + \frac{\eta_1}{2}G^TG\bigg]\Delta \bm{z} \\
& -(\Delta \bm{\lambda})^T G \bigg[2I + \eta_2(\hat{H}+H)\bigg]\Delta \bm{z}  \\
&+  \eta_2 \|\hat{H}\Delta \bm{z}\|^2 + \eta_2 (\Delta \bm{z})^T \hat{H}G^T\Delta \bm{\lambda} - \frac{\eta_2}{2} \|G^T\Delta \bm{\lambda}\|^2  \\
=& -\frac{\eta_2}{2} \|\nabla \mathcal{L}\|^2 \\
& - (\Delta \bm{z})^T \bigg[(I+\eta_2(H-\hat{H}))\hat{H} + \frac{\eta_1}{2}G^TG\bigg]\Delta \bm{z} \\
& -(\Delta \bm{\lambda})^T G \bigg[2I + \eta_2H\bigg]\Delta \bm{z}  \\
&  - \frac{\eta_2}{2} \|G^T\Delta \bm{\lambda}\|^2  \\
\leq& -\frac{\eta_2}{2} \|\nabla\mathcal{L} \|^2 +\eta_2\|(H-\hat{H})\hat{H}\|\|\Delta \bm{z}\|^2 \\
& -(\Delta \bm{z})^T\big[\hat{H}+\frac{\eta_1}{2}G^TG \big]\Delta \bm{z}  \\
& +2 \|\Delta \bm{\lambda}\|\|G\Delta \bm{z}\| + \eta_2 \|H\|\|G^T\Delta \bm{\lambda}\| \|\Delta \bm{z}\| \\
& - \frac{\eta_2}{2} \|G^T\Delta \bm{\lambda}\|^2  \\
\leq& -\frac{\eta_2}{2} \|\nabla\mathcal{L} \|^2 +2\eta_2 \Upsilon_{\rm Upper}^2\|\Delta \bm{z}\|^2 \\
& -(\Delta \bm{z})^T\big[\hat{H}+\frac{\eta_1}{2}G^TG \big]\Delta \bm{z}  \\
& +2 \|\Delta \bm{\lambda}\|\|G\Delta \bm{z}\| + \eta_2 \Upsilon_{\rm Upper} \|G^T \Delta \bm{\lambda}\| \|\Delta \bm{z}\| \\
& - \frac{\eta_2}{2} \|G^T\Delta \bm{\lambda}\|^2 .
\end{align*}
Applying Young's inequality yields
\begin{align*}
\mathcal{I}_1 \leq& -\frac{\eta_2}{2} \|\nabla\mathcal{L} \|^2 +2\eta_2 \Upsilon_{\rm Upper}^2\|\Delta \bm{z}\|^2 \\
& -(\Delta \bm{z})^T\big[\hat{H}+\frac{\eta_1}{2}G^TG \big]\Delta \bm{z}  \\
&+2 \|\Delta \bm{\lambda}\|\|G\Delta \bm{z}\| + \eta_2 \Upsilon_{\rm Upper} \|G^T \Delta \bm{\lambda}\| \|\Delta \bm{z}\| \\
& - \frac{\eta_2}{2} \|G^T\Delta \bm{\lambda}\|^2  \\
\leq& -\frac{\eta_2}{2} \|\nabla\mathcal{L} \|^2 +3\eta_2 \Upsilon_{\rm Upper}^2\|\Delta \bm{z}\|^2 \\
& -(\Delta \bm{z})^T\big[\hat{H}+\frac{\eta_1}{2}G^TG \big]\Delta \bm{z}  \\
&+2 \|\Delta \bm{\lambda}\|\|G\Delta \bm{z}\|  - \frac{\eta_2}{4} \|G^T\Delta \bm{\lambda}\|^2 .
\end{align*}
By the identity blocks associated with the initial-state and dynamics constraints, $G$ has full row rank. Moreover, Assumptions \ref{assumption:uniform-controllability} and \ref{assumption:uniform-boundedness} provide a positive lower bound on $\sigma_{\min}(G^\tau)$ that is independent of $\tau$. Combining the inequality $GG^T \succeq (\sigma_{\min}^\tau)^2I$, which follows from the definition of $\sigma_{\min}^\tau$, with Young's inequality applied to $2\|\Delta \bm{\lambda}\|\|G\Delta \bm{z}\|$ yields
\begin{align}
\mathcal{I}_1 \leq&\; -\frac{\eta_2}{2} \|\nabla\mathcal{L} \|^2 +3\eta_2 \Upsilon_{\rm Upper}^2\|\Delta \bm{z}\|^2  \nonumber \\
& -(\Delta \bm{z})^T\bigg[\hat{H}+\left(\frac{\eta_1}{2}-\frac{8}{\eta_2(\sigma_{\min}^\tau)^2} \right)G^TG \bigg]\Delta \bm{z}   \nonumber \\
&   - \frac{\eta_2 (\sigma_{\min}^\tau)^2}{8} \|\Delta \bm{\lambda}\|^2 .
\label{eq:lemma3-i1-intermediate}
\end{align}
We now split $\Delta \bm{z} = \Delta \bm{u} + \Delta \bm{v}$ where $\Delta \bm{u} \in \rm{range}(G^T)$ and $\Delta \bm{v} \in \rm{ker}(G)$ to apply Assumption \ref{assumption:uniform-lower-bound-hessian} for the second and third terms of \eqref{eq:lemma3-i1-intermediate}. Then, let $\bar{\bm{u}} $ be such that $\Delta \bm{u} = G^T\bar{\bm{u}}$ which exists by the decomposition of $\Delta \bm{z}$. We then have
\begin{align*}
&3\eta_2 \Upsilon_{\rm Upper}^2\|\Delta \bm{z}\|^2 \\
&\quad -(\Delta \bm{z})^T\bigg[\hat{H}+\left(\frac{\eta_1}{2}-\frac{8}{\eta_2(\sigma_{\min}^\tau)^2} \right)G^TG \bigg]\Delta \bm{z}  \\
&=\; 3 \eta_2 \Upsilon_{\rm Upper}^2 \|\Delta \bm{z}\|^2 - (\Delta \bm{v})^T \hat{H} \Delta \bm{v} \\
&\quad - 2(\Delta \bm{v})^T \hat{H} \Delta \bm{u} - (\Delta \bm{u})^T\hat{H} \Delta \bm{u} \\
&\quad - \bigg(\frac{\eta_1}{2} - \frac{8}{\eta_2 (\sigma_{\min}^\tau)^2} \bigg) \|GG^T \bar{\bm{u}}\|^2  \\
&\leq\; (3 \eta_2 \Upsilon_{\rm Upper}^2 - \gamma_{\rm RH}) \|\Delta \bm{z}\|^2 \\
&\quad + 2 \Upsilon_{\rm Upper} \|\Delta \bm{u}\| \|\Delta \bm{v}\| + (\gamma_{\rm RH} + \Upsilon_{\rm Upper}) \|\Delta \bm{u}\|^2  \\
&\quad  - \bigg(\frac{\eta_1 (\sigma_{\min}^\tau)^2}{2}  - \frac{8}{\eta_2} \bigg)\|G^T \bar{\bm{u}}\|^2 .
\end{align*}
Applying Young's inequality again yields
\begin{align}
&3\eta_2 \Upsilon_{\rm Upper}^2\|\Delta \bm{z}\|^2  \nonumber \\
&\quad -(\Delta \bm{z})^T\bigg[\hat{H}+\left(\frac{\eta_1}{2}-\frac{8}{\eta_2(\sigma_{\min}^\tau)^2} \right)G^TG \bigg]\Delta \bm{z}   \nonumber \\
&\leq\;  (3 \eta_2 \Upsilon_{\rm Upper}^2 - \frac{\gamma_{\rm RH}}{2}) \|\Delta \bm{z}\|^2  \nonumber \\
&\quad +   \bigg(\gamma_{\rm RH} + \Upsilon_{\rm Upper} + \frac{2 \Upsilon_{\rm Upper}^2}{\gamma_{\rm RH}} - \frac{\eta_1 (\sigma_{\min}^\tau)^2}{2}  + \frac{8}{\eta_2 }\bigg)  \nonumber \\ & \qquad \times  \|G^T \bar{\bm{u}}\|^2   \nonumber \\
&=\;  (3 \eta_2 \Upsilon_{\rm Upper}^2 - \frac{\gamma_{\rm RH}}{2}) \|\Delta \bm{z}\|^2  \nonumber \\
&\quad +   \bigg(\gamma_{\rm RH} + \Upsilon_{\rm Upper} + \frac{2 \Upsilon_{\rm Upper}^2}{\gamma_{\rm RH}} - \frac{\eta_1 (\sigma_{\min}^\tau)^2}{2}  + \frac{8}{\eta_2 }\bigg) \nonumber \\ & \qquad \times  \|\Delta \bm{u}\|^2. \label{eq:lemma3-bad-term}
\end{align}
We impose the following conditions on $\eta_1$ and $\eta_2$:
\begin{subequations}\label{eq:lemma3-i1-conditions}
    \begin{align}
\eta_2&<\frac{\gamma_{\rm RH}}{6\Upsilon_{\rm Upper}^2+(\sigma_{\min}^\tau)^2/4}, \\
\eta_1\eta_2&>\frac{2}{(\sigma_{\min}^\tau)^2}\Bigg(8+\frac{\gamma_{\rm RH}^2+\gamma_{\rm RH}\Upsilon_{\rm Upper}+2\Upsilon_{\rm Upper}^2}{6\Upsilon_{\rm Upper}^2+(\sigma_{\min}^\tau)^2/4}\Bigg).  
\end{align}
\end{subequations}
These conditions make the second term in \eqref{eq:lemma3-bad-term} negative and yield the following bound.
\begin{align}
&3\eta_2 \Upsilon_{\rm Upper}^2\|\Delta \bm{z}\|^2  \nonumber \\
\nonumber  &\quad -(\Delta \bm{z})^T\bigg[\hat{H}+\left(\frac{\eta_1}{2}-\frac{8}{\eta_2(\sigma_{\min}^\tau)^2} \right)G^TG \bigg]\Delta \bm{z} \\ &\qquad \leq -\frac{\eta_2 (\sigma_{\min}^\tau)^2}{8} \|\Delta \bm{z}\|^2\,. 
\label{eq:lemma3-decay-I1}
\end{align}
From here, plugging \eqref{eq:lemma3-decay-I1} into \eqref{eq:lemma3-i1-intermediate} yields
\begin{align}
\mathcal{I}_1 \leq -\frac{\eta_2}{2} \|\nabla \mathcal{L}\|^2 - \frac{\eta_2 (\sigma_{\min}^\tau)^2}{8}\left\|\begin{pmatrix} \Delta \bm{z} \\ \Delta \bm{\lambda} \end{pmatrix} \right\|^2\,.  \label{eq:lemma3-i1-final}
\end{align}
We now turn to the second term, $\mathcal I_2$, and apply the same expansion of $\mathcal L_{\bm\eta}$.
\begin{align}
\mathcal{I}_2 =&\begin{pmatrix}
\nabla_{\bm z}
\mathcal{L}_{\bm{\eta}}\\ \nabla_{\bm \lambda}\mathcal{L}_{\bm{\eta}}    \end{pmatrix}^T  \nonumber \begin{pmatrix}
\Delta \tilde{\bm{z}} - \Delta \bm{z} \\ \Delta \tilde{\bm{\lambda}}-\Delta \bm{\lambda}
\end{pmatrix} \nonumber \\ =&  \nonumber 
-\begin{pmatrix}
\Delta \tilde{\bm{z}} - \Delta \bm{z} \\ \Delta \tilde{\bm{\lambda}}-\Delta \bm{\lambda}
\end{pmatrix}^T \begin{pmatrix}
\begin{aligned} &(I+\eta_2 H)\hat{H}\\ &+\eta_1 G^TG\end{aligned} & (I+\eta_2 H)G^T \\ G(I+\eta_2\hat{H}) & \eta_2 GG^T
\end{pmatrix} \nonumber \\ & \times \begin{pmatrix}
\Delta \bm{z} \\ \Delta \bm{\lambda}
\end{pmatrix}\,. 
\label{eq:lemma3-i2-term}
\end{align}
Furthermore, note that the condition on the residual 
\begin{align}
&\|\bm{r}(\Delta \tilde{\bm{z}}^\tau, \Delta \tilde{\bm{\lambda}}^\tau)\| \leq \frac{\epsilon^\tau \|\nabla \mathcal{L}^\tau\|}{\|\Gamma^\tau\|\Psi^\tau}  \nonumber
\end{align}
implies that
\begin{align} \Psi^\tau \left\|\Gamma^\tau \begin{pmatrix} \Delta \tilde{\bm{z}}^\tau - \Delta \bm{z}^\tau \\ \Delta \tilde{\bm{\lambda}}^\tau -\Delta \bm{\lambda}^\tau \end{pmatrix} \right\| \leq \frac{\epsilon^\tau \|\nabla \mathcal{L}^\tau\|}{\|\Gamma^\tau\|}\,. 
\label{eq:lemma3-residual-bound}
\end{align}
Now, using \cite[Lemma 5.1]{naFOTD}, and writing $h^\tau \coloneqq \|\hat H^\tau\|$ and $\sigma_{\min}^\tau := \sigma_{\min}(G^\tau)$, we have
\begin{align}
\nonumber \|\Gamma^{-1}\|\leq&\;\frac{1}{\gamma_{\rm RH}}+\frac{2}{\sigma_{\min}^\tau}+\frac{2h}{\sigma_{\min}^\tau\gamma_{\rm RH}}
+\frac{h}{(\sigma_{\min}^\tau)^2} \\ 
&+\frac{h^2}{(\sigma_{\min}^\tau)^2\gamma_{\rm RH}} \nonumber \\ \leq&\;\frac{7}{20}\Psi.
\label{eq:lemma3-kkt-inv-bound}
\end{align}
Using \eqref{eq:lemma3-kkt-inv-bound} in conjunction with \eqref{eq:lemma3-residual-bound}
\begin{align}
\left\| \begin{pmatrix} \Delta \tilde{\bm{z}}^\tau - \Delta \bm{z}^\tau \\ \Delta \tilde{\bm{\lambda}}^\tau -\Delta \bm{\lambda}^\tau \end{pmatrix} \right\| \leq&\; \frac{\epsilon^\tau \|\nabla \mathcal{L}^\tau \|}{\|\Gamma^\tau\|}  \nonumber   \\ 
\leq&\; \frac{\epsilon^\tau \|\Gamma^\tau \| \left\| \begin{pmatrix} \Delta \bm{z}^\tau \\ \Delta \bm{\lambda}^\tau \end{pmatrix} \right\|}{\|\Gamma^\tau\|}  \nonumber \\ 
\leq&\; \epsilon^\tau \left\| \begin{pmatrix} \Delta \bm{z}^\tau \\ \Delta \bm{\lambda}^\tau \end{pmatrix} \right\|\,. 
\label{eq:lemma3-residual-final}
\end{align}
Note, under Assumptions \ref{assumption:uniform-boundedness}--\ref{assumption:compactness} (see \cite[Lemma 5.1]{naFOTD}), the banded structure of $G^\tau$ gives $\|G^\tau\| \leq 1 + 2\Upsilon_{\rm Upper}$ for any $\tau$. For brevity, we henceforth let $\Upsilon_{\rm Upper}$ denote a common uniform upper bound on $\|H^\tau\|$, $\|\hat H^\tau\|$, and $\|G^\tau\|$. Applying \eqref{eq:lemma3-residual-final}, $\|H\|, \|\hat{H}\|, \|G\| \leq \Upsilon_{\rm Upper}$ to \eqref{eq:lemma3-i2-term} with \eqref{eq:lemma3-residual-final} yields
\begin{align*}
\mathcal{I}_2 \leq&\; \epsilon \left\|\begin{pmatrix}
\Delta \bm{z} \\ \Delta \bm{\lambda} 
\end{pmatrix} \right\|^2  \big(\Upsilon_{\rm Upper} + \eta_2 \Upsilon_{\rm Upper}^2 + (\eta_1+\eta_2)\Upsilon_{\rm Upper}^2\\ &\quad + 2(1 + \eta_2\Upsilon_{\rm Upper})\Upsilon_{\rm Upper}\big)  \\
\leq&\; \epsilon \left\|\begin{pmatrix}
\Delta \bm{z} \\ \Delta \bm{\lambda} 
\end{pmatrix} \right\|^2  (3 \Upsilon_{\rm Upper} + 4 \eta_2 \Upsilon_{\rm Upper}^2 + \eta_1 \Upsilon_{\rm Upper}^2)\,. 
\end{align*}
Using \eqref{eq:lemma3-i1-conditions}, we have that 
\begin{align*}
\eta_1 \geq&\;\frac{\eta_1 \eta_2}{\eta_2} \\ \geq&\; \frac{16 / (\sigma_{\min}^\tau)^2}{\gamma_{\rm RH}/ (6 \Upsilon_{\rm Upper}^2 + (\sigma_{\min}^\tau)^2/4)} \\
=&\; \frac{96 \Upsilon_{\rm Upper}^2}{(\sigma_{\min}^\tau)^2 \gamma_{\rm RH}} + \frac{4}{\gamma_{\rm RH}} \\  \geq& \; \frac{4}{\gamma_{\rm RH}}\,. 
\end{align*}
Now, noting that by construction $\gamma_{\rm RH} \leq \Upsilon_{\rm Upper}$ since
\begin{align*}
\gamma_{\rm RH} \leq \lambda_{\rm min}(Z^T\hat{H}Z) \leq \|Z^T\hat{H} Z\| \leq \|\hat{H}\| \leq \Upsilon_{\rm Upper}\,.
\end{align*}
Hence, with $\eta_1 \geq \eta_2$ and noting $5 \eta_1 \Upsilon_{\rm Upper} > 3$, we obtain a bound on $\mathcal{I}_2$ as
\begin{align}
\mathcal{I}_2 \leq&\; \epsilon \left\|\begin{pmatrix}
\Delta \bm{z} \\ \Delta \bm{\lambda} 
\end{pmatrix} \right\|^2  (3 \Upsilon_{\rm Upper} + 4 \eta_2 \Upsilon_{\rm Upper}^2 + \eta_1 \Upsilon_{\rm Upper}^2)  \nonumber \\ 
\leq&\; \epsilon 10 \eta_1 \Upsilon_{\rm Upper}^2 \left\|\begin{pmatrix}
\Delta \bm{z} \\ \Delta \bm{\lambda} 
\end{pmatrix} \right\|^2  \,. 
\label{eq:lemma3-i2-final}
\end{align}
Combining \eqref{eq:lemma3-aug-expand}, \eqref{eq:lemma3-i1-final}, and \eqref{eq:lemma3-i2-final}, we obtain
\begin{align*}&
\begin{pmatrix}
\nabla_{\bm{z}} \mathcal{L}_{\bm{\eta}} \\
\nabla_{\bm{\lambda}} \mathcal{L}_{\bm{\eta}}
\end{pmatrix}^T \begin{pmatrix}
\Delta \tilde{\bm{z}} \\ \Delta \tilde{\bm{\lambda}}
\end{pmatrix} \nonumber \\ 
&\qquad \leq -\frac{\eta_2}{2} \|\nabla\mathcal{L}\|^2  \\ &\qquad \quad  -\left(\frac{\eta_2(\sigma_{\min}^\tau)^2}{8} - 10\eta_1\epsilon \Upsilon^2_{\rm Upper} \right) \left\|\begin{pmatrix} \Delta \bm{z} \\ \Delta \bm{\lambda} \end{pmatrix} \right\|^2\,, 
\end{align*}
which satisfies the descent condition when 
\begin{align}
\frac{\epsilon \eta_1}{\eta_2} \leq \frac{(\sigma_{\min}^\tau)^2 }{80 \Upsilon_{\rm Upper}^2}\,. \label{eq:lemma3-final-cond}
\end{align}
Since $\sigma_{\min}^\tau$ is upper and lower bounded by the assumptions, there exists a sufficiently large $\varpi\geq1$ such that \eqref{eq:wellposed-global-descent-parameter-condition} implies \eqref{eq:lemma3-i1-conditions} and \eqref{eq:lemma3-final-cond} completing the result. 
\end{proof}

\subsection{Proof of Corollary \ref{corr:achievability}}

\begin{proof}
Fix an outer iteration $\tau$ at which the stopping test does not hold. Then $\|\nabla\mathcal L^\tau\|>\mathrm{tol}_{\rm KKT}\geq0$. We show that the inner loop of Algorithm~\ref{alg:aotd-compact} terminates after finitely many sub-iterations.

Every update of \eqref{eq:adaptive-parameter-update}--\eqref{eq:adaptive-epsilon-update} satisfies
\begin{equation*}
\begin{aligned}
&\eta_1^\tau \eta_2^\tau \;\mapsto\; \nu \, \eta_1^\tau \eta_2^\tau,
\qquad\qquad
\eta_2^\tau \;\mapsto\; \eta_2^\tau / \nu, \\
&\frac{\epsilon^\tau\eta_1^\tau }{\eta_2^\tau}  \mapsto\; 
\min\left\{\epsilon^\tau_{\rm max},\frac{\epsilon^\tau}{\nu^4}\right\}
\frac{\nu^3\eta_1^\tau}{\eta_2^\tau}
\;\leq\;
\frac{1}{\nu}\frac{\epsilon^\tau\eta_1^\tau}{\eta_2^\tau}\,. 
\end{aligned}
\end{equation*}
Since $\varpi$ in Lemma~\ref{lem:wellposed-global-descent} is independent of $\tau$, the descent direction condition \eqref{eq:descent-direction-condition} is therefore reached after finitely many activations of \eqref{eq:adaptive-parameter-update}--\eqref{eq:adaptive-epsilon-update}.

Between consecutive updates the global parameters $(\eta_1^\tau, \eta_2^\tau, \epsilon^\tau)$ are frozen. By Lemmas~\ref{lem:global-residual-error}--\ref{lem:global-adaptive-well-posed}, there exist $\{\epsilon_i^\tau, b_i^\tau\}_{i \in [M-1]}$ satisfying \eqref{eq:globalResidualBound}, and since the updates \eqref{eq:subproblem-error-update}--\eqref{eq:subproblem-overlap-update} drive $\epsilon_i^\tau$ monotonically down and $b_i^\tau$ monotonically up, condition \eqref{eq:globalResidualBound} is reached in finitely many local updates within each segment of frozen global parameters. Combining yields finitely many sub-iterations overall, at the end of which both \eqref{eq:globalResidualBound} and \eqref{eq:descent-direction-condition} hold simultaneously.
\end{proof}

\subsection{Proof of Lemma \ref{lem:stability}}

\begin{proof}
The proof is broken into two parts where we first show that the outer while loop is guaranteed to stabilize and once stabilized, then the inner loop is guaranteed to stabilize.

\noindent (a) Outer loop stabilization

The parameters $\eta_1^\tau, \eta_2^\tau, \epsilon^\tau$ are updated through \eqref{eq:adaptive-parameter-update}--\eqref{eq:adaptive-epsilon-update} only when the descent condition \eqref{eq:descent-direction-condition} fails, and by Lemma \ref{lem:wellposed-global-descent} there is a constant $\varpi$, independent of $\tau$, such that \eqref{eq:descent-direction-condition} holds whenever \eqref{eq:wellposed-global-descent-parameter-condition} is met, i.e.\ $\eta_1^\tau\eta_2^\tau \geq \varpi$, $\eta_2^\tau \leq 1/\varpi$, and $\epsilon^\tau\eta_1^\tau/\eta_2^\tau \leq 1/\varpi$. Each descent-failure update sends $\eta_1 \mapsto \nu^2\eta_1$, $\eta_2 \mapsto \eta_2/\nu$, and $\epsilon \mapsto \min\{\epsilon^\tau_{\rm max}, \epsilon/\nu^4\}$, so that
\begin{align*}
\eta_1^\tau\eta_2^\tau \mapsto& \nu\,\eta_1^\tau\eta_2^\tau\,,  \\ \qquad
\eta_2^\tau \mapsto& \eta_2^\tau/\nu\,,  \qquad
\\    \frac{\epsilon^\tau\eta_1^\tau}{\eta_2^\tau} \mapsto& \min\Big\{\epsilon^\tau_{\rm max}, \tfrac{\epsilon^\tau}{\nu^4}\Big\}\frac{\nu^3\eta_1^\tau}{\eta_2^\tau} \leq \frac{1}{\nu}\cdot\frac{\epsilon^\tau\eta_1^\tau}{\eta_2^\tau}\,. 
\end{align*}
Hence each update multiplies $\eta_1^\tau\eta_2^\tau$ by exactly $\nu$ and divides both $\eta_2^\tau$ and the ratio $\epsilon^\tau\eta_1^\tau/\eta_2^\tau$ by at least $\nu$. Since the thresholds $\varpi, 1/\varpi$ in \eqref{eq:wellposed-global-descent-parameter-condition} are independent of $\tau$, all three conditions hold after at most
\begin{align*}
\max\Big\{\big\lceil\log_\nu\tfrac{\varpi}{\eta_1^0\eta_2^0}\big\rceil,\ \big\lceil\log_\nu(\varpi\eta_2^0)\big\rceil,\ \big\lceil\log_\nu\tfrac{\varpi\epsilon^0\eta_1^0}{\eta_2^0}\big\rceil,\ 0\Big\}
\end{align*}
descent-failure updates, each occurring at a finite outer iteration since the inner accuracy loop terminates finitely (Lemma \ref{lem:global-adaptive-well-posed}). Let $\bar\tau$ be the last such iteration. Then for all $\tau \geq \bar\tau$ the descent condition holds and the penalty parameters are frozen, $(\eta_1^\tau,\eta_2^\tau) = (\eta_1^{\bar\tau},\eta_2^{\bar\tau}) =: (\eta_1^\star,\eta_2^\star)$.

It remains to control $\epsilon^\tau$, which, after the descent condition can no longer fail, can only be modified through the enforcement of the cap $\epsilon^\tau_{\rm max}$. By Assumptions~\ref{assumption:uniform-controllability}--\ref{assumption:uniform-boundedness}, there exist constants $\bar{\Upsilon},\bar{\Psi}>0$ such that $\Upsilon^\tau\leq\bar{\Upsilon}$ and $\Psi^\tau\leq\bar{\Psi}$ for all $\tau$. With the penalty parameters frozen, the accuracy cap in \eqref{eq:globalEpsilonMin} obeys the iterate-independent lower bound
\begin{align}
\epsilon^\tau_{\rm max}
=&\;
\frac{(0.5-\beta)\eta_2^\star}
{(1+\eta_1^\star+\eta_2^\star)(\Psi^\tau\Upsilon^\tau)^2}
\;  \nonumber \\\geq&\; 
\frac{(0.5-\beta)\eta_2^\star}
{(1+\eta_1^\star+\eta_2^\star)(\bar{\Psi}\bar{\Upsilon})^2}
\eqqcolon
\epsilon_{\rm cap}
>0,
\quad \forall\,\tau\geq\bar\tau . 
\label{eq:epsilon-cap-floor}
\end{align}
For $\tau\geq\bar\tau$, the descent condition never fails again, so the only remaining change to $\epsilon^\tau$ is the reduction in line~8 of Algorithm~\ref{alg:aotd-compact}: if $\epsilon^\tau>\epsilon^\tau_{\rm max}$, then $\epsilon^\tau\gets\epsilon^\tau_{\rm max}/\nu$. Each such reduction places $\epsilon^\tau$ a factor $\nu$ strictly below the current cap. Suppose reductions occur at outer iterations $\tau_1<\tau_2<\cdots$, all satisfying $\tau_j\geq\bar\tau$. Between two consecutive reductions, $\epsilon^\tau$ remains unchanged, so a reduction at $\tau_{j+1}$ requires
\begin{align*}
\epsilon_{\rm max}^{\tau_j}/\nu
=
\epsilon^{\tau_{j+1}-1}
>
\epsilon_{\rm max}^{\tau_{j+1}}
\qquad\Longrightarrow\qquad
\epsilon_{\rm max}^{\tau_{j+1}}
<
\epsilon_{\rm max}^{\tau_j}/\nu .
\end{align*}
Iterating gives $\epsilon_{\rm max}^{\tau_j}<\epsilon_{\rm max}^{\tau_1}/\nu^{j-1}$. Combining this inequality with the uniform floor $\epsilon_{\rm max}^{\tau_j}\geq\epsilon_{\rm cap}$ from \eqref{eq:epsilon-cap-floor} yields $\nu^{j-1}<\epsilon_{\rm max}^{\tau_1}/\epsilon_{\rm cap}$. Hence, only finitely many cap reductions can occur. After the last reduction, $\epsilon^\tau$ remains constant. Hence, by enlarging $\bar\tau$ to include these finitely many reductions gives $\epsilon^\tau=\epsilon^{\bar\tau}>0$ for all $\tau\geq\bar\tau$, where positivity follows because $\epsilon^0>0$ and every update takes the minimum of two strictly positive quantities. Finally, $\Psi^\tau\geq20$ and $\Upsilon^\tau\geq1$ imply $(\Psi^\tau\Upsilon^\tau)^2\geq400$, so $0<\epsilon^{\bar\tau}\leq\epsilon_{\rm max}^{\bar\tau}<0.5/400<1$. Thus, $\epsilon^{\bar\tau}\in(0,1)$, establishing the first assertion of \eqref{eq:stability-local-bounds} and providing the fixed positive value used in part (b).

\noindent (b) Inner loop stabilization

Throughout this part, a \emph{post-solve accuracy failure} means a failure of \eqref{eq:globalResidualBound} after the local systems have been solved to satisfy \eqref{eq:local-residual-cond} and their directions have been composed. The evaluation at the zero direction at the beginning of each outer iteration merely forces at least one local solve and is not counted as a post-solve accuracy failure. We suppress the within-iteration local-solve pass index on $b_j^\tau$, $\epsilon_j^\tau$, and $\mathcal E^\tau$ below to avoid cumbersome notation.

Define the global residual term
\begin{align*}
\mathcal{E}^\tau \;\coloneq\; C_1\Big(\sum_{j=0}^{M-1}\rho^{\,2b_j^\tau}\Big)^{\frac{1}{2}} + C_2\Big(\sum_{j=0}^{M-1}(\epsilon_j^\tau)^2\Big)^{\frac{1}{2}}\,,
\end{align*}
the bracket of \eqref{eq:lemma2-per-subproblem-residual}, with $C_1, C_2 > 0$ and $\rho \in (0,1)$ the constants of Lemma~\ref{lem:global-residual-error}, where, per Lemma~\ref{lem:global-residual-error}, the $j$th boundary term is omitted when subproblem $j$'s window covers the full horizon. Since \eqref{eq:subproblem-overlap-update} never decreases any $b_j^\tau$ and \eqref{eq:subproblem-error-update} never increases any $\epsilon_j^\tau$, $\mathcal{E}^\tau$ is non-increasing across post-solve failures that trigger these updates.

\emph{Step 1 (sufficient local thresholds).} Suppose a post-solve accuracy check fails at $\tau \geq \bar\tau$. If $\nabla\mathcal L^\tau = \bm 0$, the current iterate is already a KKT point, so consider $\nabla\mathcal L^\tau \neq \bm 0$. Failure of the test means
\begin{equation}
\|\bm r\| \;>\; \frac{\epsilon^\tau \|\nabla\mathcal L^\tau\|}{\|\Gamma^\tau\|\,\Psi^\tau}\,. \label{eq:lemma4-failure-lower}
\end{equation}
On the other hand, the exact Newton step satisfies $(\Delta\bm z^\tau, \Delta\bm\lambda^\tau) = -(\Gamma^\tau)^{-1}\nabla\mathcal L^\tau$, so the uniform bound on the KKT inverse gives $\|(\Delta \bm z^\tau, \Delta \bm\lambda^\tau)\| \leq \|(\Gamma^\tau)^{-1}\|\,\|\nabla\mathcal L^\tau\| \leq C_{\Gamma^{-1}}\|\nabla\mathcal L^\tau\|$, and combining this with \eqref{eq:lemma2-per-subproblem-residual} yields the opposing upper bound
\begin{align}
\|\bm r\| \;\leq&\; \|\Gamma^\tau\|\, \mathcal{E}^\tau \|(\Delta \bm z^\tau, \Delta \bm\lambda^\tau)\| \;  \nonumber \\
\leq&\; \|\Gamma^\tau\|\, C_{\Gamma^{-1}}\, \|\nabla\mathcal L^\tau\|\; \mathcal{E}^\tau\,. 
\label{eq:lemma4-failure-upper}
\end{align}
Chaining \eqref{eq:lemma4-failure-lower} with \eqref{eq:lemma4-failure-upper} and dividing through by $\|\nabla\mathcal L^\tau\| > 0$ isolates the potential: 
\begin{equation*}
\mathcal{E}^\tau\allowbreak  > \epsilon^\tau/(\|\Gamma^\tau\|^2 C_{\Gamma^{-1}} \Psi^\tau).
\end{equation*}
Finally, part (a) supplies $\epsilon^\tau=\epsilon^{\bar\tau}>0$ for all $\tau\geq\bar\tau$, and Assumptions~\ref{assumption:uniform-boundedness}--\ref{assumption:compactness} supply the uniform bounds $\|\Gamma^\tau\| \leq C_\Gamma$ and $\Psi^\tau \leq \bar\Psi$, so every post-solve accuracy failure forces
\begin{equation}
\mathcal{E}^\tau  > \frac{\epsilon^\tau}{\|\Gamma^\tau\|^{2}\, C_{\Gamma^{-1}}\, \Psi^\tau} \geq \frac{\epsilon^{\bar\tau}}{C_\Gamma^2 C_{\Gamma^{-1}} \bar\Psi} =: \delta_1 >  0, \label{eq:lemma4-blocking-floor}
\end{equation}
a positive threshold independent of $\tau$. Consequently, there exist \emph{iterate-independent} local thresholds
\begin{align}
b^\star \;\coloneq\;& \max\left\{1,\left\lceil \log_{1/\rho}\frac{2\sqrt M\,C_1}{\delta_1}\right\rceil\right\}\,, \notag
\\
\epsilon^\star  \coloneq\;& \frac{\delta_1}{2\sqrt M\, C_2}\,, 
\label{eq:lemma4-targets}
\end{align}
such that if at some $\tau \geq \bar\tau$ \emph{every} subproblem satisfies $b_j^\tau \geq b^\star$ and $\epsilon_j^\tau \leq \epsilon^\star$, then $\mathcal{E}^\tau \leq \delta_1$. Indeed, each $\ell_2$ sum in $\mathcal{E}^\tau$ is at most $\sqrt M$ times its largest term, and every subproblem satisfies $C_1\rho^{\,b_j^\tau} \leq C_1 \rho^{\,b^\star} \leq \delta_1/(2\sqrt M)$ (the boundary term of any subproblem whose window covers the full horizon being omitted outright), so $C_1\big(\sum_j \rho^{\,2b_j^\tau}\big)^{1/2} \leq \delta_1/2$ and
\begin{align*}
\mathcal{E}^\tau \;\leq\; \frac{\delta_1}{2} + \sqrt M\, C_2\,\epsilon^\star \;\leq\; \frac{\delta_1}{2} + \frac{\delta_1}{2} \;=\; \delta_1\,,
\end{align*}
and hence, since $\mathcal{E}^\tau$ is non-increasing, by \eqref{eq:lemma4-blocking-floor} every subsequent post-solve accuracy check passes.

\emph{Step 2 (monotone progress reaches the thresholds in finitely many post-solve failures).} Exactly as in the proof of Lemma~\ref{lem:global-adaptive-well-posed}, on every post-solve failing pass \emph{every} subproblem contracts its tolerance by at least the fixed factor $e^{-\varrho_\epsilon/\sqrt M}$ (the floor in \eqref{eq:subproblem-error-update}) and enlarges its nominal overlap by at least one stage (the $\max$ in \eqref{eq:subproblem-overlap-update}). Hence the tolerance threshold in \eqref{eq:lemma4-targets} is met by \emph{every} subproblem after at most
\begin{equation*}
k^{\epsilon} \;\coloneq\; \max\left\{0,\; \max_i \left\lceil \frac{\sqrt M}{\varrho_\epsilon}\,\ln\frac{\epsilon_i^{\bar\tau}}{\epsilon^\star}\right\rceil\right\}
\end{equation*}
post-solve failing passes, and the overlap threshold is met by \emph{every} subproblem after at most $k^{b} \coloneq \max\{0,\; b^\star - \min_j b_j^{\bar\tau}\} \leq b^\star$ post-solve failing passes, where the outer $\max\{0, \cdot\}$ covers the case in which a threshold is already met on entry. Both budgets are independent of $\tau$ because $\delta_1$ is. Therefore, the post-solve accuracy check can fail and trigger \eqref{eq:subproblem-error-update}--\eqref{eq:subproblem-overlap-update} on at most $\max\{k^{\epsilon},\, k^{b}\} < \infty$ passes over all $\tau \geq \bar\tau$. After the last such failure, the first post-solve accuracy check at every subsequent outer iteration passes, so the local parameters are never updated again. Enlarging $\bar\tau$ to absorb these finitely many post-solve failures gives $(\epsilon_i^\tau, b_i^\tau) = (\epsilon_i^{\bar\tau}, b_i^{\bar\tau})$ for all $i$ and $\tau \geq \bar\tau$.

Finally, each local tolerance starts positive and undergoes only finitely many multiplicative updates by positive factors; hence $\epsilon_i^{\bar\tau}>0$, completing \eqref{eq:stability-local-bounds}. Although the nominal overlaps may grow before freezing at the finite values $b_i^{\bar\tau}$, the effective overlaps are trivially bounded by the horizon ($b_i^{L,\tau} \leq n_i$ and $b_i^{R,\tau} \leq N - n_{i+1}$ by the clipping in \eqref{eq:effective-overlaps}), so every window is capped at $[0, N]$.
\end{proof}

\subsection{Proof of Theorem \ref{thm:global}}

\begin{proof}
We will use the shorthand notation $\mathcal{L}_{\bm{\eta}^\tau}^{\tau+1}:= \mathcal{L}_{\bm{\eta}^\tau}(\bm{z}^\tau + \alpha^\tau \Delta \tilde{\bm{z}}^\tau, \bm{\lambda}^\tau + \alpha^\tau \Delta \tilde{\bm{\lambda}}^\tau)$ and likewise $\mathcal{L}_{\bm{\eta}^\tau}^\tau = \mathcal{L}_{\bm{\eta}^\tau}(\bm{z}^\tau, \bm{\lambda}^\tau)$. For the fixed-horizon problem, Assumption \ref{assumption:compactness} implies that there exists a finite constant $\Upsilon_{\bm{\eta}} > 0$ such that $\sup_{\mathcal{Z} \times \Lambda} \|\nabla^2 \mathcal{L}_{\bm{\eta}}(\bm{z}, \bm{\lambda})\| \leq \Upsilon_{\bm{\eta}}$. Applying Taylor's theorem to the updates yields
\begin{align*}
&\mathcal{L}_{\bm{\eta}^\tau}^{\tau + 1} \leq\; \mathcal{L}_{\bm{\eta}^\tau}^\tau + \alpha^\tau \begin{pmatrix}
\nabla_{\bm z}\mathcal{L}_{\bm{\eta}^\tau}^\tau \\ \nabla_{\bm \lambda} \mathcal{L}_{\bm{\eta}^\tau}^\tau 
\end{pmatrix}^T \begin{pmatrix}
\Delta \tilde{\bm{z}} ^\tau\\ \Delta \tilde{\bm{\lambda}}^\tau
\end{pmatrix} \\
&\quad + \frac{(\alpha^\tau)^2 \Upsilon_{\bm{\eta}}}{2} \left\|\begin{pmatrix}
\Delta \tilde{\bm{z}}^\tau \\\Delta \tilde{\bm{\lambda}}^\tau
\end{pmatrix} \right\|^2 \,. 
\end{align*}
Using the \eqref{eq:lemma3-residual-final}, we have that 
\begin{equation}
\hskip-0.2cm\left\|\begin{pmatrix}
\Delta \tilde{\bm{z}} \\\Delta \tilde{\bm{\lambda}}
\end{pmatrix} \right\| \leq (\epsilon^\tau + 1) \left\|\begin{pmatrix}
\Delta \bm{z}^\tau \\ \Delta \bm{\lambda}^\tau
\end{pmatrix}\right\| \leq 2\left\|\begin{pmatrix}
\Delta \bm{z}^\tau \\ \Delta \bm{\lambda}^\tau
\end{pmatrix}\right\| \,. \label{eq:perturbed-bound}
\end{equation}
Hence, from here, we have 
\begin{align*}
\mathcal{L}_{\bm{\eta}^\tau}^{\tau + 1} \leq&\; \mathcal{L}_{\bm{\eta}^\tau}^\tau + \alpha^\tau \begin{pmatrix}
\nabla_{\bm z}\mathcal{L}_{\bm{\eta}^\tau}^\tau \\ \nabla_{\bm \lambda} \mathcal{L}_{\bm{\eta}^\tau}^\tau
\end{pmatrix}^T \begin{pmatrix}
\Delta \tilde{\bm{z}}^\tau \\ \Delta \tilde{\bm{\lambda}}^\tau
\end{pmatrix} \\
& + 2(\alpha^\tau)^2 \Upsilon_{\bm{\eta}} \left\|\begin{pmatrix}
\Delta\bm{z}^\tau \\\Delta \bm{\lambda}^\tau
\end{pmatrix} \right\|^2  \\
\leq&\; \mathcal{L}_{\bm{\eta}^\tau}^\tau + \alpha^\tau \begin{pmatrix}
\nabla_{\bm z}\mathcal{L}_{\bm{\eta}^\tau}^\tau \\ \nabla_{\bm \lambda} \mathcal{L}_{\bm{\eta}^\tau}^\tau 
\end{pmatrix}^T \begin{pmatrix}
\Delta \tilde{\bm{z}}^\tau \\ \Delta \tilde{\bm{\lambda}}^\tau
\end{pmatrix} \\
& + 2(\alpha^\tau)^2 \Upsilon_{\bm{\eta}} C_{\Gamma^{-1}}^2 \left\|\nabla \mathcal{L}^\tau\right\|^2\,. 
\end{align*}
Applying the descent direction condition \eqref{eq:descent-direction-condition} yields
\begin{align*}
\mathcal{L}_{\bm{\eta}^\tau}^{\tau + 1} \leq&\; \mathcal{L}_{\bm{\eta}^\tau}^\tau + \alpha^\tau \begin{pmatrix}
\nabla_{\bm z}\mathcal{L}_{\bm{\eta}^\tau}^\tau \\ \nabla_{\bm \lambda} \mathcal{L}_{\bm{\eta}^\tau}^\tau 
\end{pmatrix}^T \begin{pmatrix}
\Delta \tilde{\bm{z}}^\tau \\ \Delta \tilde{\bm{\lambda}}^\tau
\end{pmatrix} \\
& - \frac{4(\alpha^\tau)^2 \Upsilon_{\bm{\eta}} C_{\Gamma^{-1}}^2 }{\eta_2^\tau} \begin{pmatrix}
\nabla_{\bm z}\mathcal{L}_{\bm{\eta}^\tau}^\tau \\ \nabla_{\bm \lambda} \mathcal{L}_{\bm{\eta}^\tau}^\tau 
\end{pmatrix}^T \begin{pmatrix}
\Delta \tilde{\bm{z}} \\ \Delta \tilde{\bm{\lambda}}
\end{pmatrix}  \\
\leq&\; \mathcal{L}_{\bm{\eta}^\tau}^\tau + \alpha^\tau \begin{pmatrix}
\nabla_{\bm z}\mathcal{L}_{\bm{\eta}^\tau}^\tau \\ \nabla_{\bm \lambda} \mathcal{L}_{\bm{\eta}^\tau}^\tau 
\end{pmatrix}^T \begin{pmatrix}
\Delta \tilde{\bm{z}}^\tau \\ \Delta \tilde{\bm{\lambda}}^\tau
\end{pmatrix} \\
&- \frac{4(\alpha^\tau)^2 \Upsilon_{\bm{\eta}} C_{\Gamma^{-1}}^2 }{\eta_2^{\bar{\tau}}} \begin{pmatrix}
\nabla_{\bm z}\mathcal{L}_{\bm{\eta}^\tau}^\tau \\ \nabla_{\bm \lambda} \mathcal{L}_{\bm{\eta}^\tau}^\tau 
\end{pmatrix}^T \begin{pmatrix}
\Delta \tilde{\bm{z}} \\ \Delta \tilde{\bm{\lambda}}
\end{pmatrix}  \\
\leq&\; \mathcal{L}_{\bm{\eta}^\tau}^\tau  + \alpha^\tau\left(1- \frac{4\alpha^\tau \Upsilon_{\bm{\eta}} C_{\Gamma^{-1}}^2}{\eta_2^{\bar{\tau}}}\right) \nonumber \\ & \times \begin{pmatrix}
\nabla_{\bm z}\mathcal{L}_{\bm{\eta}^\tau}^\tau \\ \nabla_{\bm \lambda} \mathcal{L}_{\bm{\eta}^\tau}^\tau 
\end{pmatrix}^T \begin{pmatrix}
\Delta \tilde{\bm{z}}^\tau \\ \Delta \tilde{\bm{\lambda}}^\tau
\end{pmatrix}\,. 
\end{align*}
Choosing $\alpha^\tau$ such that
\begin{align*}
1-\frac{4 \Upsilon_{\bm{\eta}} C_{\Gamma^{-1}}^2}{\eta_2^{\bar{\tau}}}\alpha^\tau \geq \beta \Leftrightarrow \alpha^\tau \leq \frac{(1-\beta)\eta_2^{\bar{\tau}}}{4\Upsilon_{\bm{\eta}} C_{\Gamma^{-1}}^2}\,,  
\end{align*}
ensures the line search condition \eqref{eq:line-search-cond} is satisfied. \newpage 
Since backtracking starts at $1$ and contracts the trial step by $\delta$, it accepts a step satisfying $\alpha^\tau \geq \alpha_{\rm min}$ for every $\tau$, where
\begin{equation*}
\alpha_{\rm min} \;\coloneq\; \delta\min\left\{1,\frac{(1-\beta)\eta_2^{\bar{\tau}}}{4\Upsilon_{\bm{\eta}} C_{\Gamma^{-1}}^2}\right\} > 0\,.
\end{equation*}
This proves the first statement of the theorem.

Using the Armijo condition \eqref{eq:line-search-cond} and the descent condition \eqref{eq:descent-direction-condition}, we have that
\begin{align*}
\mathcal{L}_{\bm{\eta}^\tau}^{\tau + 1} - \mathcal{L}_{\bm{\eta}^\tau}^\tau \leq&\; \alpha_{\rm min} \beta \begin{pmatrix}
\nabla_{\bm z}\mathcal{L}_{\bm{\eta}^\tau}^\tau \\ \nabla_{\bm \lambda} \mathcal{L}_{\bm{\eta}^\tau}^\tau 
\end{pmatrix}^T \begin{pmatrix}
\Delta \tilde{\bm{z}}^\tau \\ \Delta \tilde{\bm{\lambda}}^\tau
\end{pmatrix} \\
\leq & 
-\alpha_{\rm min}\beta \eta_2^{\bar{\tau}} \|\nabla \mathcal{L}^\tau\|^2 / 2\,. 
\end{align*}
Finally, summing over $\tau \geq \bar{\tau}$, we have that 
\begin{align*}
\sum_{\tau = \bar{\tau}}^\infty \|\nabla \mathcal{L}^\tau\|^2 \leq \frac{2}{\eta_2^{\bar{\tau}} \alpha_{\rm min}\beta} \left(\mathcal{L}_{\bm{\eta}^{\bar{\tau}}}^{\bar{\tau}} - \min_{\mathcal Z \times \Lambda }\mathcal{L}_{\bm{\eta}^{\bar{\tau}}}(\bm{z},\bm{\lambda})\right) < \infty \,.
\end{align*}
This implies that $\|\nabla \mathcal{L}^\tau\| \to 0$ as $\tau \to \infty$. 
\end{proof}

\setcounter{equation}{0}
\renewcommand{\theequation}{B.\arabic{equation}}

\section{Proofs in Section \ref{sec:local-convergence}}\label{appendix:proofs2}

\subsection{Proof of Lemma \ref{lem:unit-step-size}}

\begin{proof}
As in the proof of Theorem \ref{thm:global}, we will use the shorthand notation $\mathcal{L}_{\bm{\eta}^\tau}^{\tau+1}:= \mathcal{L}_{\bm{\eta}^\tau}(\bm{z}^\tau + \alpha^\tau \Delta \tilde{\bm{z}}^\tau, \bm{\lambda}^\tau + \alpha^\tau \Delta \tilde{\bm{\lambda}}^\tau)$ and likewise $\mathcal{L}_{\bm{\eta}^\tau}^\tau = \mathcal{L}_{\bm{\eta}^\tau}(\bm{z}^\tau, \bm{\lambda}^\tau)$. Define
\begin{align*}
&\bm s^\tau \coloneqq
\begin{pmatrix}\Delta\bm z^\tau\\ \Delta\bm\lambda^\tau\end{pmatrix}, 
\tilde{\bm s}^\tau \coloneqq
\begin{pmatrix}\Delta\tilde{\bm z}^\tau\\ \Delta\tilde{\bm\lambda}^\tau\end{pmatrix}, 
\bm e^\tau \coloneqq \tilde{\bm s}^\tau-\bm s^\tau,  \\
&\bm g_{\bm\eta}^\tau \coloneqq
\begin{pmatrix}
\nabla_{\bm z}\mathcal L_{\bm\eta^\tau}^\tau\\
\nabla_{\bm\lambda}\mathcal L_{\bm\eta^\tau}^\tau
\end{pmatrix}, 
P^\tau \coloneqq
\begin{pmatrix}
I+\eta_2^\tau H^\tau & \eta_1^\tau(G^\tau)^T\\
\eta_2^\tau G^\tau & I
\end{pmatrix}. 
\end{align*}
We need to show that, for large enough $\tau$
\begin{align*}
\mathcal{L}_{\bm{\eta}^\tau}^{\tau+1}
\leq \mathcal{L}_{\bm{\eta}^\tau}^\tau
+ \beta (\bm g_{\bm\eta}^\tau)^T\tilde{\bm s}^\tau.
\end{align*}
Using the thrice differentiability in Assumption \ref{assumption:compactness}, we have that $\nabla^2\mathcal{L}_{\bm{\eta}}$ is continuous and hence 
\begin{equation}
\mathcal{L}_{\bm{\eta}^\tau}^{\tau + 1}
\leq \mathcal{L}_{\bm{\eta}^\tau}^\tau
+(\bm g_{\bm\eta}^\tau)^T\tilde{\bm s}^\tau
+\frac{1}{2}(\tilde{\bm s}^\tau)^T
\nabla^2\mathcal{L}_{\bm{\eta}^\tau}^\tau\tilde{\bm s}^\tau
+o(\|\tilde{\bm s}^\tau\|^2). \label{eq:lemma5-taylor-1}
\end{equation}
Now, establishing notation, for any function $a(x): \mathbb{R}^n \to \mathbb{R}^{m_1 \times m_2}$ and any vector $b \in \mathbb{R}^{m_1}$, let the inner product be given by 
\begin{equation*}
\langle \nabla a(x), b \rangle\allowbreak  := \nabla^T(a^T(x)b)\allowbreak  = \sum_{j=1}^{m_1}b_j \nabla^Ta_j(x) \allowbreak \in \mathbb{R}^{m_2 \times n}
\end{equation*}
for $a^T(x) = (a_1(x), \cdots, a_{m_1}(x))$. Then, by direct calculation, we have
\begin{align*}
\nabla^2_{\bm{z}}\mathcal{L}_{\bm{\eta}^\tau} & =\; H^\tau + \eta_1^\tau (G^\tau)^TG^\tau + \eta_2^\tau (H^\tau)^2 \\
&\qquad + \eta_1^\tau \langle \nabla_{\bm{z}}G^\tau, f^\tau \rangle + \eta_2^\tau \langle \nabla_{\bm{z}}H^\tau, \nabla_{\bm z}\mathcal{L}^\tau\rangle \,,  \\
\nabla_{\bm{z\lambda}}\mathcal{L}_{\bm{\eta}^\tau} & =\; (G^\tau)^T + \eta_2^\tau ( H^\tau(G^\tau)^T + \langle \nabla_{\bm{\lambda}}H^\tau, \nabla_{\bm{z}}\mathcal{L}^\tau \rangle ) \,,  \\
\nabla_{\bm{\lambda}}^2 \mathcal{L}_{\bm{\eta}^\tau} & =\; \eta_2^\tau G^\tau (G^\tau)^T\,. 
\end{align*}
Hence, define
\begin{equation*}
\mathcal{H}^\tau := \begin{pmatrix}
H^\tau + \eta_1^\tau (G^\tau)^TG^\tau + \eta_2^\tau (H^\tau)^2 & (I+\eta_2^\tau H^\tau)(G^\tau)^T \\ G^\tau (I+\eta_2^\tau H^\tau) & \eta_2^\tau G^\tau (G^\tau)^T
\end{pmatrix}. 
\end{equation*}
Using the fact $\|\nabla \mathcal{L}^\tau\| = \|(\nabla_{\bm{z}}\mathcal{L}^\tau, f^\tau)\| \to 0$, we get 
\begin{equation}\label{eq:lemma5-hessian-vanishing}
\|\nabla^2 \mathcal{L}_{\bm{\eta}^\tau}^\tau - \mathcal{H}^\tau\| = o(1)\,. \end{equation}
The augmented-Lagrangian gradient can be written compactly as
\begin{equation*}
\bm g_{\bm\eta}^\tau
=P^\tau\nabla\mathcal L^\tau
=-P^\tau\Gamma^\tau\bm s^\tau.
\end{equation*}
Using \eqref{eq:lemma5-taylor-1}, \eqref{eq:lemma5-hessian-vanishing}, and $\bm e^\tau=\tilde{\bm s}^\tau-\bm s^\tau$, we obtain
\begin{align*}
\mathcal L_{\bm\eta^\tau}^{\tau+1}
\leq{}& \mathcal L_{\bm\eta^\tau}^{\tau}
+\frac12(\bm g_{\bm\eta}^\tau)^T\tilde{\bm s}^\tau
+\frac12(\tilde{\bm s}^\tau)^T\mathcal H^\tau\bm e^\tau \nonumber \\
&+\frac12(\tilde{\bm s}^\tau)^T
(\mathcal H^\tau-P^\tau\Gamma^\tau)\bm s^\tau
+o(\|\tilde{\bm s}^\tau\|^2).
\end{align*}
Direct multiplication gives
\begin{align*}
\mathcal H^\tau-P^\tau\Gamma^\tau
=
\begin{pmatrix}
(I+\eta_2^\tau H^\tau)(H^\tau-\hat H^\tau) & \bm0\\
\eta_2^\tau G^\tau(H^\tau-\hat H^\tau) & \bm0
\end{pmatrix}.
\end{align*}
We first bound the last term. From \eqref{eq:perturbed-bound}, $\|\tilde{\bm s}^\tau\|\leq2\|\bm s^\tau\|$. Additionally, $\|H^\tau\| \allowbreak \leq \Upsilon_{\rm Upper}$ by Assumption \ref{assumption:compactness}, while $\|G^\tau\| \leq 1+2\Upsilon_{\rm Upper}$ by Assumption \ref{assumption:uniform-boundedness} and \cite[Lemma 5.1]{naFOTD}. Hence, we obtain by Assumption \ref{assumption:modified-hessian-vanishes}
\begin{align*}
\bigg|\frac12(&\tilde{\bm s}^\tau)^T
(\mathcal H^\tau-P^\tau\Gamma^\tau)\bm s^\tau\bigg| \\
\quad \leq&\|\bm s^\tau\|\,\|H^\tau-\hat H^\tau\|\,
\|\Delta\bm z^\tau\|
\bigl(1+\eta_2^0(1+3\Upsilon_{\rm Upper})\bigr)  \\
\leq&\; o(\|\bm s^\tau\|^2). 
\end{align*}
To bound the cross term, we can again apply \eqref{eq:perturbed-bound} along with \eqref{eq:lemma3-residual-final} yielding
\begin{equation*}
\left|\frac12(\tilde{\bm s}^\tau)^T\mathcal H^\tau\bm e^\tau\right|
\leq \epsilon^\tau\|\mathcal H^\tau\|\,\|\bm s^\tau\|^2.
\end{equation*}
To bound $\|\mathcal{H}^\tau\|$, note first that while 
\begin{equation*}
\Upsilon^\tau = \max\{\|G^\tau\|, \|\hat H^\tau\|, 1\}
\end{equation*}
is defined through the modified Hessian $\hat H^\tau$, Assumption \ref{assumption:modified-hessian-vanishes} gives $\|H^\tau\| \allowbreak \leq \|\hat H^\tau\| + \|H^\tau - \hat H^\tau\| \allowbreak \leq \Upsilon^\tau + o(1)$, so that both $\|H^\tau\|$ and $\|G^\tau\|$ are bounded by $\Upsilon^\tau$ up to an $o(1)$ term that is absorbed into the remainder for sufficiently large $\tau$. Using this, we have that
\begin{align*}
\|H^\tau + \eta_1^\tau (G^\tau)^TG^\tau + \eta_2^\tau (H^\tau)^2 \| \leq&\;\Upsilon^\tau+(\eta_1^\tau+\eta_2^\tau) (\Upsilon^\tau)^2  \\
\|(I+\eta_2^\tau H^\tau)(G^\tau)^T\| \leq&\; (1+\eta_2^\tau \Upsilon^\tau)\Upsilon^\tau \\
\|\eta_2^\tau G^\tau (G^\tau)^T \|\leq&\; \eta_2^\tau( \Upsilon^\tau)^2\,.
\end{align*}
Hence, we obtain
\begin{align*}
\|\mathcal{H}^\tau\| \leq 3 \Upsilon^\tau + (\eta_1^\tau + 4 \eta_2^\tau )(\Upsilon^\tau)^2\,. 
\end{align*}
This results in 
\begin{align}
& \hskip-0.2cm\mathcal{L}_{\bm{\eta}^\tau}^{\tau+1}
\leq \mathcal{L}_{\bm{\eta}^\tau}^\tau
+\frac12(\bm g_{\bm\eta}^\tau)^T\tilde{\bm s}^\tau  \nonumber \\
& \qquad +4\epsilon^\tau(\Upsilon^\tau)^2
(1+\eta_1^\tau+\eta_2^\tau)\|\bm s^\tau\|^2
+o(\|\bm s^\tau\|^2). 
\label{eq:lemma5-intermediate}
\end{align}
We now can invoke the fact that $\epsilon^\tau \leq \epsilon^\tau_{\rm max}$. Using this, we have that 
\begin{align}
4.05\epsilon^\tau&(\Upsilon^\tau)^2 (1+\eta_1^\tau + \eta_2^\tau) \|\bm s^\tau\|^2  \nonumber \\ 
\leq&\; \frac{4.05(0.5-\beta)\eta_2^\tau }{(\Psi^\tau)^2} \|\bm s^\tau\|^2   \nonumber \\
\leq&\;  \frac{8.1(0.5-\beta)\eta_2^\tau}{2(\Psi^\tau)^2}\|(\Gamma^\tau)^{-1}\|^2 \|\nabla \mathcal{L}^\tau\|^2   \nonumber \\
\leq&\; \frac{(0.5-\beta)\eta_2^\tau}{2} \|\nabla \mathcal{L}^\tau\|^2   \nonumber \\
\leq&\;-(0.5-\beta)(\bm g_{\bm\eta}^\tau)^T\tilde{\bm s}^\tau\,, 
\label{eq:lemma5-3rd-term-bound}
\end{align}
where the first inequality uses $\epsilon^\tau \leq \epsilon^\tau_{\rm max}$ together with the definition of $\epsilon^\tau_{\rm max}$ in \eqref{eq:globalEpsilonMin} (which cancels the factor $(\Upsilon^\tau)^2(1+\eta_1^\tau+\eta_2^\tau)$). The second inequality uses $\|\bm s^\tau\| \leq \|(\Gamma^\tau)^{-1}\|\,\|\nabla\mathcal{L}^\tau\|$ and the third uses $\|(\Gamma^\tau)^{-1}\| \leq (7/20)\Psi^\tau$ from \eqref{eq:lemma3-kkt-inv-bound}, so that $8.1\,(7/20)^2 < 1$. The last uses the descent direction condition \eqref{eq:descent-direction-condition}. Note that we have replaced the coefficient $4$ in \eqref{eq:lemma5-intermediate} by $4.05$: since $\epsilon^\tau=\epsilon^{\bar\tau}>0$ for all sufficiently large $\tau$ (Lemma \ref{lem:stability}), the surplus $0.05\,\epsilon^\tau(\Upsilon^\tau)^2(1+\eta_1^\tau+\eta_2^\tau)\|\bm s^\tau\|^2$ is a fixed positive multiple of $\|\bm s^\tau\|^2$ and thereby absorbs the $o(\|\bm s^\tau\|^2)$ remainder of \eqref{eq:lemma5-intermediate} for sufficiently large $\tau$. Substituting \eqref{eq:lemma5-3rd-term-bound} into \eqref{eq:lemma5-intermediate} therefore yields
\begin{equation*}
\mathcal{L}_{\bm{\eta}^\tau}^{\tau + 1}
\leq \mathcal{L}_{\bm{\eta}^\tau}^\tau
+\beta(\bm g_{\bm\eta}^\tau)^T\tilde{\bm s}^\tau,
\end{equation*}
which is exactly the line search condition \eqref{eq:line-search-cond} with $\alpha^\tau = 1$. Hence, the unit step size is accepted for all sufficiently large $\tau$, completing the claim.    
\end{proof}

\subsection{Proof of Theorem \ref{thm:local-linear}}

\begin{proof}    
By Lemma~\ref{lem:unit-step-size}, $\alpha^\tau=1$ for all sufficiently large $\tau$. For such $\tau$, the updates satisfy $\bm{z}^{\tau+1}=\bm{z}^\tau+\Delta\tilde{\bm{z}}^\tau$ and $\bm{\lambda}^{\tau+1}=\bm{\lambda}^\tau+\Delta\tilde{\bm{\lambda}}^\tau$. We therefore decompose the error at the next iterate as
\begin{align}
\begin{pmatrix}
\bm{z}^\tau + \Delta \tilde{\bm{z}}^\tau - \bm{z}^\ast \\ \bm{\lambda}^\tau + \Delta \tilde{\bm{\lambda}}^\tau - \bm{\lambda}^\ast
\end{pmatrix} 
& = \begin{pmatrix}
\bm{z}^\tau + \Delta \bm{z}^\tau - \bm{z}^\ast \\
\bm{\lambda}^\tau + \Delta \bm{\lambda}^\tau - \bm{\lambda}^\ast
\end{pmatrix}  \nonumber \\
& \quad + 
\begin{pmatrix}
\Delta \tilde{\bm{z}}^\tau - \Delta \bm{z}^\tau \\ \Delta \tilde{\bm{\lambda}}^\tau - \Delta \bm{\lambda}^\tau
\end{pmatrix}\,. 
\label{eq:thm2-expansion}
\end{align}
Using the definition of the Newton system, we have the first term satisfies
\begin{align}
&  \begin{pmatrix}
\bm{z}^\tau + \Delta \bm{z}^\tau - \bm{z}^\ast \\
\bm{\lambda}^\tau + \Delta \bm{\lambda}^\tau - \bm{\lambda}^\ast
\end{pmatrix}  \nonumber\\
&\; = (\Gamma^\tau)^{-1} \Gamma^\tau \begin{pmatrix}
\bm{z}^\tau - \bm{z}^\ast \\ \bm{\lambda}^\tau - \bm{\lambda}^\ast
\end{pmatrix} + \begin{pmatrix}\Delta \bm{z}^\tau \\ \Delta \bm{\lambda}^\tau \end{pmatrix}   \nonumber \\
&\;  =  (\Gamma^\tau)^{-1} \Gamma^\tau \begin{pmatrix}
\bm{z}^\tau - \bm{z}^\ast \\ \bm{\lambda}^\tau - \bm{\lambda}^\ast
\end{pmatrix} - (\Gamma^\tau)^{-1} \nabla \mathcal{L}^\tau   \nonumber \\
&  \;= (\Gamma^\tau)^{-1} \left(\Gamma^\tau \begin{pmatrix}
\bm{z}^\tau - \bm{z}^\ast \\ \bm{\lambda}^\tau - \bm{\lambda}^\ast
\end{pmatrix}  - \nabla \mathcal{L}^\tau \right)   \nonumber \\
&  \;'= (\Gamma^\tau)^{-1}  \Bigg(\Gamma^\tau \begin{pmatrix}
\bm{z}^\tau - \bm{z}^\ast \\ \bm{\lambda}^\tau - \bm{\lambda}^\ast
\end{pmatrix} - (\nabla \mathcal{L}^\tau - \nabla \mathcal{L}(\bm{z}^\ast, \bm{\lambda}^\ast)) \Bigg). 
\label{eq:thm2-term1}
\end{align}
Now, define the matrices
\begin{align*}
H^\tau(t) :=&\; H(\bm{z}^\tau + t(\bm{z}^\ast -\bm{z}^\tau), \bm{\lambda}^\tau + t(\bm{\lambda}^\ast - \bm{\lambda}^\tau))\,,  \\ 
G^\tau(t) :=&\; G(\bm{z}^\tau + t(\bm{z}^\ast -\bm{z}^\tau)) \,. 
\end{align*}
For large enough $\tau$, convergence of the iterates and openness of $\mathcal U$ ensure that the line segment joining $(\bm z^\tau,\bm\lambda^\tau)$ and $(\bm z^\ast,\bm\lambda^\ast)$ lies in $\mathcal U$. Therefore, Taylor's \\  theorem gives
\begin{align}
\nabla \mathcal{L}^{\tau}& -\nabla \mathcal{L}(\bm{z}^\ast, \bm{\lambda}^\ast)  \nonumber \\
& =\; \int_0^1 \nabla^2\mathcal{L}(\bm{z}^\tau + t(\bm{z}^\ast - \bm{z}^\tau), \bm{\lambda}^\tau + t(\bm{\lambda}^\ast - \bm{\lambda}^\tau))  \nonumber \\ &\quad  \times \begin{pmatrix}
\bm{z}^\tau - \bm{z}^\ast \\\bm{\lambda}^\tau - \bm{\lambda}^\ast
\end{pmatrix} dt   \nonumber \\
&=\; \int_0^1 \begin{pmatrix}
H^\tau(t) & (G^\tau(t))^T \\ G^\tau(t) & \bm{0} \end{pmatrix}\begin{pmatrix}
\bm{z}^\tau - \bm{z}^\ast \\\bm{\lambda}^\tau - \bm{\lambda}^\ast
\end{pmatrix} dt  \,. 
\label{eq:thm2-taylors}
\end{align}
Hence, combining \eqref{eq:thm2-term1} with \eqref{eq:thm2-taylors} yields
\begin{align*}
&\begin{pmatrix}
\bm{z}^\tau + \Delta \bm{z}^\tau - \bm{z}^\ast \\
\bm{\lambda}^\tau + \Delta \bm{\lambda}^\tau - \bm{\lambda}^\ast
\end{pmatrix} \\
&\quad =\;  (\Gamma^\tau)^{-1} \left( \int_0^1 \begin{pmatrix}
\hat{H}^\tau - H^\tau(t) & (G^\tau)^T - (G^\tau(t))^T \\ G^\tau - G^\tau(t) & \bm{0}
\end{pmatrix} \right. \\ &\qquad \left. \begin{pmatrix}
\bm{z}^\tau - \bm{z}^\ast \\ \bm{\lambda}^\tau - \bm{\lambda}^\ast 
\end{pmatrix} dt \right)\,. 
\end{align*}
Now, using the fact that the KKT inverse is bounded, we obtain
\begin{align}
&\left\|  \begin{pmatrix}
\bm{z}^\tau + \Delta \bm{z}^\tau - \bm{z}^\ast \\
\bm{\lambda}^\tau + \Delta \bm{\lambda}^\tau - \bm{\lambda}^\ast
\end{pmatrix}  \right\|  \nonumber \\
\qquad &\leq\;  \|(\Gamma^\tau)^{-1} \| \left\| \int_0^1 \begin{pmatrix}
\hat{H}^\tau - H^\tau(t) & (G^\tau)^T - (G^\tau(t))^T \\ G^\tau - G^\tau(t) & \bm{0}
\end{pmatrix} \right.  \nonumber \\ &\qquad  \times \left. \begin{pmatrix}
\bm{z}^\tau - \bm{z}^\ast \\ \bm{\lambda}^\tau - \bm{\lambda}^\ast 
\end{pmatrix} dt \right\|   \nonumber \\
\qquad &\leq C_{\Gamma^{-1}} \int_0^1 \left\| \begin{pmatrix}
\hat{H}^\tau - H^\tau(t) & (G^\tau)^T - (G^\tau(t))^T \\ G^\tau - G^\tau(t) & \bm{0}
\end{pmatrix} \right\|  \nonumber \\ &\qquad \times \left \|\begin{pmatrix}
\bm{z}^\tau - \bm{z}^\ast \\ \bm{\lambda}^\tau - \bm{\lambda}^\ast 
\end{pmatrix}  \right\| dt\,. 
\label{eq:thm2-intermediate-bound-2}
\end{align}
Let $C_H$ and $C_G$ denote the Lipschitz constants for the global matrices $H$ and $G$ induced by Assumption~\ref{assumption:local-lipschitz} and their block structure. Using these constants, the triangle inequality, and Assumption \ref{assumption:modified-hessian-vanishes}, we have
\begin{align}
& \Bigg\| \begin{pmatrix}
\hat{H}^\tau - H^\tau(t) & (G^\tau)^T - (G^\tau(t))^T \\ G^\tau - G^\tau(t) & \bm{0}
\end{pmatrix} \Bigg\|  \nonumber \\
& \leq\; \|\hat{H}^\tau - H^\tau(t)\| + 2 \|G^\tau - G^\tau(t)\|   \nonumber \\
&\leq\; \|\hat{H}^\tau - H^\tau\| + \|H^\tau - H^\tau(t)\| + 2 \|G^\tau - G^\tau(t)\|   \nonumber \\
&\leq\; o(1) + C_H t \left \|\begin{pmatrix}
\bm{z}^\tau - \bm{z}^\ast \\ \bm{\lambda}^\tau - \bm{\lambda}^\ast
\end{pmatrix} \right\| + 2C_G t \left\| \bm{z}^\tau - \bm{z}^\ast \right\|   \nonumber \\
&\leq\;o(1) + O\left( \left \|\begin{pmatrix}
\bm{z}^\tau - \bm{z}^\ast \\ \bm{\lambda}^\tau - \bm{\lambda}^\ast
\end{pmatrix} \right\| \right)\,. 
\label{eq:thm2-matrix-bound}
\end{align}
Combining \eqref{eq:thm2-matrix-bound} with \eqref{eq:thm2-intermediate-bound-2} yields
\begin{align}
\bigg\| &  \begin{pmatrix}
\bm{z}^\tau + \Delta \bm{z}^\tau - \bm{z}^\ast \\
\bm{\lambda}^\tau + \Delta \bm{\lambda}^\tau - \bm{\lambda}^\ast
\end{pmatrix}  \bigg\|  \nonumber \\
\leq&\;  C_{\Gamma^{-1}}\left( o(1) +   O\left( \left \|\begin{pmatrix}
\bm{z}^\tau - \bm{z}^\ast \\ \bm{\lambda}^\tau - \bm{\lambda}^\ast
\end{pmatrix} \right\| \right) \right) \left\|\begin{pmatrix}
\bm{z}^\tau - \bm{z}^\ast \\ \bm{\lambda}^\tau - \bm{\lambda}^\ast 
\end{pmatrix} \right\|   \nonumber \\
\leq&\; o\!\left(\left\|\begin{pmatrix}
\bm{z}^\tau - \bm{z}^\ast \\ \bm{\lambda}^\tau - \bm{\lambda}^\ast 
\end{pmatrix} \right\|\right) + O\!\left(\left\|\begin{pmatrix}
\bm{z}^\tau - \bm{z}^\ast \\ \bm{\lambda}^\tau - \bm{\lambda}^\ast 
\end{pmatrix} \right\|^2\right). 
\label{eq:thm2-exact-newton-bound}
\end{align}
Combining \eqref{eq:thm2-exact-newton-bound} with \eqref{eq:thm2-expansion} and \eqref{eq:lemma3-residual-final} yields
\begin{align*}
&\left\| \begin{pmatrix}
\bm{z}^\tau + \Delta \tilde{\bm{z}}^\tau - \bm{z}^\ast \\ \bm{\lambda}^\tau + \Delta \tilde{\bm{\lambda}}^\tau - \bm{\lambda}^\ast
\end{pmatrix} \right\|  \nonumber \\ 
\qquad &\leq\; \left\| \begin{pmatrix}
\bm{z}^\tau + \Delta \bm{z}^\tau - \bm{z}^\ast \\
\bm{\lambda}^\tau + \Delta \bm{\lambda}^\tau - \bm{\lambda}^\ast
\end{pmatrix} \right\| + \left\| 
\begin{pmatrix}
\Delta \tilde{\bm{z}}^\tau - \Delta \bm{z}^\tau \\ \Delta \tilde{\bm{\lambda}}^\tau - \Delta \bm{\lambda}^\tau
\end{pmatrix} \right\|  \\
&\leq\;  \left\| \begin{pmatrix}
\bm{z}^\tau + \Delta \bm{z}^\tau - \bm{z}^\ast \\
\bm{\lambda}^\tau + \Delta \bm{\lambda}^\tau - \bm{\lambda}^\ast
\end{pmatrix} \right\| +\epsilon^\tau  \left\|\begin{pmatrix}
\Delta \bm{z}^\tau \\ \Delta \bm{\lambda}^\tau
\end{pmatrix} \right\|  \\
&\leq\;   \left\| \begin{pmatrix}
\bm{z}^\tau + \Delta \bm{z}^\tau - \bm{z}^\ast \\
\bm{\lambda}^\tau + \Delta \bm{\lambda}^\tau - \bm{\lambda}^\ast
\end{pmatrix} \right\| \\
&\quad +\epsilon^\tau \left( \left\|\begin{pmatrix}
\bm{z}^\tau - \bm{z}^\ast \\ \bm{\lambda}^\tau - \bm{\lambda}^\ast
\end{pmatrix} \right\| + \left\| \begin{pmatrix}
\bm{z}^\tau + \Delta \bm{z}^\tau - \bm{z}^\ast \\ \bm{\lambda}^\tau + \Delta \bm{\lambda}^\tau - \bm{\lambda}^\ast 
\end{pmatrix} \right\| \right)  \\
&\leq\; \epsilon^\tau \left\| \begin{pmatrix}
\bm{z}^\tau - \bm{z}^\ast \\ \bm{\lambda}^\tau - \bm{\lambda}^\ast 
\end{pmatrix} \right\| + o\left(\left\| \begin{pmatrix}
\bm{z}^\tau - \bm{z}^\ast \\ \bm{\lambda}^\tau - \bm{\lambda}^\ast 
\end{pmatrix} \right\| \right)\,. 
\end{align*}
Then, by $\epsilon^\tau = \epsilon^{\bar\tau} < 1$ for all $\tau \geq \bar\tau$ and the exact stabilization given in Lemma \ref{lem:stability} completes the result.
\end{proof}

\section{Numerical Results: Additional Experiments and Details}\label{appendix:experimental_details}

We briefly discuss the baseline implementations and metric calculations used in this study. Note that all wall-clock experiments in this appendix were run on an AMD Ryzen 9 7950X CPU on a single uncontested thread serially. Lastly, we provide the full hyperparameters for reproducibility in Table \ref{tab:experiment-config}.

\subsection{Baseline Implementations} \label{appendix:baselines}
\begin{itemize}[leftmargin=1.4em,itemsep=3pt,topsep=3pt]
\item \textbf{IPOPT} \cite{10.5555/3114195.3114663} The full-horizon NLP \eqref{eq:main-problem} is handed to Ipopt (via CasADi with exact first/second derivatives) with no temporal decomposition.
\item \textbf{MultiShoot} \cite{BOCK19841603} The horizon is split into $M$ \emph{non-overlapping} intervals, each an independent initial value problem constrained NLP solved to convergence, with continuity (state-matching) constraints at the interval boundaries coordinated by an outer Newton iteration.
\item \textbf{ADMM} \cite{8186925} We deploy ADMM with fixed consensus penalty $\chi=10$: the horizon is decomposed into $M$ temporal blocks, each block solved to convergence, and the boundary/consensus variables are coupled through scaled dual updates.
\item \textbf{Schwarz} \cite{9840913} The overlapping-Schwarz (exact-OTD) scheme solves each nonlinear local OCP \eqref{eq:decomposed-main-problem} to convergence at a fixed overlap $b$ and composes the result with $\mathcal{C}$. It is the nonlinear-local-OCP baseline for our decomposition and thus isolates the cost of solving each nonlinear local OCP to convergence rather than taking a single inexact SQP step.
\item \textbf{FOTD (LU) and FOTD (GMRES)} \cite{naFOTD} At each outer iteration a \emph{single} SQP step is taken for each nonlinear local OCP by solving its local KKT linear system \eqref{eq:newton}. The direct variant uses an exact LU factorization. For comparison, we also solve the same local KKT system iteratively to $\epsilon_i^{\rm FOTD}=10^{-14}$ for swing and $10^{-11}$ for Burgers using GMRES with an inner-iteration cap of $100$. For all GMRES solves, the original local residual is checked every five Krylov iterations. We use the standard incomplete LU preconditioner~\cite{meijerink1977iterative} for all GMRES solves in the experiments. FOTD requires highly accurate local solves for a fair comparison, and without preconditioning GMRES stalls before reaching the prescribed tolerances on the indefinite saddle-point KKT systems in these OCPs.
\item \textbf{AOTD (GMRES) and AOTD (sGMRES)} The proposed adaptive overlapping temporal decomposition (Algorithm~\ref{alg:aotd-compact}) with an incomplete LU preconditioner and the GMRES inner solver (inner-iteration cap $100$, tolerance set adaptively) to solve the local KKT system. As in FOTD (GMRES), AOTD (GMRES) checks the original local residual every five Krylov iterations. The preconditioner is rebuilt at each accuracy pass when the overlap changes, because enlarging $b_i^\tau$ changes the local window and KKT matrix dimension. We additionally compare against a restarted Gaussian-sketched residual-minimization variant inspired by \cite{doi:10.1137/23M1565413}. Suppressing the subproblem and SQP iteration indices, let $\Gamma_{\mathrm{loc}}\bm s_{\rm loc}=\bm f$ denote the local KKT system and $P_{\rm loc}$ its left preconditioner, with $\bm s_{\rm loc}$ the exact local primal--dual direction and $\tilde{\bm s}_{\rm loc}$ its current approximation. At each restart, define $\widehat\Gamma_{\mathrm{loc}}=P_{\rm loc}^{-1}\Gamma_{\mathrm{loc}}$ and $\widehat{\bm r}=P_{\rm loc}^{-1}(\bm f-\Gamma_{\mathrm{loc}}\tilde{\bm s}_{\rm loc})$. Each restart cycle constructs a normalized power basis $\bm V$ of at most $m$ vectors, draws an independent sketch $\bm S\in\mathbb{R}^{\ell\times n}$ with entries $S_{ij}\sim\mathcal{N}(0,1/\ell)$ and $\ell=\min\{2m+5,n\}$, and computes the update from the sketched least-squares problem $\min_{\bm y}\|\bm S(\widehat\Gamma_{\mathrm{loc}}\bm V\bm y-\widehat{\bm r})\|$, followed by the update $\tilde{\bm s}_{\rm loc}\leftarrow\tilde{\bm s}_{\rm loc}+\bm V\bm y$. We use $(m,\ell)=(5,15)$ for swing and $(10,25)$ for Burgers, with caps of $100$ and $200$ matrix-vector products, respectively, across all restart cycles. For both GMRES and sGMRES, the local stopping condition \eqref{eq:local-residual-cond} is checked using the original, unsketched and unpreconditioned local KKT residual. After composing the local directions, the global accuracy condition \eqref{eq:globalResidualBound} is checked using the full unsketched global residual.
\end{itemize}
\subsection{Cost Calculations: Estimated FLOPs and Wall-Clock} \label{appendix:cost}

\subsubsection{Estimated FLOPs}
We report estimated FLOPs under one convention for all methods, where $\operatorname{nnz}(\cdot)$ denotes the number of structural nonzeros of a matrix. For a window of $N_{\rm loc}$ stages, the KKT system has half-bandwidth $b_w := 2n_x+n_u$ and dimension $n := N_{\rm loc}\,b_w+2n_x$. Thus, one banded LU factorization costs $\Phi := 2n\,b_w^2$ FLOPs. The nonlinear-programming baselines (IPOPT with $N_{\rm loc}=N$, and the $M$ windows, blocks, or intervals of Schwarz, ADMM, and MultiShoot) solve each NLP with an interior-point method and are charged $\Phi$ per iteration, summed over all subproblems and coordination iterations.

For FOTD and AOTD, costs accumulate per outer iteration over the local solves actually executed. Suppressing the subproblem and SQP iteration indices, write $H_{\rm loc}:=\hat H_i^\tau$, $G_{\rm loc}:=G_i^\tau$, and $\Gamma_{\rm loc}:=\Gamma_i^\tau$. An LU solve is charged $\Phi$. A $K$-iteration GMRES solve on the local KKT matrix $\Gamma_{\rm loc}$ is charged
\begin{equation*}
F_{\mathrm{GMRES}}
= 2K\operatorname{nnz}(\Gamma_{\rm loc}) + 2K^2 n
+ F_{\mathrm{build}} + K F_{\mathrm{apply}},
\end{equation*}
where the first two terms cover matrix--vector products and full reorthogonalization, and the last two the ILU--Schur preconditioner. For the latter, let $n_{z,\rm loc}$ be the number of local primal variables, $L$ the number of nonzeros in the ILU factors of the Schur complement $S_{\rm Schur}=G_{\rm loc}\operatorname{diag}(H_{\rm loc})^{-1}G_{\rm loc}^\top$, and $A_{\rm loc}:=G_{\rm loc}\operatorname{diag}(H_{\rm loc})^{-1}$. We charge
\begin{equation*}
F_{\mathrm{build}}=F_S+2L, \qquad
F_{\mathrm{apply}}=n_{z,\rm loc}+2L,
\end{equation*}
with the standard sparse-product count
\begin{equation*}
F_S=\operatorname{nnz}(G_{\rm loc})
+2\!\!\sum_{(i,j)\,:\,(A_{\rm loc})_{ij}\neq 0}\!\!\operatorname{nnz}((G_{\rm loc})_{:,j}).
\end{equation*}
The sketched variant (sGMRES), with restart dimension $m$ and $\ell=\min\{2m+5,\,n\}$ sketch rows as in Table~\ref{tab:experiment-config}, is charged $2\ell nm+2\ell n+2\ell m^2+2nm$ per restart for sketching and least squares, plus $2\operatorname{nnz}(\Gamma_{\rm loc})$ per matrix--vector product and preconditioner costs.

Preconditioner construction is charged once per build, and the factors are reused across inner iterations. On Burgers, AOTD further reuses factors across accuracy passes and outer iterations while the local dimension is unchanged, rebuilding once the iteration count exceeds twice that of a fresh build. On swing, AOTD and FOTD rebuild for every solve. For all methods we exclude matrix assembly, line-search, and communication costs, as well as IPOPT's internal regularization and corrections.

\subsubsection{Wall-Clock}
The wall-clock estimate is computed over outer iterations $\tau$ and synchronized accuracy-loop passes $j\in\mathcal J_\tau$ as
\begin{equation*}
T_{\parallel}
=\sum_{\tau}\Bigl(\,T^{\tau}_{\mathrm{ls}}
+\sum_{j\in\mathcal J_\tau}
\bigl(T^{\tau,j}_{\mathrm{glb}}
+\max_{i}\,T^{\tau,j}_{\mathrm{loc},i}\bigr)\Bigr),
\end{equation*}
where $T^{\tau,j}_{\mathrm{glb}}$ collects the global computations of pass $j$, $T^{\tau,j}_{\mathrm{loc},i}$ the local solve of subproblem $i$, and $T^{\tau}_{\mathrm{ls}}$ the line-search evaluations after acceptance.

\begin{table*}[tp]\centering \scriptsize
\setlength{\tabcolsep}{6pt}
\setlength{\aboverulesep}{1pt}\setlength{\belowrulesep}{1pt}
\caption{Full experimental configuration for the two benchmark OCPs. Problem-specific quantities appear in the first two blocks. The AOTD algorithm parameters and solver stack are shared except where the columns differ. A dash (---) marks a parameter that does not apply to that problem. Initial conditions are described in Section~\ref{sec:simulations}.}
\label{tab:experiment-config}
\begin{tabular}{llcc}
\toprule
Parameter & Symbol & NE39 swing OCP & Viscous Burgers \\
\midrule
\multicolumn{4}{l}{\emph{Problem and discretization}}\\
State / control dimension            & $n_x,\ n_u$                    & $20,\ 10$    & $24,\ 24$ \\
Time step                            & $\Delta t$                     & $0.1$        & $0.004$ \\
Horizon length                       & $N$                            & $1000$       & $10{,}000,\ 50{,}000$ \\
Horizon time                         & $T=N\Delta t$                  & $100$        & $40,\ 200$ \\
Viscosity, control penalty           & $\kappa,\ \psi$                & ---          & $0.02,\ 0.001$ \\
State / stage weight                 & $Q_k^{\rm cost}$                            & $10\,I$      & $\Delta t\cdot I$ \\
Control weight                       & $R_k^{\rm cost}$                            & $0.05\diag(c_1,\ldots,c_{10})$\textsuperscript{$\dagger$} & $\psi\cdot\Delta t\cdot I$ \\
Terminal weight                      & $Q_N^{\rm cost}$                          & $10\,I$      & $I$ \\
Target                               & $y_{\rm des}$                  & origin       & $0$ \\
\midrule
\multicolumn{4}{l}{\emph{Decomposition and baselines}}\\
FOTD inner tolerance                 & $\epsilon_i^{\rm FOTD}$        & $10^{-14}$   & $10^{-11}$ \\
Interface penalty                    & $\mu$                          & $10$         & $10$ \\
Initial merit weights (fixed for FOTD) & $\eta_1^0,\ \eta_2^0$        & $10,\ 0.5$   & $10,\ 0.5$ \\
Backtracking factor                  & $\delta$                       & $0.5$        & $0.5$ \\
Seeds (IC noise)                     & ---                            & $5$          & $5$ \\
\midrule
\multicolumn{4}{l}{\emph{AOTD parameters}}\\
Overlap / tol.\ adaptation rates     & $\varrho_b,\ \varrho_\epsilon$ & $32,\ 0.5$   & $200,\ 0.1$ \\
$\eta$-update / forcing rate         & $\nu$                          & $1.01$       & $2$ \\
Initial global tolerance             & $\epsilon^0$                   & $1$          & $1$ \\
Armijo constant                      & $\beta$                        & $10^{-3}$    & $10^{-4}$ \\
Initial overlap (all $i$)            & $b_i^0$                        & $10$         & $10$ \\
Initial inner tolerance (all $i$)    & $\epsilon_i^0$                 & $10^{-1}$    & $5\times10^{-3}$ \\
Sketch restart dimension             & $m$                            & $5$          & $10$ \\
Sketch rows                          & $\ell=2m+5$                    & $15$         & $25$ \\
\bottomrule
\end{tabular}
\par\smallskip
\begin{minipage}{\textwidth}
\scriptsize
\textsuperscript{$\dagger$}The generator-specific coefficients are
$\bm c=(c_1,\ldots,c_{10})^\top
=(0.94162537,\,1.15464676,\,0.55487802,\,0.71176656,\,1.27476119,\allowbreak
1.02529959,\,1.29460420,\,1.01800112,\,0.99465694,\,0.73467357)^\top$.
\end{minipage}
\end{table*}

\subsection{Nonlinear Frequency Control for Stability of Power Grids Ablation Study: Analysis on Parameter Choices of $\varrho_b$, $\varrho_\epsilon$, $\eta_1^0$, $\eta_2^0$, and $\nu$} \label{appendix:swing-ablation-study}

\begin{figure*}[tp]
\centering
\includegraphics[width=\linewidth]{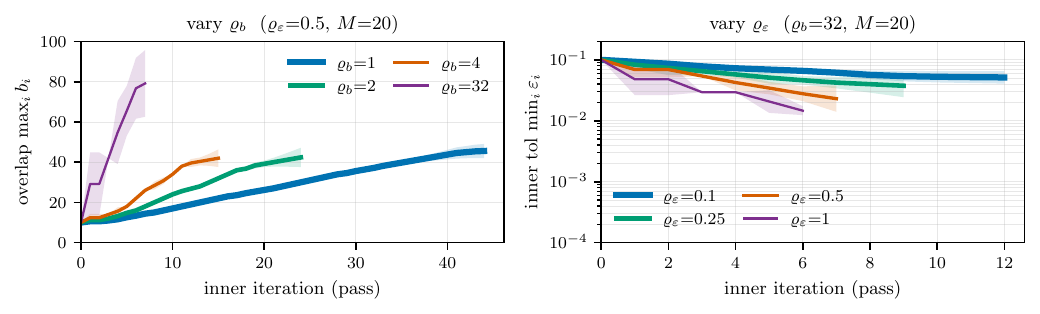}
\caption{Effect of the AOTD adaptation rates on the NE39 swing OCP ($N{=}1000$, $M{=}20$), over the same $5$ seeds used in Table~\ref{tab:highlight_swing}. \textbf{(left)} Ablation study for various adaptive overlap updating rates $\varrho_b\in\{1,2,4,32\}$ (fixed $\varrho_\epsilon{=}0.5$). \textbf{(right)} Ablation study for tolerance adaptation rates $\varrho_\epsilon\in\{0.1,0.25,0.5,1\}$ (fixed $\varrho_b{=}32$). Note, as shown in Table \ref{tab:rate_ablation}, the method converges for all rate choices and all seeds, although the numbers of overlap updates and outer iterations vary.}
\label{fig:ablation-rho-fg}
\end{figure*}

\begin{table*}[h]\scriptsize
\centering
\caption{AOTD adaptation-rate ablation $(\varrho_b,\varrho_\epsilon)$ on the NE39 swing OCP ($N=1000$, $M=20$) (mean$\pm$stdev over five seeds). The \# overlap updates column counts failed accuracy-loop passes of Algorithm~\ref{alg:aotd-compact} on which the local tolerances $\epsilon_i^\tau$ and overlaps $b_i^\tau$ are updated, summed over all outer iterations. The local tolerance floor is $10^{-9}$.}
\label{tab:rate_ablation}
\tablesize\setlength{\tabcolsep}{6pt}
\begin{tabular}{cc r r cc c c}
\toprule
$\varrho_b$ & $\varrho_\epsilon$ & \shortstack{$\|\nabla\mathcal{L}\|$\\($\times 10^{-7}$)} & \shortstack{FLOPs\\($\times 10^{7}$)} & \shortstack{\# overlap\\updates} & \shortstack{\# outer\\iterations} & $b_i^{\max}$ & $\epsilon_i^{\min}$ \\
\midrule
$1$ & $0.1$ & $2.2\pm2.5$ & $73.6\pm16.3$ & $33.4\pm4.1$ & $5.6\pm0.5$ & $43.4\pm4.1$ & $(1.9\pm1.0)\times10^{-3}$ \\
$1$ & $0.25$ & $2.2\pm2.5$ & $83.4\pm18.8$ & $33.4\pm4.1$ & $5.6\pm0.5$ & $43.4\pm4.1$ & $(7.5\pm7.6)\times10^{-6}$ \\
$1$ & $0.5$ & $0.8\pm0.6$ & $113.2\pm21.4$ & $35.6\pm3.5$ & $5.8\pm0.4$ & $45.6\pm3.5$ & $(1.2\pm0.4)\times10^{-9}$ \\
$1$ & $1$ & $0.8\pm0.6$ & $134.5\pm23.5$ & $35.6\pm3.5$ & $5.6\pm0.5$ & $45.6\pm3.5$ & $(1.0\pm0.0)\times10^{-9}$ \\
$2$ & $0.1$ & $2.7\pm2.8$ & $32.5\pm6.3$ & $16.4\pm2.3$ & $5.4\pm0.5$ & $42.8\pm4.7$ & $(1.4\pm0.4)\times10^{-2}$ \\
$2$ & $0.25$ & $1.2\pm1.3$ & $37.0\pm5.8$ & $17.4\pm1.6$ & $5.6\pm0.5$ & $44.8\pm3.2$ & $(5.6\pm2.8)\times10^{-4}$ \\
$2$ & $0.5$ & $2.6\pm2.3$ & $37.2\pm8.4$ & $16.2\pm2.4$ & $5.4\pm0.5$ & $42.4\pm4.8$ & $(1.2\pm1.3)\times10^{-5}$ \\
$2$ & $1$ & $2.5\pm2.4$ & $43.8\pm10.9$ & $16.4\pm2.7$ & $5.2\pm0.4$ & $42.8\pm5.3$ & $(3.3\pm4.5)\times10^{-9}$ \\
$4$ & $0.1$ & $5.2\pm2.8$ & $15.7\pm4.2$ & $7.6\pm1.4$ & $4.8\pm0.7$ & $40.4\pm5.4$ & $(4.0\pm0.6)\times10^{-2}$ \\
$4$ & $0.25$ & $2.7\pm1.9$ & $18.2\pm2.9$ & $8.2\pm1.0$ & $5.2\pm0.4$ & $42.8\pm3.9$ & $(8.6\pm2.2)\times10^{-3}$ \\
$4$ & $0.5$ & $3.1\pm1.9$ & $18.3\pm2.9$ & $8.0\pm1.1$ & $5.2\pm0.4$ & $42.0\pm4.4$ & $(9.1\pm4.3)\times10^{-4}$ \\
$4$ & $1$ & $3.1\pm2.0$ & $19.8\pm3.2$ & $8.2\pm1.5$ & $5.0\pm0.0$ & $42.6\pm5.5$ & $(1.0\pm0.8)\times10^{-5}$ \\
$32$ & $0.1$ & $0.6\pm0.8$ & $14.0\pm2.3$ & $5.8\pm1.9$ & $5.2\pm0.4$ & $107.8\pm4.4$ & $(5.3\pm1.3)\times10^{-2}$ \\
$32$ & $0.25$ & $2.7\pm2.9$ & $9.4\pm1.6$ & $3.4\pm1.4$ & $4.4\pm0.5$ & $93.8\pm15.9$ & $(4.1\pm1.4)\times10^{-2}$ \\
$32$ & $0.5$ & $1.2\pm1.0$ & $9.2\pm1.3$ & $2.8\pm0.7$ & $4.6\pm0.5$ & $79.2\pm16.6$ & $(2.5\pm0.9)\times10^{-2}$ \\
$32$ & $1$ & $1.2\pm1.0$ & $8.1\pm0.9$ & $2.0\pm0.0$ & $4.6\pm0.5$ & $57.6\pm5.5$ & $(1.5\pm0.3)\times10^{-2}$ \\
\bottomrule
\end{tabular}
\end{table*}

\begin{figure*}[!tp]
\centering
\includegraphics[width=\linewidth]{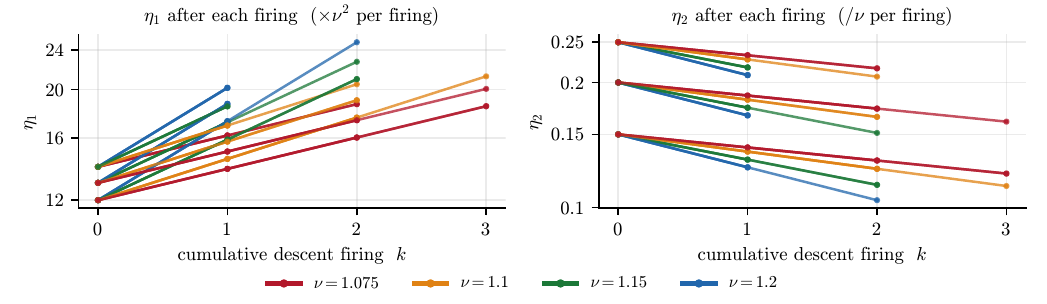}
\caption{Comparison of various $\eta_1^0 \in \{12,13,14\}$, $\eta_2^0 \in \{0.15,0.20,0.25\}$ and $\nu \in \{1.075,1.10,1.15,1.20\}$ values. The initializations are purposely chosen poorly to require the descent condition update.}
\label{fig:ablation_eta_init}
\end{figure*}

We now study the sensitivity of AOTD to its user-specified parameters on the NE39 swing OCP ($N=1000$, $M=20$). The adaptation-rate study is run over the same $5$ seeds as Table~\ref{tab:highlight_swing} and every entry of Table~\ref{tab:rate_ablation} is reported as mean$\pm$stdev across those seeds. The globalization-parameter study that follows is run on a single fixed seed. Recall from Section~\ref{sec:main-algorithm-design} that the method exposes only a small set of tuning knobs: the adaptation rates $\varrho_b$ and $\varrho_\epsilon$ governing how aggressively the overlaps and inner tolerances are updated in \eqref{eq:subproblem-overlap-update} and \eqref{eq:subproblem-error-update}, the initial globalization penalties $\eta_1^0$ and $\eta_2^0$ in the augmented Lagrangian \eqref{eq:augmentedLagrangian}, and the update factor $\nu$ controlling how these penalties change when the descent condition \eqref{eq:descent-direction-condition} fails. The purpose of this study is to confirm that, on this problem, convergence is insensitive to these choices across the ranges swept below, so that the parameters need not be retuned per run.

\subsubsection{Adaptation Rates $\varrho_b$ and $\varrho_\epsilon$}
Figure~\ref{fig:ablation-rho-fg} varies the two adaptation rates one at a time ($\varrho_b\in\{1,2,4,32\}$ at fixed $\varrho_\epsilon{=}0.5$, and $\varrho_\epsilon\in\{0.1,0.25,0.5,1\}$ at fixed $\varrho_b{=}32$), while Table~\ref{tab:rate_ablation} reports the full $\varrho_b\in\{1,2,4,32\}\times\varrho_\epsilon\in\{0.1,0.25,0.5,1\}$ grid with $\nu=1.01$. The method reaches the $10^{-6}$ tolerance for all $80$ runs ($16$ combinations over $5$ seeds). Increasing $\varrho_b$ generally reduces overlap-update counts, although outer-iteration counts are not strictly monotone. At $\varrho_\epsilon=1$, the update count falls from $35.6\pm3.5$ at $\varrho_b=1$ to $2.0\pm0.0$ at $\varrho_b=32$, while outer iterations fall from $5.6\pm0.5$ to $4.6\pm0.5$. The effect of stronger tolerance reduction depends on the overlap rate: it generally increases total estimated FLOPs for $\varrho_b\leq4$, but at $\varrho_b=32$ it reduces update passes and the attained overlap. The lowest mean estimated cost is $(8.1\pm0.9)\times10^7$ FLOPs at $(\varrho_b,\varrho_\epsilon)=(32,1)$, followed by $(9.2\pm1.3)\times10^7$ at $(32,0.5)$ and $(9.4\pm1.6)\times10^7$ at $(32,0.25)$. Thus, mild tolerance updates are not uniformly preferable.

\subsubsection{Globalization Parameters $\eta_1^0$, $\eta_2^0$, and $\nu$}
Figure~\ref{fig:ablation_eta_init} varies the initial penalties $\eta_1^0 \in \{12,13,14\}$ and $\eta_2^0 \in \{0.15,0.20,0.25\}$ together with the update factor $\nu \in \{1.075,1.10,1.15,1.20\}$. All $36$ tested settings converge, although some require increasing the maximum number of inner accuracy-loop passes from $50$ to $60$. To stress the globalization mechanism, the penalties are deliberately initialized away from values that immediately satisfy the descent condition \eqref{eq:descent-direction-condition}, so that the adaptive update \eqref{eq:adaptive-parameter-update} is forced to activate. Larger values of $\nu$ correct the penalties more aggressively and therefore reach a valid descent direction in fewer or equal firings, while smaller values of $\nu$ approach the same stabilized regime through a larger number of smaller adjustments. In all cases the penalties settle after finitely many updates, in agreement with the stabilization guarantee of Lemma~\ref{lem:stability}, after which the iterates proceed with fixed globalization parameters. The rates change the balance of work, while the globalization sweep changes firing counts without changing the three outer iterations. We state this as robustness \emph{within the swept ranges on this problem}. Establishing parameter robustness on unseen problem families would require a correspondingly broader study.

\subsection{Additional Results on Wall-Clock and Parallel Computation Time} \label{appendix:wall_cost}

Tables~\ref{tab:appendix-timing} and~\ref{tab:burgers_appendix} report timings for the NE39 swing and Burgers OCPs; the swing runs use the seeds in Table~\ref{tab:highlight_swing}. Parallel times are idealized estimates from serial runs under the model in Appendix~\ref{appendix:cost}, assuming one processor per subproblem and no communication or synchronization overhead. FOTD has the lowest mean estimates on swing, while an AOTD variant has the lowest mean in three of the six Burgers configurations ($N=10{,}000$ at $M=50,250$ and $N=50{,}000$ at $M=250$). Overlap changes require sparse KKT rebuilds that can offset FLOP savings; reducing this overhead with mutable sparse matrices is left to future work.

\begin{table*}[tp]\centering
\scriptsize 
\caption{Wall clock timing on NE39 swing OCP ($N{=}1000$, mean$\pm$stdev over 5 seeds); parallel times are idealized estimates (Appendix~\ref{appendix:wall_cost}). $M$ is the number of temporal subproblems. We use `---' to indicate the run stalled. Entries `out/in' are means over five seeds; outer iteration means are reported to one decimal, and inner iteration means are rounded to the nearest integer. `Out' counts applied SQP steps for FOTD/AOTD. `In' totals local direct KKT solves for LU, Krylov iterations for GMRES, and KKT matrix-vector products for sGMRES (including residual checks), summed over all subproblems and solve passes. Iteration counts for the other baselines are not reported.}
\label{tab:appendix-timing}
\begin{tabular}{l|rrrr|rrrr}
\toprule
& \multicolumn{4}{c|}{parallel wall $T_{\parallel}$ (s)} & \multicolumn{4}{c}{outer steps / inner steps (solver dependent)} \\
\cmidrule(lr){2-5}\cmidrule(lr){6-9}
Method & $M{=}4$ & $M{=}10$ & $M{=}20$ & $M{=}50$ & $M{=}4$ & $M{=}10$ & $M{=}20$ & $M{=}50$ \\
\midrule
IPOPT & \multicolumn{4}{c|}{$21.2{\pm}0.2$} & \multicolumn{4}{c}{--} \\
\midrule
MultiShoot & $11.8{\pm}0.1$ & --- & --- & --- & -- & -- & -- & -- \\
ADMM ($\chi{=}10$) & $11.7{\pm}0.0$ & $9.8{\pm}0.2$ & --- & --- & -- & -- & -- & -- \\
Schwarz & & & & & & & & \\
\quad $b{=}1$ & $12.0{\pm}0.2$ & --- & --- & --- & -- & -- & -- & -- \\
\quad $b{=}10$ & $12.6{\pm}0.0$ & $11.7{\pm}0.2$ & --- & --- & -- & -- & -- & -- \\
\quad $b{=}50$ & $15.5{\pm}0.0$ & $19.0{\pm}0.2$ & $26.7{\pm}0.2$ & $51.2{\pm}0.2$ & -- & -- & -- & -- \\
\quad $b{=}75$ & $17.5{\pm}0.0$ & $23.6{\pm}0.2$ & $35.3{\pm}0.2$ & $72.1{\pm}0.3$ & -- & -- & -- & -- \\
FOTD (LU) & & & & & & & & \\
\quad $b{=}10$ & {\boldmath $0.8{\pm}0.0$} & --- & --- & --- & 3.0/12 & --- & --- & --- \\
\quad $b{=}50$ & $0.8{\pm}0.0$ & {\boldmath $0.8{\pm}0.0$} & {\boldmath $0.8{\pm}0.0$} & --- & 3.0/12 & 3.0/30 & 3.0/60 & --- \\
\quad $b{=}75$ & $0.9{\pm}0.0$ & $0.9{\pm}0.0$ & $0.8{\pm}0.0$ & {\boldmath $0.8{\pm}0.0$} & 3.0/12 & 3.0/30 & 3.0/60 & 3.0/150 \\
FOTD (GMRES) & & & & & & & & \\
\quad $b{=}10$ & $0.9{\pm}0.0$ & --- & --- & --- & 3.0/221 & --- & --- & --- \\
\quad $b{=}50$ & $0.9{\pm}0.0$ & $0.9{\pm}0.0$ & $0.9{\pm}0.0$ & --- & 3.0/221 & 3.0/221 & 3.0/337 & --- \\
\quad $b{=}75$ & $0.9{\pm}0.0$ & $0.9{\pm}0.0$ & $0.9{\pm}0.0$ & $0.9{\pm}0.0$ & 3.0/221 & 3.0/224 & 3.0/445 & 3.0/891 \\
\midrule
AOTD (GMRES) & $1.4{\pm}0.2$ & $1.3{\pm}0.2$ & $1.3{\pm}0.1$ & $1.3{\pm}0.0$ & 4.8/76 & 4.8/80 & 4.6/109 & 4.0/251 \\
AOTD (sGMRES) & $1.3{\pm}0.1$ & $1.2{\pm}0.1$ & $1.4{\pm}0.2$ & $1.4{\pm}0.1$ & 4.4/127 & 4.4/158 & 5.0/322 & 4.2/1190 \\
\bottomrule
\end{tabular}
\end{table*}

\begin{table*}[tp]\centering    \scriptsize 
\caption{Wall clock timing on the Burgers OCP (mean$\pm$stdev over 5 seeds); parallel times are idealized estimates (Appendix~\ref{appendix:wall_cost}). $M$ is the number of temporal subproblems. We use `---' to indicate the run stalled. Entries `out/in' are means over five seeds; outer iteration means are reported to one decimal, and inner iteration means are rounded to the nearest integer. `Out' counts applied SQP steps for FOTD/AOTD and coordination sweeps for Schwarz. `In' totals local direct KKT solves for LU, Krylov iterations for GMRES, KKT matrix-vector products for sGMRES (including residual checks), and local IPOPT iterations for Schwarz, summed over all subproblems and solve passes.} 
\label{tab:burgers_appendix}
\begin{tabular}{l|rrr|rrr}
\toprule
& \multicolumn{3}{c|}{parallel wall $T_{\parallel}$ (s)} & \multicolumn{3}{c}{outer steps / inner steps (solver dependent)} \\
\cmidrule(lr){2-4}\cmidrule(lr){5-7}
Method & $M{=}10$ & $M{=}50$ & $M{=}250$ & $M{=}10$ & $M{=}50$ & $M{=}250$ \\
\midrule
\multicolumn{7}{l}{\emph{$N = 10{,}000$ \quad (core lengths $1000/200/40$)}} \\
\midrule
Schwarz &  &  &  &  &  &  \\
\quad $b{=}50$ & $40.2{\pm}0.4$ & $39.3{\pm}0.6$ & $63.8{\pm}0.7$ & 2.0/33 & 2.0/159 & 3.0/797 \\
FOTD (LU) &  &  &  &  &  &  \\
\quad $b{=}10$ & $47.4{\pm}8.8$ & $57.9{\pm}9.2$ & -- & 5.6/56 & 6.6/330 & -- \\
\quad $b{=}50$ & {\boldmath $31.0{\pm}9.1$} & $35.3{\pm}9.5$ & $38.6{\pm}12.8$ & 3.6/36 & 4.0/200 & 4.2/1050 \\
FOTD (GMRES) &  &  &  &  &  &  \\
\quad $b{=}10$ & -- & -- & -- & -- & -- & -- \\
\quad $b{=}50$ & $31.8{\pm}9.2$ & $35.6{\pm}9.7$ & -- & 3.6/979 & 4.0/5328 & -- \\
AOTD (GMRES) & $34.9{\pm}3.2$ & $32.7{\pm}3.1$ & {\boldmath $32.9{\pm}4.0$} & 4.2/553 & 3.8/3372 & 3.6/22592 \\
AOTD (sGMRES) & $33.1{\pm}5.0$ & {\boldmath $31.1{\pm}4.0$} & $33.0{\pm}4.0$ & 4.0/692 & 3.6/4306 & 3.6/30565 \\
\midrule[\heavyrulewidth]
\multicolumn{7}{l}{\emph{$N = 50{,}000$ \quad (core lengths $5000/1000/200$)}} \\
\midrule
FOTD (LU) &  &  &  &  &  &  \\
\quad $b{=}10$ & $185.3{\pm}32.4$ & $224.0{\pm}43.5$ & $257.5{\pm}42.2$ & 4.4/44 & 5.6/280 & 6.6/1650 \\
\quad $b{=}50$ & {\boldmath $152.1{\pm}44.6$} & {\boldmath $146.3{\pm}42.5$} & $158.7{\pm}44.1$ & 3.6/36 & 3.6/180 & 4.0/1000 \\
FOTD (GMRES) &  &  &  &  &  &  \\
\quad $b{=}10$ & -- & -- & -- & -- & -- & -- \\
\quad $b{=}50$ & $160.5{\pm}47.9$ & $147.1{\pm}43.9$ & $159.1{\pm}44.7$ & 3.6/963 & 3.6/4651 & 4.0/25596 \\
AOTD (GMRES) & $163.4{\pm}0.7$ & $167.3{\pm}15.7$ & $147.2{\pm}15.3$ & 4.0/487 & 4.2/2633 & 3.8/16571 \\
AOTD (sGMRES) & $162.3{\pm}0.3$ & $158.1{\pm}25.3$ & {\boldmath $139.9{\pm}18.7$} & 4.0/620 & 4.0/3099 & 3.6/19631 \\
\bottomrule
\end{tabular}
\end{table*}

\end{document}